\documentclass[11pt]{amsart}
\usepackage{graphicx} % Required for inserting images
\usepackage{epsfig,hyperref}
\usepackage{mathrsfs,amsmath,amssymb,amsthm,graphicx}
\usepackage[table]{xcolor}
\usepackage{enumitem}
\usepackage{float}
\usepackage{array}
\usepackage{subcaption}
\usepackage[ruled,vlined]{algorithm2e}     
\usepackage[margin=1in]{geometry}
\SetAlCapHSkip{0pt}    
\theoremstyle{definition}
\newtheorem{definition}{Definition}[section]
\newtheorem{theorem}{Theorem}[section]

\newtheorem{lemma}[theorem]{Lemma}

\newtheorem{proposition}{Proposition}[section]
\usepackage{lipsum}
\usepackage{mwe}

\title{Data-driven computation for periodic stochastic differential equations}
\author{Yao Li}
\address{Yao Li: Department of Mathematics and Statistics, University
  of Massachusetts Amherst, Amherst, Massachusetts, 01002, USA}
\email{yaoli@math.umass.edu}
\author{Jiatong Sun}
\address{Jiatong Sun: Department of Mathematical and Statistical Sciences, University
  of Alberta, Edmonton, Alberta, Canada, T6G 2N8}
\email{jiaton11@ualberta.ca}
\thanks{Yao Li is partially supported by NSF DMS-2510209.}


\begin{document}
\begin{abstract}
    Many stochastic differential equations in various applications like coupled neuronal oscillators are driven by time-periodic forces. In this paper, we extend several data-driven computational tools from autonomous Fokker-Planck equation to the time-periodic setting. This allows us to efficiently compute the time-periodic invariant probability measure using either grid-base method or artificial neural network solver, and estimate the speed of convergence towards the time-periodic invariant probability measure. We analyze the convergence of our algorithms and test their performances with several numerical examples. 
\end{abstract}

\maketitle

\section{Introduction}
The time evolution of a stochastic differential equation is governed by the Fokker-Planck equation. When the stochastic differential equation is autonomous (independent of time), under some assumptions, it usually admits an invariant probability measure whose probability density function solves the corresponding steady-state Fokker-Planck equation \cite{huang2015integral, huang2015integral, bogachev2022fokker}. Under additional assumptions, the stochastic differential equation generally has stochastic stability, meaning that the law of the stochastic differential equation, which is also the solution to the Fokker-Planck equation, converges to the invariant probability density function \cite{mattingly2002ergodicity, hairer2010convergence}. Many stochastic differential equations in applications, such as coupled neuronal oscillators \cite{ermentrout2010mathematical}, are driven by time-dependent external forces. In particular, periodic stimulation is frequently used in many numerical simulation of neuronal networks \cite{kromer2020long, wang2022modulation} and neuroscience experiments \cite{melland2023attractor, sceniak2001visual}. Partially motivated by this, one scenario that has recently gained attention is the periodic Fokker-Planck equation, in which the coefficients are $T$-periodic for some constant $T > 0$. It is known that under some suitable conditions, the solution of the periodic Fokker-Planck equation converges to a periodic probability measure \cite{ji2019existence, ji2021convergence}. This motivates us to study the stochastic stability of periodic stochastic differential equations. 

Similar to the case of the autonomous Fokker-Planck equation, rigorous results have their own strengths and limitations. For example, one can prove the existence of a periodic invariant probability measure and derive some concentration results. However, analytical study alone often does not provide a clear description of the periodic invariant probability measure. Similarly, rigorous studies may prove the convergence to the periodic invariant probability measure, but the speed of convergence is often not quantitative. Therefore, it is necessary to study computational methods for periodic stochastic differential equations and their associated stochastic Fokker-Planck equations. 

The goal of this paper is to extend a series of data-driven computational methods for autonomous Fokker-Planck equations to periodic Fokker-Planck equations. We introduce both a grid-based method and a neural network solver that compute the periodic invariant probability measure. We also study how to use coupling techniques to estimate the rate of convergence toward the periodic invariant probability measure. The central theme of this paper is to address periodicity in algorithms, which is key to improving their performance. 

%Briefly introduce grid-based algorithm
The first approach we consider is a grid-based data-driven finite difference solver for computing the periodic invariant probability measure. This method formulates the stationary periodic Fokker-Planck equation as a constrained optimization problem, combining Monte Carlo approximations of the periodic invariant measure with the discretization of the periodic Fokker–Planck equation in both spatial and temporal variables (using the finite difference method). The periodicity in time is enforced through a cyclic boundary condition that connects the last and first temporal layers, ensuring consistency across one full time period. The discretization in spatial variables follows the scheme introduced in \cite{zhai2022deep}. The key feature of this method is that it pushes most error terms to the boundary of the numerical domain in both time direction and spatial directions, which is proved by a combination of rigorous analysis and numerical computations. In particular, under the periodic boundary condition on time, the error term concentrates at the spatial boundary regardless of the time variable.

%Briefly introduce NN-solver
The neural network solver approximates the periodic invariant probability measure by minimizing a loss function. The idea is to convert the constraint optimization problem into an unconstrained optimization problem and to propose the associated loss function for artificial neural network training. Different from the grid-based finite difference method, we enforce the periodic boundary condition by adding an additional loss term to the standard artificial neural network. All collocation points used for training are sampled from the same Monte Carlo simulation as in the grid-based method. This mesh-free method performs well for high-dimensional problems, where traditional grid-based methods become computationally difficult.

%coupling method
To numerically compute the convergence rate toward the periodic invariant probability measure, we adopt the coupling mechanisms developed in \cite{li2020numerical}. Specifically, we numerically estimate the geometric ergodicity by analyzing the distribution of coupling times obtained from pairs of stochastic processes starting from different initial positions but driven by the same stochastic differential equations. These coupled processes are constructed using a mixture of classical coupling methods, including independent coupling, synchronous coupling, reflection coupling, and maximal coupling. Each coupling method has advantages in different cases, and their combination provides more efficient and reliable numerical coupling. The key innovation in our approach is the introduction of a Floquet-type decomposition for coupling time distribution. Unlike the autonomous case, where the coupling time distribution typically exhibits exponential tails with a constant pre-factor, we observe that in time-periodic stochastic differential equations, the pre-factor becomes a $T$-periodic function. This insight allows a more accurate characterization of the convergence behavior of stochastic processes over time. We develop an algorithm to estimate the coupling time distribution and extract the periodic pre-factor via line fitting. Numerical results show strong consistency with the theoretical analysis, which confirms the validity of the approximation method and the presence of time-periodic effects in the convergence rate.

Then we apply our algorithms to several numerical examples in different dimensions. These examples evaluate the performance of our algorithms and verify our conclusion that the coupling probability should have a periodic pre-factor function. In particular, in Section 6.3, we demonstrate that our artificial neural network solver works well for a periodically driven coupled neuronal oscillators. In many applications, coupled neuronal oscillators are naturally driven by external stimulations that are often periodic. Due to the challenge of high dimensionality, it is essential to understand those externally driven neuronal networks numerically. One potential important extension of our work is the computation of eigenfunctions of periodic Fokker-Planck equations, because the first eigenfunction pair (called the Q-function) describes the leading term of the transition probability kernel \cite{perez2023universal}. We expect the regression method in \cite{kreider2025artificial} to hold after some modification. This will be addressed in our subsequent work. 

The organization of this paper is as follows. Section 2 provides the necessary probability and dynamical system backgrounds. In Section 3, we describe the numerical algorithms for grid-based data-driven methods and rigorously analyze the convergence of the method. We also show that the error will concentrate on the boundary of both the spatial and temporal domains, and we emphasize the importance of the periodic condition on the time variable. In Section 4, the algorithm for data-driven neural network solver and sampling of collocation points and reference points are described. In Section 5, we review the theoretical results and coupling mechanisms in \cite{li2020numerical}. By constructing a Floquet-type decomposition of the coupling time distribution, we study the periodic behavior of the pre-factor in the coupling time distribution, and describe the algorithm for estimating the coupling time, convergence rate, and the periodic pre-factor. Section 6 contains numerical examples of stochastic differential equations in different dimensions. We conclude this paper in Section 7 with future discussions and potential works.

\section{Probability preliminaries}
\subsection{Fokker-Planck Equation and Periodic Probability Measure.}

We consider the stochastic differential equation (SDE):
\begin{equation} \label{Eqn:SDE}
     \mathrm{d} \boldsymbol{X}_t=f\left(\boldsymbol{X}_t, t\right) \mathrm{d} t+\sigma\left(\boldsymbol{X}_t, t\right) \mathrm{d} \boldsymbol{W}_t,
\end{equation}
where the vector field $f: \mathbb{R}^n \times \mathbb{R} \rightarrow \mathbb{R}^n$ is $T$-periodic in its second variable, the coefficient matrix $\sigma: \mathbb{R}^n \times \mathbb{R} \rightarrow \mathbb{R}^{n \times n}$ is $T$-periodic in its second variable, and $\boldsymbol{W}_t$ is an $n$-dimensional standard Wiener process. 

Under sufficient regularity conditions on $f$ and $\sigma$, the stochastic differential equation (\ref{Eqn:SDE}) admits a unique solution $\boldsymbol{X}_t$, which defines a $T$-periodic Markov process. Let $P_{s,t}(\boldsymbol{x}, \cdot)$ denote the transition probability measure from time $s$ to $t$, given the initial condition $\boldsymbol{X}_s = \boldsymbol{x}$. We assume that $\boldsymbol{X}_t$ admits a unique periodic invariant probability measure $\mu$, with associated time-marginals $\{\mu_t\}_{t \in \mathbb{R}}$ satisfying
\begin{align}
    &\mu_t P_{s,t}(A) = \int_{\mathbb{R}^n} P_{s,t}(\boldsymbol{x}, A) \, \mathrm{d} \mu_t(\boldsymbol{x}), \quad \forall A \in \mathcal{B}(\mathbb{R}^n),\\
    &\mu_{t+T} = \mu_t, \quad P_{s+T, t+T}(\boldsymbol{x}, \cdot) = P_{s,t}(\boldsymbol{x}, \cdot).
\end{align}

The time evolution of the probability density of the solution process $\boldsymbol{X}_t$ is characterized by the Fokker-Planck equation, which is also known as the Kolmogorov forward equation
\begin{equation} \label{Eqn:FPE}
    0=\mathcal{L} u := - u_t -\sum_{i=1}^n\left(f_i u\right)_{x_i}+\frac{1}{2} \sum_{i, j=1}^n\left(\Sigma_{i, j} u\right)_{x_i x_j},
\end{equation}
where $u(\boldsymbol{x}, t)$ denotes the periodic probability density function of the stochastic process $\boldsymbol{X}_t$ at time $t, \Sigma=\sigma \sigma^T$ is the diffusion coefficient and is $T$-periodic in its second variable, and subscripts $t$ and $x_i$ denote partial derivatives. 

\begin{definition}[\textbf{Periodic Probability Measure}]
A Borel measure $\mu$ on $\mathbb{R}^n \times \mathbb{R}$ is called a periodic probability solution to (\ref{Eqn:FPE}) if there is a family of Borel probability measures $\left\{\mu_t\right\}_{t \in \mathbb{R}}$ on $\mathbb{R}^n$ satisfying
$$
\begin{gathered}
\mu_t=\mu_{t+T}, \quad \forall t \in \mathbb{R}, \\
\Sigma_{i, j}, f_i \in L_{l o c}^1\left(\mathbb{R}^n \times \mathbb{R}, \mathrm{d} \mu_t \mathrm{d} t\right), \quad \forall i, j \in\{1, \ldots, n\},
\end{gathered}
$$
and
$$
\int_{\mathbb{R}} \int_{\mathbb{R}^n} \mathcal{L} \phi (\boldsymbol{x}, t) \, \mathrm{d} \mu_t (\boldsymbol{x})\mathrm{d} t=0, \quad \forall \phi \in C_0^{2,1}(\mathbb{R}^n \times \mathbb{R}),
$$
such that $\mathrm{d} \mu=\mathrm{d} \mu_t \mathrm{d} t$. 
\end{definition}

Throughout this paper, we assume that the Fokker-Planck equation \eqref{Eqn:FPE} admits a unique periodic probability measure, which is the periodic invariant probability measure of the stochastic differential equation \eqref{Eqn:SDE}. In addition, we assume that for any initial point \( \boldsymbol{x} \in \mathbb{R}^n \), the transition probability measure \( P_{s,t}(\boldsymbol{x}, \cdot) \) converges to the periodic measure $\mu_t$ as \( t \to \infty \). We refer to \cite{ji2019existence} and \cite{ji2021convergence} for the detailed conditions ensuring the existence and uniqueness of the invariant periodic probability measure (called the periodic probability solution in those papers), as well as the convergence to the invariant periodic probability measure.

In this paper, we focus on cases where the periodic probability measure of (\ref{Eqn:FPE}) admits a smooth density function $\mathbb{R}^n \times \mathbb{R}\ni (\boldsymbol{x},t) \mapsto u(\boldsymbol{x}, t) \in \mathbb{R}$ satisfying the following \textbf{periodic Fokker-Planck equation}
\begin{equation} \label{Eqn:density function equation}
    \left\{\begin{array}{l}
        %\mathcal{L} u=u_t \\
        \mathcal{L} u=0 \\
        u(\boldsymbol{x}, t+T) = u(\boldsymbol{x}, t), \quad \forall \boldsymbol{x} \in \mathbb{R}^n , \, \forall t \in \mathbb{R},\\
        \int_{\mathbb{R}^n} u(\boldsymbol{x}, t) \mathrm{d} \boldsymbol{x}=1, \quad \forall t \in \mathbb{R}.
    \end{array}\right.
\end{equation}

%We assume the density function $u(\boldsymbol{x}, t)$ of the periodic probability measure exists, and we denote $\mathrm{d} \mu:=u(\boldsymbol{x}, t) \mathrm{d} \boldsymbol{x} \mathrm{d} t$, and $\mathrm{d} \mu_t=u(\boldsymbol{x}, t) \mathrm{d} \boldsymbol{x}$. For any Borel measurable set $A$, we have $\mu(A)=\iint_A u(\boldsymbol{x}, t) \mathrm{d} \boldsymbol{x} \mathrm{d} t$.

%Throughout the present paper, we assume the existence and uniqueness of the solution to the periodic Fokker-Planck equation. 

%% geometric ergodicity
\subsection{Geometric Ergodicity and Coupling} 

\begin{definition}[\textbf{Periodic Ergodicity}]
Let $E \subset \mathbb{R}^n$ be the state space. The process $\boldsymbol{X}_t$ is said to be \textit{geometrically ergodic} with rate $r>0$ if for $\mu_t$-a.e. $\boldsymbol{x} \in E$,
\[
\limsup _{t \rightarrow \infty} \frac{1}{t} \log \left(\left\|P_{s,t}(\boldsymbol{x}, \cdot)-\mu_t\right\|_{T V}\right)=-r,
\]
where $\|\mu-\nu\|_{T V}:=2 \sup _{A \in \mathcal{B}}|\mu(A)-\nu(A)|$ is the total variation distance between probability measures on $(E, \mathcal{B})$, $\mathcal{B}$ is $\sigma-$algebra.
\end{definition}

\begin{definition}[\textbf{Periodic Contraction}]
The process $\boldsymbol{X}_t$ is said to be \textit{geometrically contracting} with rate $r>0$ if for $\mu_t \times \mu_t$-almost every initial pair $(\boldsymbol{x}, \boldsymbol{y}) \in E \times E$, it holds that
\[
\limsup _{t \rightarrow \infty} \frac{1}{t} \log \left(\left\|P_{s,t}(\boldsymbol{x}, \cdot)-P_{s,t}(\boldsymbol{y}, \cdot)\right\|_{T V}\right)=-r .
\]
\end{definition}

It is easy to see that geometric ergodicity implies geometric contraction. In the case of geometric ergodicity, one usually has the estimate
\[
\left\|P_{s, t}(\mathbf{x}, \cdot)-\mu_t\right\|_{\mathrm{TV}} \leq C(\mathbf{x}, s) e^{-r(t-s)}, \, 0<s<t.
\]

%% coupling
Next, we recall the definitions of coupling of probability measures and coupling of stochastic processes. 
\begin{definition}[\textbf{Coupling of Probability Measure}]
\label{def: coupling of measure}
    Let $\mu$ and $\nu$ be two probability measures on $(E, \mathcal{B})$. A \textit{coupling} of $\mu$ and $\nu$ is a probability measure $\gamma$ on $E \times E$ whose first and second marginals are $\mu$ and $\nu$ respectively, that is,
    \[
    \gamma(A \times E)=\mu(A) \quad \text { and } \quad \gamma(E \times A)=\nu(A), \quad \forall A \in \mathcal{B} .
    \]
\end{definition}

Let $X$ and $Y$ be random variables with distributions $\mu$ and $\nu$ respectively, then
\begin{equation} \label{eqn: inequality of measure}
    \|\mu-\nu\|_{T V} \leq 2 \mathbb{P}[X \neq Y].
\end{equation}
This is a well-known inequality that shows the total variation distance between probability measures $\mu$ and $\nu$ is bounded by the difference of random variables realizing them. Coupling is a useful tool to bound the distance between probability measures.

\begin{definition}[\textbf{Coupling of Stochastic Process}]
    Let $\boldsymbol{X}=\left\{X_t ; t \in \mathcal{T}\right\}$ and $\boldsymbol{Y}=\left\{Y_t ; t \in \mathcal{T}\right\}$ be two stochastic processes on $(E, \mathcal{B})$. A \textit{coupling} of $\boldsymbol{X}$ and $\boldsymbol{Y}$ is a stochastic process $(\boldsymbol{X}, \boldsymbol{Y})=\left\{\left(\mathcal{X}_t, \mathcal{Y}_t\right) ; t \in \mathcal{T}\right\}$ on $E \times E$ such that
    \begin{itemize}
        \item[(i)] The first and second marginal processes $\left\{\mathcal{X}_t\right\}$ and $\left\{\mathcal{Y}_t\right\}$ are respective copies of $\boldsymbol{X}$ and $\boldsymbol{Y}$;
        \item[(ii)] If $s \in \mathcal{T}$ is such that $\mathcal{X}_s=\mathcal{Y}_s$, then $\mathcal{X}_t=\mathcal{Y}_t$ for all $t \geq s$.
    \end{itemize}
\end{definition}

Define the stopping time $\tau_c:=\inf _{t \geq 0}\left\{\mathcal{X}_t=\mathcal{Y}_t\right\}$ as the first meeting time of $\mathcal{X}_t$ and $\mathcal{Y}_t$. $\tau_c$ is called the \textit{coupling time}. A coupling $(\boldsymbol{X}, \boldsymbol{Y})$ is said to be \textit{successful} if the coupling time $\tau_c$ is almost surely finite, that is, $\mathbb{P}\left[\tau_c<\infty\right]=1$.

\section{Data-driven finite difference Fokker-Planck solver}
\subsection{Algorithm description}
Let $u(\boldsymbol{x}, t)$ be the periodic probability density function of $\mu$ at time $t$. Then $u(\boldsymbol{x}, t)$ should satisfy the stationary Fokker-Planck equation (\ref{Eqn:FPE}). In addition, $u(\boldsymbol{x}, t)$ has a probabilistic representation because for any set $A \subset \mathbb{R}^n\times \mathbb{R}$ we have
\begin{equation}
\label{Eq:MC}
    \frac{1}{T}\int_Au(\boldsymbol{x}, t) \mathrm{d}\boldsymbol{x}\mathrm{d}t = \lim_{t\rightarrow \infty}\frac{1}{t} \int_0^t \mathbf{1}_{\{X_t \in A\}} \mathrm{d}t
\end{equation}
because we assume the existence, uniqueness, and convergence of the periodic invariant probability measure. 
%From the periodic condition in (\ref{Eqn:density function equation}) and convergence, we have the \textbf{time-averaged stationary density} over one period

%\begin{equation}
%u(\boldsymbol{x})=\lim _{s \rightarrow \infty} \frac{1}{s} \int_0^s u(\boldsymbol{x}, t) \mathrm{d} t = \frac{1}{T} \int_0^T u(\boldsymbol{x}, t) \, \mathrm{d}t.
%\end{equation} %% not sure %%

%1. finite difference: discretized Fokker-Planck operator

For simplicity, we consider the $\mathbb{R}^2 \times \mathbb{R}$ case, but the algorithm works in $\mathbb{R}^n \times \mathbb{R}$ for any dimension $n$. Now, assume that we would like to solve $u = u(\boldsymbol{x},t)$ numerically on a 2D $\times$ 1D domain $\Omega=\left[a_0, b_0\right] \times\left[a_1, b_1\right]\times\left[t_1, t_2\right]$. Specifically, here we assume that the time domain $\left[t_1, t_2\right]$ spans exactly one period of the system, that is, $t_2 = t_1 + T$, where $T$ is the known temporal period. To do this, an $N \times M \times L$ grid is constructed on $\Omega$ with grid size $h=\left(b_0-a_0\right) / N=\left(b_1-a_1\right) / M$, $\delta=\left(t_2-t_1\right) / L$. Since $u$ is the density function, we approximate it at the center of each of the $N \times M \times L$ boxes $\left\{O_{i, j,k}\right\}_{i=1, j=1,k=1}^{i=N, j=M,k=L}$ with $O_{i, j, k}=\left[a_0+(i-1) h, a_0+i h\right] \times\left[a_1+(j-1) h, a_1+j h\right] \times \left[t_1+(k-1) \delta, t_1+k \delta\right]$. Let $\mathbf{u}=\left\{u_{i, j, k}\right\}_{i=1, j=1, k=1}^{i=N, j=M, k=L}$ be this numerical solution on $D$ that we are interested in. $\mathbf{u}$ can be considered as a vector in $\mathbb{R}^{NML}$. An entry of $\mathbf{u}$, denoted by $u_{i,j,k}$, approximates the probability density function $u$ at the center of the $(i,j,k)$-box with coordinate $\left(i h+a_0-h / 2, j h+a_1-h / 2, k \delta + t_1 -\delta/2\right)$. 

Now, we consider $u$ as the solution to the Fokker-Planck equation (\ref{Eqn:FPE}) on $\Omega$, without knowing the boundary condition. Therefore, we can only discretize the operator $\mathcal{L}$ in the interior of $\Omega$ with respect to all $(N-2)(M-2)(L-1)$ interior boxes. (Note that the grid points at time $t_2$ are not used due to the periodic boundary condition. When a grid point at time $t_2$ is needed in the time discretization, we use the grid point at time $t_1$ instead.) Using the finite difference method, the discretization of the Fokker-Planck equation with respect to each center point gives a linear relation among $\left\{ u_{i,j,k} \right\}$. This produces a linear constraint for $\mathbf{u}$, denoted by
$$
\mathbf{A u}=\mathbf{0}
$$
where $\mathbf{A}$ is an $(N - 2)(M - 2)(L-1) \times NM(L-1)$ matrix. $\mathbf{A}$ is said to be the discretized Fokker-Planck operator.

% periodic condition
%Next, we consider the periodic condition for the solution $u$. Under the assumption that $t_2 = t_1 + T$, the numerical solution $\mathbf{u}$ should satisfy $u_{i,j,1} = u_{i,j,L}$, $\forall i = 1, \dots, N;\ j = 1, \dots, M$, where the indices $k=1$ and $k=L$ correspond to the start and end of the period respectively. Now this periodic condition gives us another linear constraint for $\mathbf{u}$:
%$$
%\mathbf{P u}=\mathbf{0}
%$$
%where $\mathbf{P} \in \mathbb{R}^{NM \times NML}$ is a sparse matrix that enforces the temporal periodicity. Each row of $\mathbf{P}$ corresponds to a spatial grid index $(i,j)$, and contains two nonzero entries: $\mathbf{P}_{r,\,\text{ind}(i,j,1)} = 1$, $\mathbf{P}_{r,\,\text{ind}(i,j,L)} = -1$, where the row index $r = (j - 1)N + (i - 1)$ and the flattend index $\text{ind}(i,j,k) = (k - 1)(NM) + (j - 1) N + (i - 1)$.

%2. Monte Carlo method (reference solution to boundary-free PDE $Lu=u_t$)
Then, following the idea that the periodic invariant probability measure admits a probabilistic representation given in equation \eqref{Eq:MC}, we can get the Monte Carlo approximation of $\mathbf{u}$ in the following way. Let $\left\{\mathbf{X}_n\right\}_{n=1}^{\mathbf{N}}$ be a long numerical trajectory of the time-$\delta$ sample chain of $X_t$, i.e., $\mathbf{X}_n=X_{n \delta}$, where $\delta>0$ is the time step size of the Monte Carlo simulation. Let $\mathbf{v}=\left\{v_{i, j, k}\right\}_{i=1, j=1, k=1}^{i=N, j=M, k=L-1}$ such that
$$
v_{i, j, k}=\frac{1}{\mathbf{N} h^2 \delta} \sum_{n=1}^{\mathbf{N}} \mathbf{1}_{O_{i, j, k}}\left(X_n\right) .
$$
It follows from the ergodicity of (\ref{Eqn:SDE}) that $\mathbf{v}$ is an approximate periodic stationary solution of (\ref{Eqn:FPE}). Again, we denote the $N \times M \times (L-1)$ vector reshaped from approximate solution $\mathbf{v}$ by $\mathbf{v}$ as well.

%3. optimization problem
Next, we look for the solution of the following optimization problem
\begin{equation} 
\label{Eqn:optimization}
\begin{array}{r}
\min \|\mathbf{u}-\mathbf{v}\|_2 \\
\text { subject to } \mathbf{A} \mathbf{u}=\mathbf{0}
\end{array}
\end{equation}
This is called the least norm problem. Vector
\begin{equation}
\mathbf{u}=\mathbf{A}^T\left(\mathbf{A} \mathbf{A}^T\right)^{-1}(-\mathbf{A} \mathbf{v})+\mathbf{v}
\end{equation}
solves the optimization problem.

%In the periodic case, we additionally require the solution $\mathbf{u}$ to satisfy the periodic condition $\mathbf{Pu} = \mathbf{0}$, as described above. This leads to the modified optimization problem:
%\begin{equation} \label{Eqn:optimization}
%\begin{array}{r}
%\min \|\mathbf{u}-\mathbf{v}\|_2 \\
%\text { subject to } \mathbf{A} \mathbf{u}=\mathbf{0}\\
%\mathbf{P} \mathbf{u}=\mathbf{0}
%\end{array}
%\end{equation}
%which incorporates both the discretized Fokker-Planck equation and the periodic boundary condition in time.

\subsection{Error analysis through projections}
Inspired by Section 2.2 of \cite{dobson2019efficient}, we aim to show that the solution $\mathbf{u}$ to the optimization problem (\ref{Eqn:optimization}) is a good approximation of the global analytical solution $u$ on $\mathbb{R}^2 \times \mathbb{R}$ . Let $\mathbf{u}^{\text {ext }}=\left\{u_{i, j, k}^{\text {ext }}\right\}=\left \{u\left(i h+a_0-h / 2, j h+a_1-h / 2, k \delta + t_1 -\delta/2\right)\right \}$ be the values of the exact solution $u$ at the centers of the boxes.

We consider the following assumption \textbf{(H)}:
\begin{itemize}
    \item[\textbf{(a)}] For $i=1, \ldots, N, j=1, \ldots, M, k=1, \ldots, L-1$, $ v_{i, j, k}-u_{i, j, k}^{\text {ext }}$ are i.i.d. random variables with mean 0 and variance $\zeta^2$. 
    \item[\textbf{(b)}] The finite difference scheme for Equation (\ref{Eqn:FPE}) is convergent for the boundary value problem on $\left[a_0, b_0\right] \times\left[a_1, b_1\right]\times\left[t_1, t_2\right]$ with $L^{\infty}$ error $O\left(h^p+\delta^q\right)$ for some constants $p,q >0$, as $h,\delta \rightarrow 0$.
\end{itemize}
We use $h \delta^{\frac1 2} \mathbb{E}\left[\left\|\mathbf{u}-\mathbf{u}^{\text {ext }}\right\|_2\right]$, the $L^2$ numerical integration of the error term, to measure the performance of the algorithm. Before solving the optimization problem (\ref{Eqn:optimization}), we have
\[h \delta^{\frac1 2} \mathbb{E}\left[\left\|\mathbf{v}-\mathbf{u}^{\text {ext }}\right\|_2\right] = O(\zeta \,\,h N \,(\delta (L-1))^{\frac1 2})=O(\zeta).\]
\begin{theorem} \label{thm: error analysis}
Assume \textbf{(H)} holds. We have the following bound for the $L^2$ error of the optimization problem
\[ h \delta^{\frac1 2}\mathbb{E}\left[\left\|\mathbf{u}-\mathbf{u}^{\text {ext }}\right\|_2\right] \leq O\left(h^{\frac1 2} \zeta\right) + O\left(h^p + \delta^q \right). \]
\end{theorem}

\begin{proof}
Let $w$ be the auxiliary vector that satisfies the linear constraint in Equation (\ref{Eqn:optimization}). Consider the Fokker-Planck equation on the extended domain $\widetilde{D}=\left[a_0-h, b_0+h\right] \times\left[a_1-h, b_1+h\right]\times\left[t_1, t_2\right] \subset \mathbb{R}^2 \times \mathbb{R}$ with boundary condition
\begin{equation} \label{eqn: boundary}
\begin{cases}\mathcal{L} w=0 & (x, y, t) \in \widetilde{D} \\ w(x, y, t)=u(x, y, t) & (x, y, t) \in \partial \widetilde{D}\end{cases}\,,
\end{equation}
where $h$ and $\delta$ are the mesh sizes, and $u$ is periodic in a way that $u(t_1,x,y) = u(t_2, x, y)$. The problem is well-posed and has a unique solution $u(x, y, t),\;(x, y, t) \in \widetilde{D}$.

Now, discretize (\ref{eqn: boundary}) by the same $(x,y,t)$ – grid steps $(h,h,\delta)$.  Label the unknown discrete values inside \(\widetilde D\) by 
\[
\mathbf{u}^{\rm lin} 
\;=\;\bigl\{\,w(i\,h + a_0 - \tfrac{h}{2},\;j\,h + a_1 - \tfrac{h}{2},\;k\,\delta + t_1 - \tfrac{\delta}{2})\bigr\},
\]
for $1 \le i\le N,\;1\le j\le M,\;1\le k\le L-1$, and put all prescribed boundary values (where \((x,y,t)\in\partial\widetilde D\)) into vector \(\mathbf{u}_0\), i.e. $\mathbf{u}_0$ are the values of $u(x,y)$ at grid points on the boundary $\partial \widetilde{D}$ of $\widetilde{D}$. We get the following form of the discretization of (\ref{eqn: boundary}) by the finite difference method over \(\widetilde D\)
\[
\left[\begin{array}{ll}
\mathbf{A} & \mathbf{0} \\
\mathbf{B} & \mathbf{C}
\end{array}\right]\left[\begin{array}{l}
\mathbf{u}^{\text {lin}} \\
\mathbf{u}_{\mathbf{0}}
\end{array}\right]=\left[\begin{array}{l}
\mathbf{0} \\
\mathbf{0}
\end{array}\right].
\]

Here:
\begin{itemize}
  \item The upper block $\mathbf{A}$ is exactly the same discretization of $\mathcal{L}$ on the interior $\widetilde D\setminus\partial\widetilde D$. In particular, $\mathbf{A}$ has \((N\!-\!2)(M\!-\!2)(L-1)\) rows, and $\mathbf{A}\,\mathbf{u}^{\rm lin}=0$ is exactly “$\mathcal{L}\,w=0$” at each grid point interior to $\widetilde D$.
  \item The lower block \(\bigl[\mathbf{B}\;\;\mathbf{C}\bigr]\) consists of the extended $(2(M-1)+2(N-1))(L-1)$ equations that enforce $w = u$ on all grid‐points of \(\partial\widetilde D\).
\end{itemize}

The extended linear system is now well-posed. By assumption \textbf{(H)}-(b), we have
\[
\left\|\mathbf{u}^{\mathrm{lin}}-\mathbf{u}^{\mathrm{ext}}\right\|_{\infty}=O\left(h^p + \delta ^q\right) .
\]
Thus by the triangle inequality, to estimate $\left\|\mathbf{u}-\mathbf{u}^{\operatorname{ext}}\right\|_2$, it is sufficient to estimate 
\[
\left\|\mathbf{u}-\mathbf{u}^{\operatorname{lin}}\right\|_2.
\]
Let $P$ be the projection matrix to $\operatorname{Ker}(\mathbf{A})$. Then equation (\ref{Eqn:optimization}) implies $\mathbf{u}=P \mathbf{v}$. Since $\mathbf{u}^{\text {lin }} \in \operatorname{Ker}(A)$, we have
\[
\mathbf{u}-\mathbf{u}^{\mathrm{lin}}=P \mathbf{v}-\mathbf{u}^{\mathrm{lin}}=P\left(\mathbf{v}-\mathbf{u}^{\mathrm{lin}}\right)=P\left(\mathbf{v}-\mathbf{u}^{\mathrm{ext}}\right)+P\left(\mathbf{u}^{\mathrm{ext}}-\mathbf{u}^{\mathrm{lin}}\right) .
\]
Take the $L^2$ norm on both sides and apply the triangle inequality to get
\begin{equation} \label{eqn:proof_1}
    h \delta^{\frac1 2}\left\|\mathbf{u}-\mathbf{u}^{\text {lin }}\right\|_2 \leq h \delta^{\frac1 2}\left\|P\left(\mathbf{v}-\mathbf{u}^{\text {ext }}\right)\right\|_2+h \delta^{\frac1 2}\left\|P\left(\mathbf{u}^{\text {ext }}-\mathbf{u}^{\text {lin }}\right)\right\|_2.
\end{equation}
The second term is easy to bound because for $\mathbf{u}^{\text {ext}}, \mathbf{u}^{\text {lin }}\in \mathbb{R}^{NM(L-1)}$,
\begin{equation} \label{eqn:proof_2}
    \begin{aligned}
    h \delta^{\frac1 2}\left\|P\left(\mathbf{u}^{\text {ext }}-\mathbf{u}^{\text {lin }}\right)\right\|_2 & \leq h \delta^{\frac1 2}\left\|\mathbf{u}^{\text {ext }}-\mathbf{u}^{\text {lin }}\right\|_2 \\
& \leq h \delta^{\frac1 2} \sqrt{NM(L-1)}\left\|\mathbf{u}^{\text {ext }}-\mathbf{u}^{\text {lin }}\right\|_{\infty} \\
& = h N (\delta (L-1))^{\frac1 2}\; O\left(h^p + \delta ^q\right)=O\left(h^p + \delta ^q\right)
    \end{aligned}
\end{equation}

By assumption \textbf{(H)}, $\mathbf{w}=\mathbf{v}-\mathbf{u}^{\text {ext }}$ is a random vector with i.i.d. entries with mean 0 and variance $\zeta^2$, and $P$ projects $\mathbf{w}$ from $\mathbb{R}^{N \times M \times (L-1)}$ onto $\operatorname{Ker}(\mathbf{A})$. Let 
\[
r=\operatorname{rank}\left(\mathbf{A} \mathbf{A}^T\right)=(N-2)(M-2) (L-1),
\]
then 
\[
\operatorname{dim(ker}(A))=N M (L-1)-r =(2 N+2 M-4) (L-1).
\]
Let $S \in S O\left(NM(L-1)\right)$ be an orthogonal matrix such that the first $(2 N+2 M-4) (L-1)$ columns of $S^T$ form an orthonormal basis of $\operatorname{Ker}(\mathbf{A})$. Let $\mathbf{s}_1, \cdots, \mathbf{s}_{NM(L-1)}$ be column vectors of $S^T$. Then $S$ is a change-of-coordinate matrix such that $\operatorname{Ker}(\mathbf{A})$ is spanned by $\mathbf{s}_1, \cdots, \mathbf{s}_{(2 N+2 M-4) (L-1)}$.

Let
\[
S \mathbf{w}=\left[\hat{w}_1, \cdots, \hat{w}_{NM(L-1)}\right]^T,
\]
We have
\[
P \mathbf{w}=\sum_{i=1}^{(2 N+2 M-4) (L-1)} \hat{w}_i \mathbf{s}_i.
\]
This implies
\[
\mathbb{E}\left[\|P \mathbf{w}\|_2\right]=\mathbb{E}\left[\left(\sum_{i=1}^{(2 N+2 M-4) (L-1)} \hat{w}_i^2\right)^{1 / 2}\right] \leq\left(\sum_{i=1}^{(2 N+2 M-4) (L-1)} \mathbb{E}\left[\hat{w}_i^2\right]\right)^{1 / 2},
\]
because $\left\{\mathbf{s}_{i=1}^{(2 N+2 M-4) (L-1)}\right\}$ are orthonormal vectors.

We have
\[
\hat{w}_i=\sum_{j=1}^{NM(L-1)} S_{j i} w_i,
\]
where $w_i$ is the $i$-th entry of $\mathbf{w}$. $S$ is orthogonal hence
\[
\sum_{j=1}^{NM(L-1)} S_{j i}^2=1.
\]
Recall that entries of $\mathbf{w}$ are i.i.d. random variables with mean zero and variance $\zeta^2$. This implies
\[
\mathbb{E}\left[\hat{w}_i^2\right]=\zeta^2 \sum_{j=1}^{NM(L-1)} S_{j i}^2=\zeta^2
\]
Hence
\[
\mathbb{E}\left[\left\|P(\mathbf{v}-\mathbf{u}^{\text {ext }})\right\|_2\right]=\mathbb{E}\left[\|P \mathbf{w}\|_2\right] \leq \sqrt{(2 N+2 M-4) L} \cdot \zeta
\]
Then
\begin{equation} \label{eqn:proof_3}
    \begin{aligned}
    h \delta^{\frac1 2}\mathbb{E}\left[\left\|P\left(\mathbf{v}-\mathbf{u}^{\text {ext }}\right)\right\|_2\right]&=O\left( h \,\delta^{\frac1 2}\sqrt{(2 N+2 M-4) (L-1)} \cdot \zeta \right)\\
    &= O\left( h \,N^{\frac{1}{2}} \cdot\delta^{\frac1 2} L^{\frac{1}{2}} \cdot \zeta\right) = O\left( h^ \frac{1}{2} \,\zeta\right)
    \end{aligned}
\end{equation}

The proof is completed by combining Equations (\ref{eqn:proof_1}), (\ref{eqn:proof_2}), and (\ref{eqn:proof_3}).

\end{proof}

\subsection{Concentration of errors}\label{section: concentration of errors}
As discussed in \cite{dobson2019efficient}, for the purely spatial case, the error term between the solution to the optimization problem and the values of the exact solution of the Fokker-Planck equation always “piles up” right along the boundary of the computational domain. In fact, once the basis of $\operatorname{Ker}(\mathbf{A})$ can be explicitly determined in 1D and 2D spatial case, any random projection into $\operatorname{Ker}(\mathbf{A})$ must concentrate its mass along the boundary of the domain (see Proposition 2.1 in \cite{dobson2019efficient}). Inspired by this observation, we can carry over a similar idea to study the concentration of errors in $\mathbb{R}^n \times \mathbb{R}$, taking into account the time variable. We will analyze, for the periodic Fokker-Planck equation, how the error term $\mathbf{u}-\mathbf{u}^{\text {ext}}$ concentrates at the boundary of the domain, which includes both the spatial boundary layers and the two endpoints of the time period. A strong concentration of error term at the boundary layer means the empirical performance of our algorithm is actually much better than the theoretical bound given in Theorem \ref{thm: error analysis}. 

In the general case, the basis of $\operatorname{Ker}(\mathbf{A})$ cannot be explicitly obtained. Now we need to consider a numerical domain with $N \times N \times (L-1)$ grids, corresponding to a square spatial domain at each discrete time level. Let $\Theta_D$ denote the subspace spanned by coordinate vectors corresponding to the boundary layer with thickness $D$. In other words,
\[
\begin{aligned}
  \Theta_D \;=\; \operatorname{span}\bigl\{
    \mathbf{e}_{i,j,k} \mid\; &i \le D \;\text{or}\; j \le D \;\text{or}\; k \le D \;\text{or}\\
      &i \ge N - D \;\text{or}\; j \ge N - D \;\text{or}\; k \ge (L-1) - D
  \bigr\}\,,
\end{aligned}
\]
where $\mathbf{e}_{i,j,k}$ is the elementary vector corresponding to node $(i,j,k)$. 

Denote $dk:= \operatorname{dim}(\operatorname{Ker(\mathbf{A})})=(2 N+2 M-4) (L-1)$, as computed in the proof of Theorem \ref{thm: error analysis}. Then we define a sequence of angles $0 \leq \theta_1 \leq \cdots \leq \theta_{dk} \leq \pi / 2$ that describe the angle between $\operatorname{Ker}(\mathbf{A})$ and $\Theta_D$. Here $\{ \theta_i\}_{i=1}^{dk}$ are called principal angles. 

We can compute the principal angles between $\operatorname{Ker}(\mathbf{A})$ and $\Theta_D$. If most principal angles are small, $\operatorname{Ker}(\mathbf{A})$ is almost parallel to $\Theta_D$, and the projection of a random vector onto $\Theta_D$ preserves most of its length. Equivalently, we can say that there is a concentration of errors at the boundary of the domain.

Here is the construction of principal angles $\theta_1, \theta_2,\cdots, \theta_{dk}$. The first one is
\[
\theta_1=\min \left\{\left.\arccos \left(\frac{\boldsymbol{\alpha} \cdot \boldsymbol{\beta}}{\|\boldsymbol{\alpha}\|\|\boldsymbol{\beta}\|}\right) \right\rvert\, \boldsymbol{\alpha} \in \operatorname{Ker}(\mathbf{A}), \boldsymbol{\beta} \in \Theta_D\right\}=\angle\left(\boldsymbol{\alpha}_1, \boldsymbol{\beta}_1\right) .
\]
Other angles are defined recursively with
\[
\theta_i=\min \left\{\left.\arccos \left(\frac{\boldsymbol{\alpha} \cdot \boldsymbol{\beta}}{\|\boldsymbol{\alpha}\|\|\boldsymbol{\beta}\|}\right) \right\rvert\, \boldsymbol{\alpha} \in \operatorname{Ker}(\mathbf{A}), \boldsymbol{\beta} \in \Theta_D, \boldsymbol{\alpha} \perp \boldsymbol{\alpha}_j, \boldsymbol{\beta} \perp \boldsymbol{\beta}_j, \forall 1 \leq j \leq i-1\right\},
\]
such that $\angle\left(\boldsymbol{\alpha}_i, \boldsymbol{\beta}_i\right)=\theta_i$. Without loss of generality, assume $\left\|\boldsymbol{\alpha}_i\right\|=1$ for all $1 \leq i \leq$ $dk$. Since $\operatorname{dim}\left(\Theta_D\right) \geq \operatorname{dim}(\operatorname{Ker}(\mathbf{A}))$, it is easy to see that $\left\{\boldsymbol{\alpha}_1, \cdots, \boldsymbol{\alpha}_{dk}\right\}$ forms an orthonormal basis of $\operatorname{Ker}(\mathbf{A})$. Recall that $\mathbf{u} \in \operatorname{Ker}(\mathbf{A})$ and that the error $\mathbf{u}-\mathbf{u}^{\text {ext }}$ is approximated by the projection of a random vector $\mathbf{w}$ with i.i.d. entries onto the subspace $\operatorname{Ker}(\mathbf{A})$. Hence we can further assume that $\boldsymbol{\xi}=\mathbf{u}-\mathbf{u}^{\text {ext }}$ is approximated by a random vector
\begin{equation} \label{eqn: xi}
    \boldsymbol{\xi}=\sum_{i=1}^{dk} c_i \boldsymbol{\alpha}_i,
\end{equation}
where $c_i$ are i.i.d. random variables with mean 0 and variance $\zeta^2$. Define
\[
p_D(\boldsymbol{\xi})=\frac{\mathbb{E}\left[\left\|P_{\Theta_D} \boldsymbol{\xi}\right\|\right]}{\mathbb{E}[\|\boldsymbol{\xi}\|]}
\]
as the mean weight of $\boldsymbol{\xi}$ projected onto the boundary layer, where $P_{\Theta_D}$ is the projection matrix onto $\Theta_D$. Assuming $\boldsymbol{\xi}$ satisfies Equation (\ref{eqn: xi}), we have
\[
p_D(\boldsymbol{\xi})=\frac{1}{dk} \sum_{i=1}^{dk} \cos \left(\theta_i\right)
\]
Therefore, $p_D(\boldsymbol{\xi})$ measures the degree of concentration of errors on the boundary layer with thickness $D$. The greater the value of $p_D(\boldsymbol{\xi})$ is, the higher the probability of the concentration of errors at the boundary.

Now we consider the following three examples in $2D \times 1D$ and observe the related principal angles between $\operatorname{Ker}(\mathbf{A})$ and $\Theta_D$ for $D=1,2,3$.

\begin{itemize}
    \item Example 1 (diffusion process without drift)
    \[
    d \boldsymbol{X}_t= \varepsilon \, \mathrm{d} \boldsymbol{W}_t \,,
    \]
    with $\varepsilon = 1$.
    \item Example 2
    \[
    \left\{\begin{array}{l}
    d X_t=\left(-(X_t+Y_t-f(t))+\frac{1}{2} f^{\prime}(t)\right) \mathrm{d} t+\varepsilon_1 \, \mathrm{d} W_t^x \\
    d Y_t=\left(-(X_t+Y_t-f(t))+\frac{1}{2} f^{\prime}(t)\right) \mathrm{d} t+\varepsilon_2 \, \mathrm{d} W_t^y
    \end{array}\right. \, ,
    \]
    with $f(t) = \sin{t}$, and $\varepsilon_1 = \varepsilon_2 = 1$.
    \item Example 3
    \[
    \left\{\begin{array}{l}
    d X_t=\left(-4 X_t\left(X_t^2+Y_t^2-f(t)\right)+\frac{X_t f^{\prime}(t)}{2\left(X_t^2+Y_t^2\right)}\right) \mathrm{d} t+ \varepsilon_1 \, \mathrm{d} W_t^x \\
    d Y_t=\left(-4 Y_t\left(X_t^2+Y_t^2-f(t)\right)+\frac{Y_t f^{\prime}(t)}{2\left(X_t^2+Y_t^2\right)}\right) \mathrm{d} t+\varepsilon_2 \, \mathrm{d} W_t^y
    \end{array}\right. \, ,
    \]
    with $f(t) = 5 + \sin{t}$, and $\varepsilon_1 = \varepsilon_2 = \sqrt{2}$.
\end{itemize}

%%% Figure 1 %%%
\begin{figure}[htbp]
  \centering

  % Row 1
  \begin{subfigure}[b]{0.32\textwidth} %0.32
    \centering
    \includegraphics[width=1.15\textwidth]{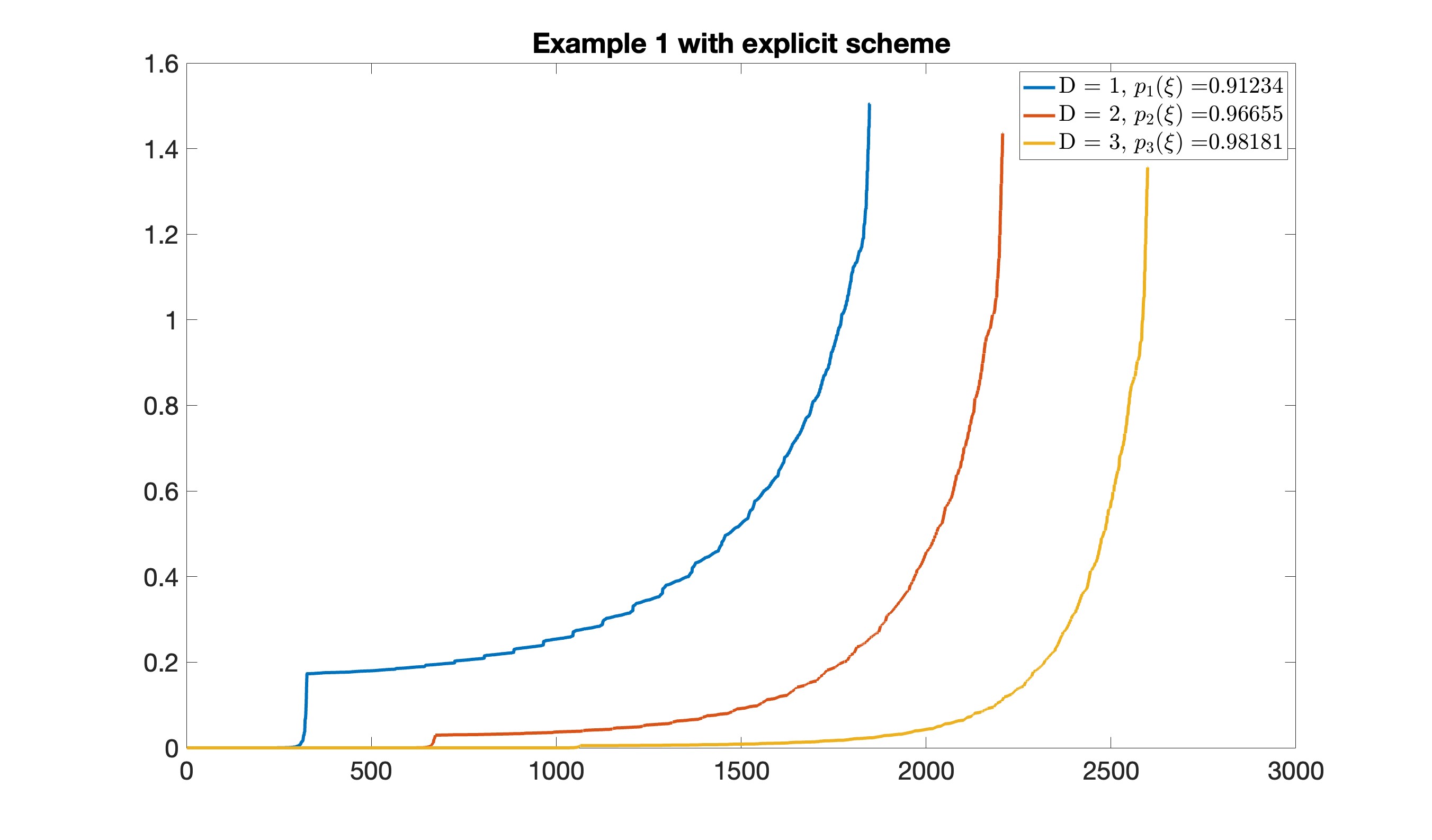} %1.2
    %\caption{Caption 1}
    %\label{fig:img1}
  \end{subfigure}
  \hfill
  \begin{subfigure}[b]{0.32\textwidth}
    \centering
    \includegraphics[width=1.15\textwidth]{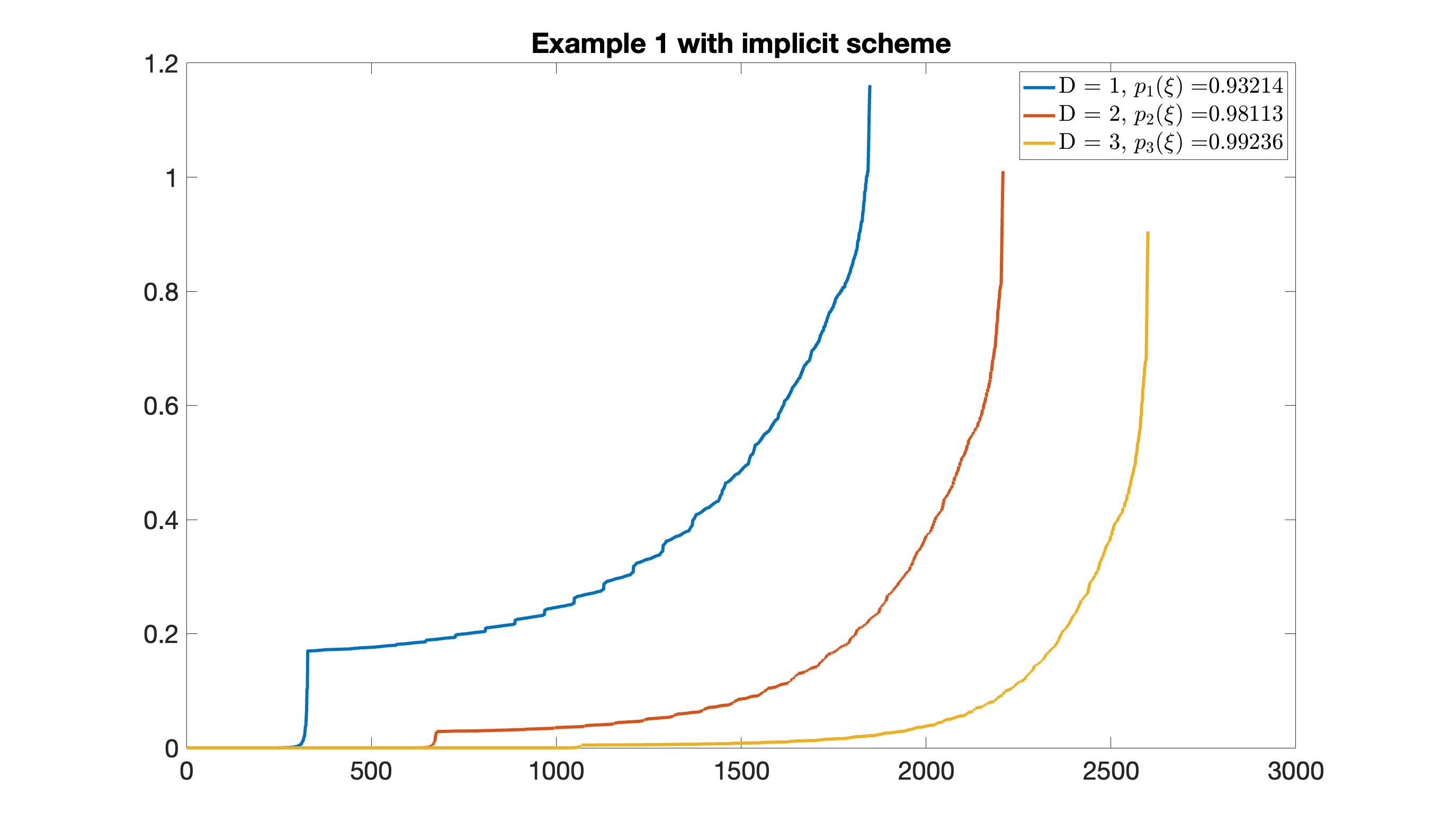}
    %\caption{Example 1}
    %\label{fig:img2}
  \end{subfigure}
  \hfill
  \begin{subfigure}[b]{0.32\textwidth}
    \centering
    \includegraphics[width=1.15\textwidth]{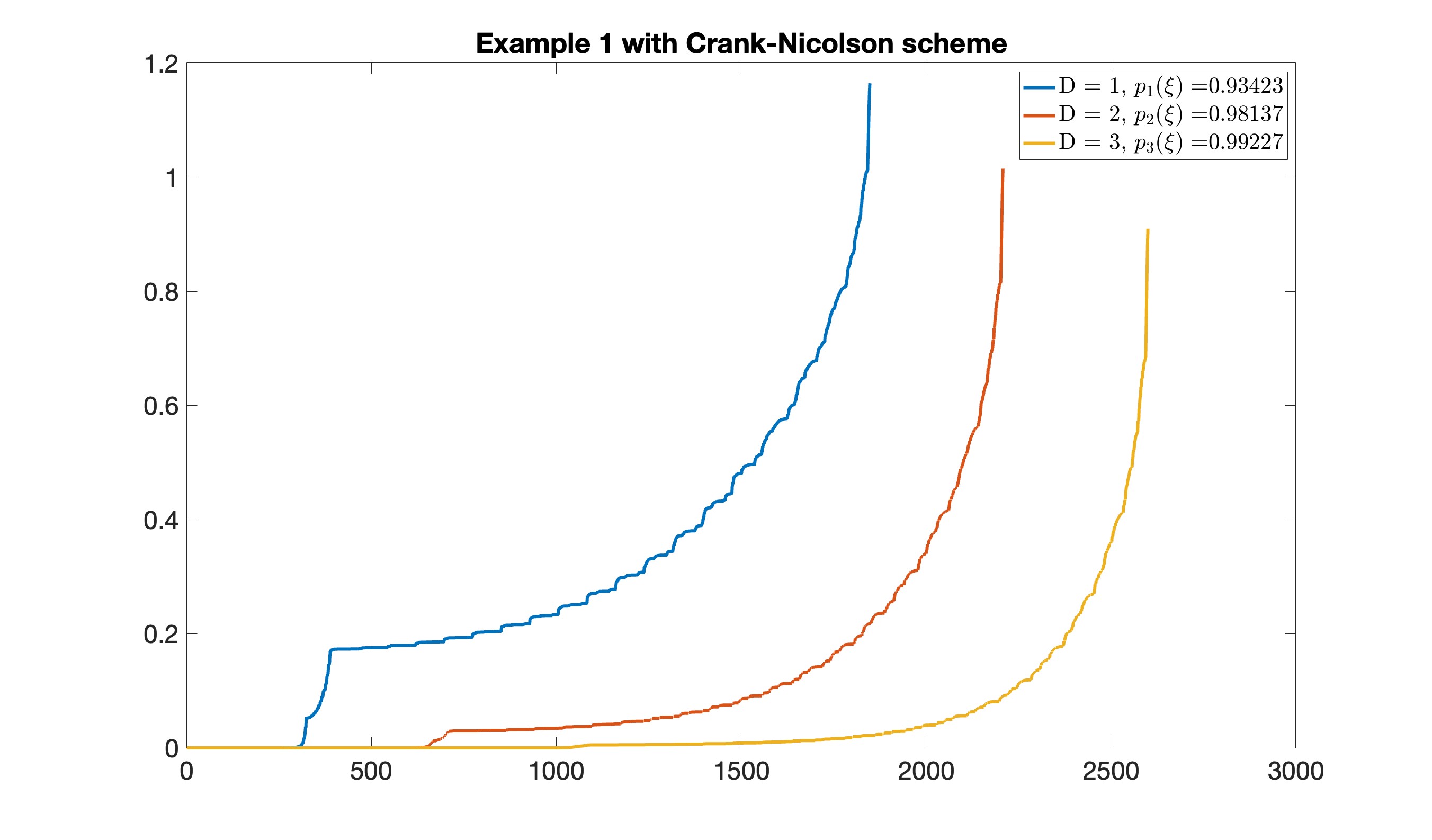}
    %\caption{Caption 3}
    %\label{fig:img3}
  \end{subfigure}
  %\caption{Row 1: Example 1}
  
  \vspace{1.7em} % small vertical space between rows

  % Row 2
  \begin{subfigure}[b]{0.32\textwidth}
    \centering
    \includegraphics[width=1.15\textwidth]{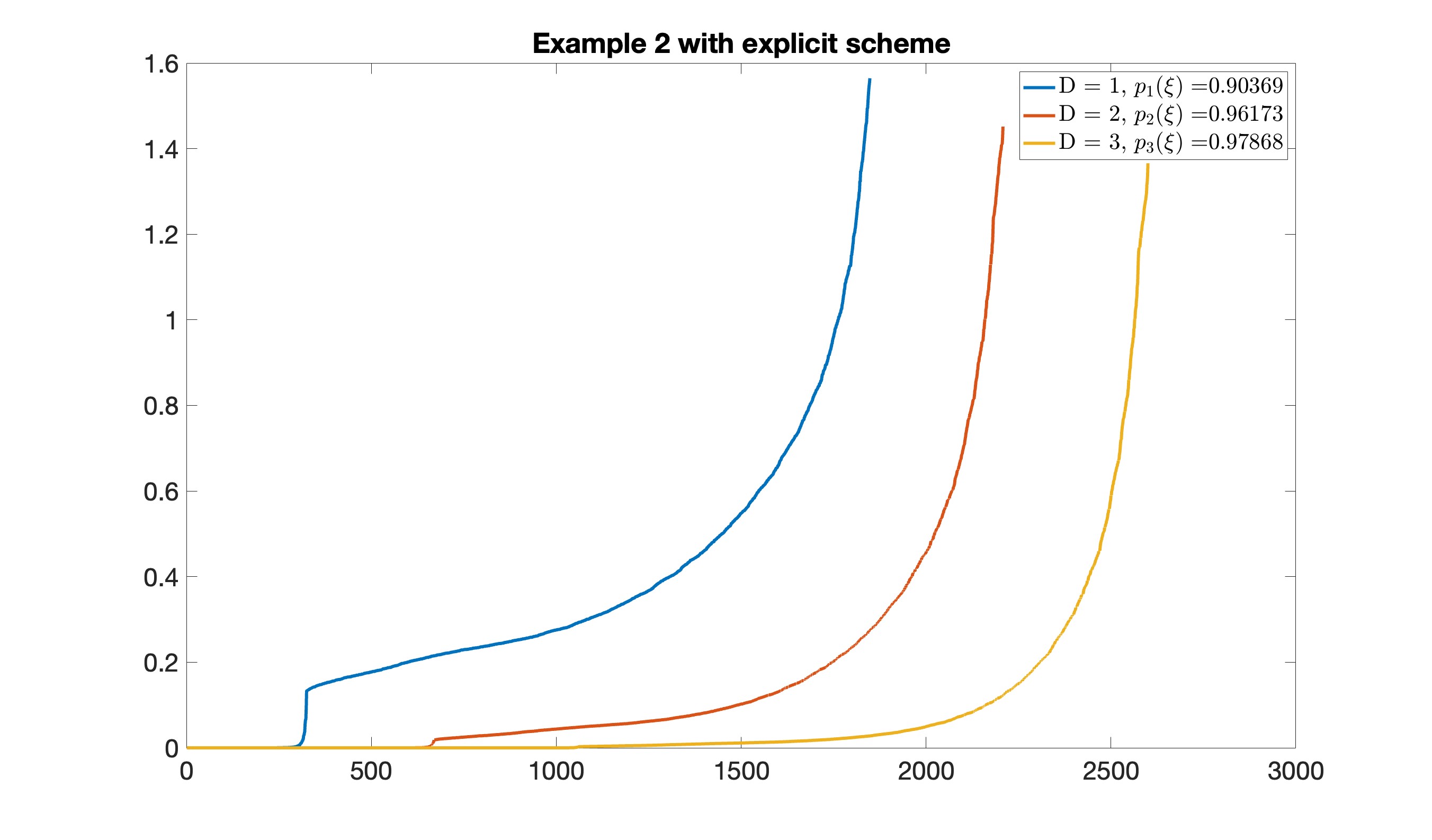}
    %\caption{Caption 4}
    \label{fig:img4}
  \end{subfigure}
  \hfill
  \begin{subfigure}[b]{0.32\textwidth}
    \centering
    \includegraphics[width=1.15\textwidth]{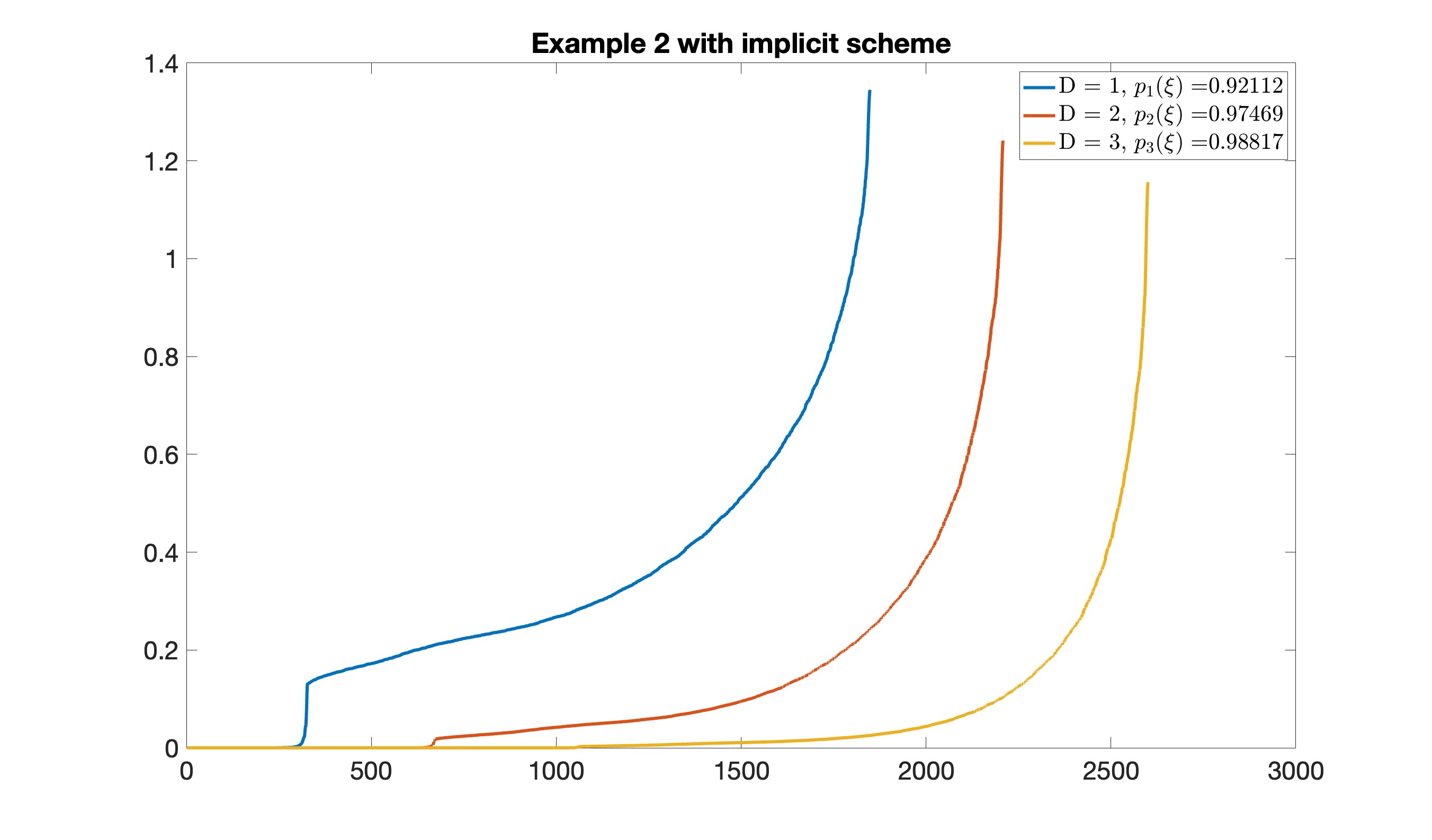}
    %\caption{Example 2}
    \label{fig:img5}
  \end{subfigure}
  \hfill
  \begin{subfigure}[b]{0.32\textwidth}
    \centering
    \includegraphics[width=1.15\textwidth]{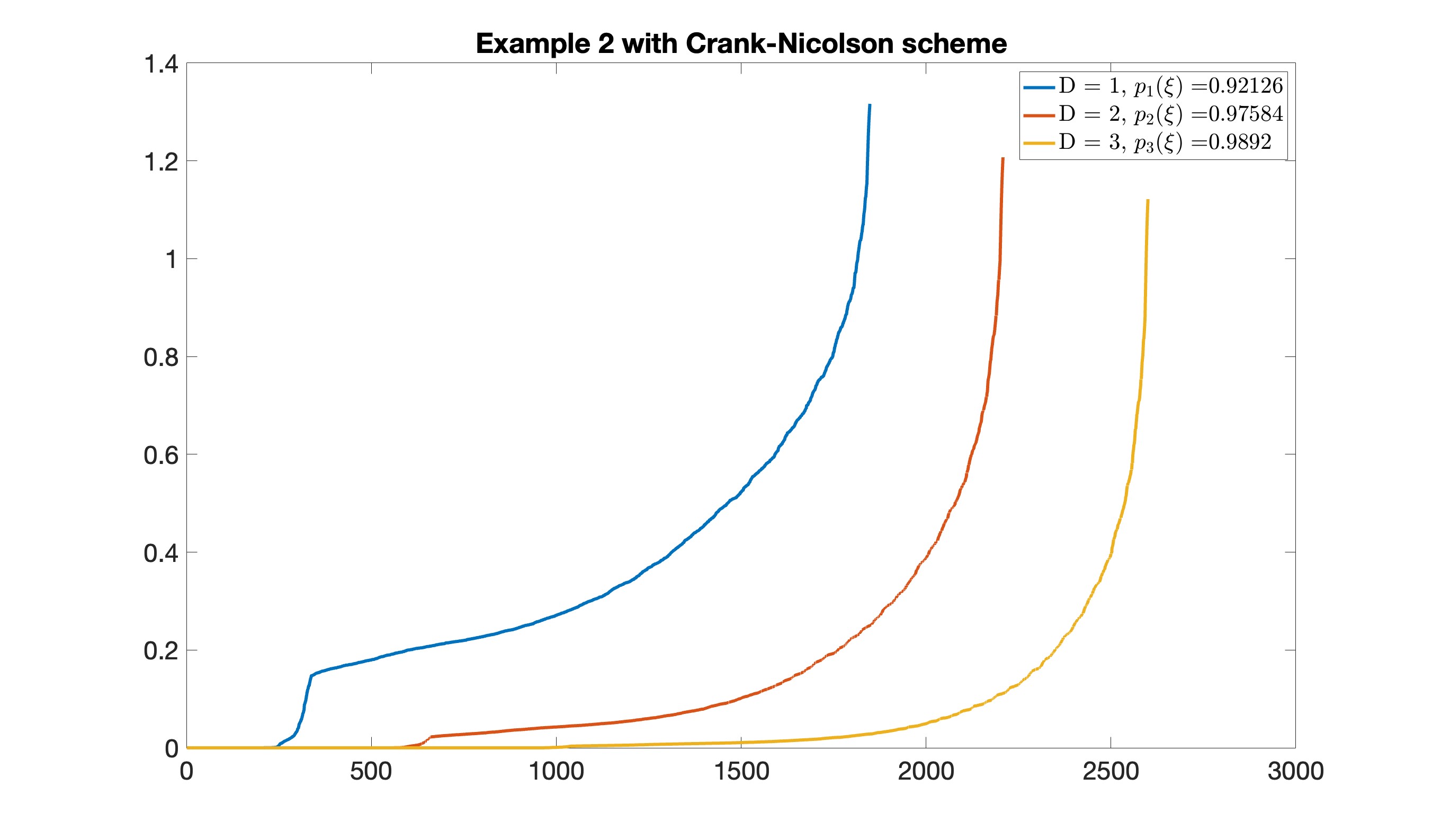}
    %\caption{Caption 6}
    \label{fig:img6}
  \end{subfigure}
  %\caption{Example 2}

  \vspace{1em} % small vertical space between rows

  % Row 3
  \begin{subfigure}[b]{0.32\textwidth}
    \centering
    \includegraphics[width=1.15\textwidth]{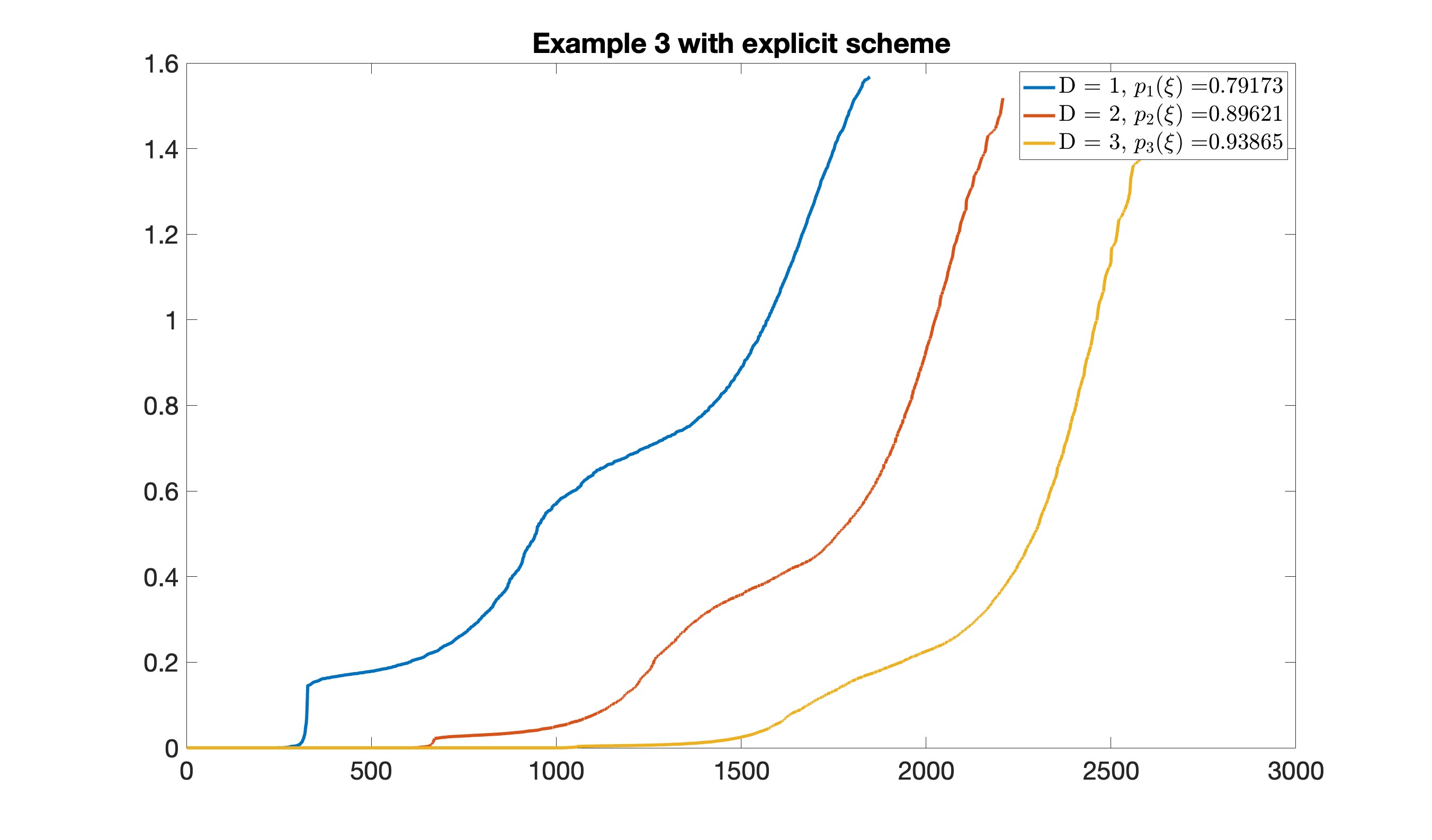}
    %\caption{Caption 7}
    \label{fig:img7}
  \end{subfigure}
  \hfill
  \begin{subfigure}[b]{0.32\textwidth}
    \centering
    \includegraphics[width=1.15\textwidth]{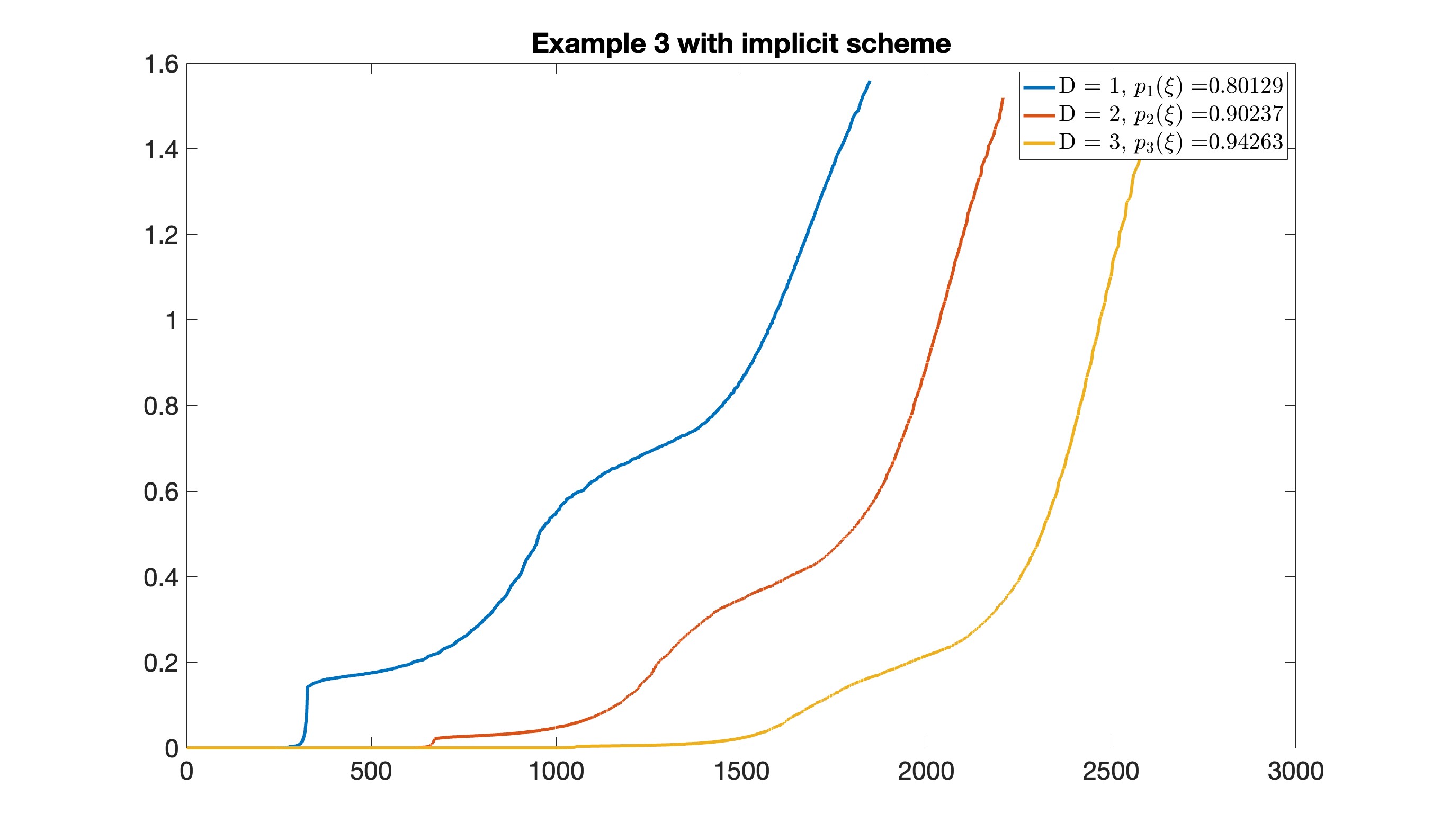}
    %\caption{Example 3}
    \label{fig:img8}
  \end{subfigure}
  \hfill
  \begin{subfigure}[b]{0.32\textwidth}
    \centering
    \includegraphics[width=1.15\textwidth]{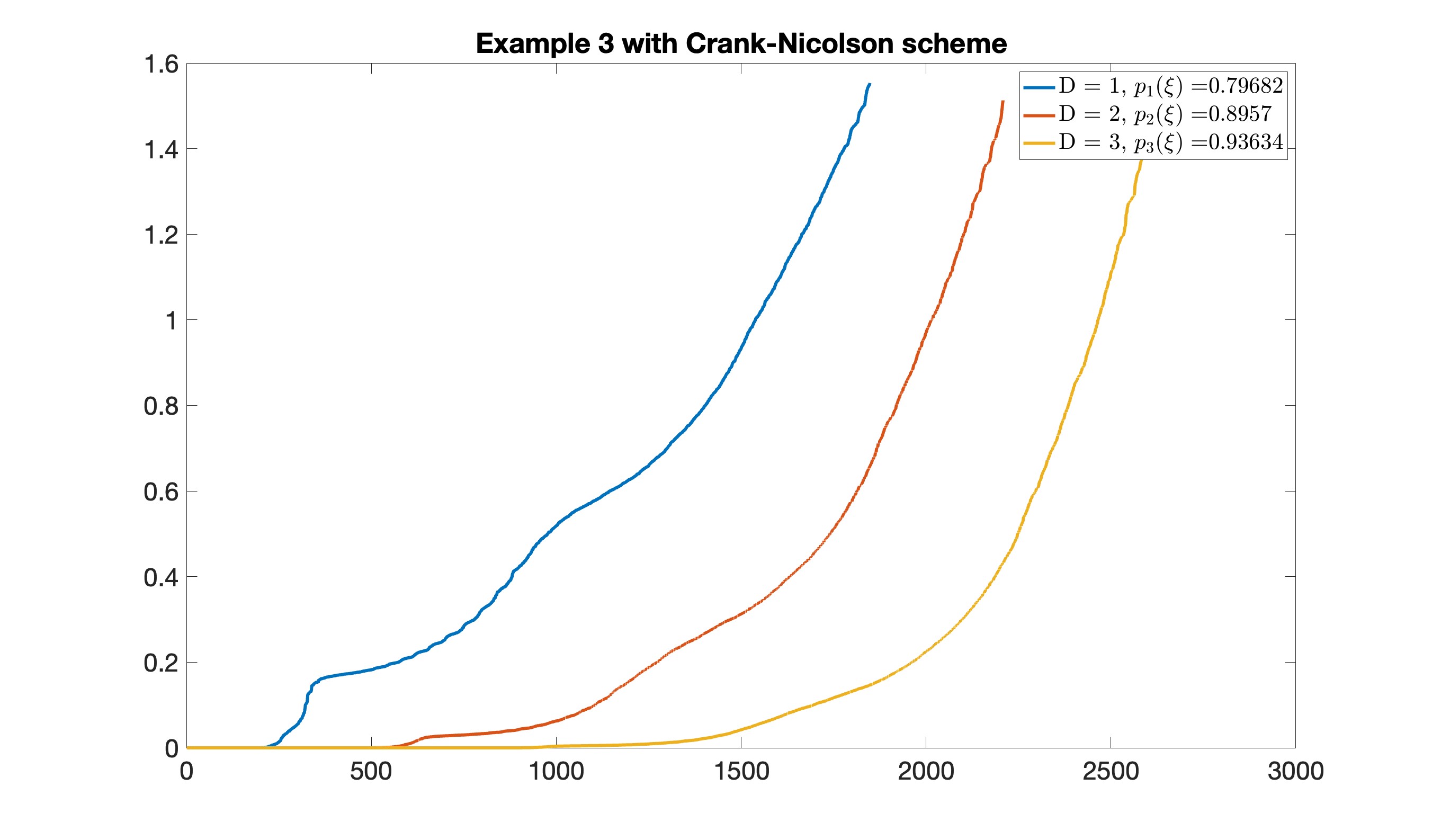}
    %\caption{Caption 9}
    \label{fig:img9}
  \end{subfigure}
  %\caption{Row 3: Example 3}

  \caption{principal angles between $\operatorname{Ker}\mathbf{A}$ and $\Theta_D$ for $D=1,2,3$. First row: Example 1; second row: Example 2; third row: Example 3. Left: Forward Euler method; middle: Backward Euler method; right: Crank-Nicolson method.}
  \label{principal angles for three mthods}
\end{figure}

Principal angles can be numerically computed by an SVD decomposition. In Figure \ref{principal angles for three mthods}, we plot all principal angles of Example 1, 2, 3 for $D=1,2,3$ in an ascending order with three discretization schemes of 2D periodic Fokker-Planck equations, which are Forward Euler method, Backward Euler method and the Crank-Nicolson method, respectively. Each row in the figure corresponds to one example. In each row, the left panel corresponds to Forward Euler method, the middle corresponds to Backward Euler method, and the right corresponds to Crank-Nicolson method. The size of a block is $(20+D)\times (20+D)\times (20 + D)$. We observe that the mean weight of $\boldsymbol{\xi}$ projected onto $\Theta_D$ is very large. As a result, most of the error term $\mathbf{u}-\mathbf{u}^{\text {ext}}$ is concentrated in the boundary layer. We also remark that different discretization schemes do not have significant influence on the degree of error concentration. However, the Crank-Nicolson scheme is preferred, since it yields a larger mean weight than the implicit and explicit schemes, meaning that more of the error is "pushed" to the boundary.

%%% Figure 2 %%%
%%% implicit
\begin{figure}
    \centering
    \begin{minipage}{0.45\linewidth}
        \centering
        \includegraphics[width=1.15\linewidth]{with_drift_implicit.jpg} % first figure itself
    \end{minipage}\hfill
    \begin{minipage}{0.45\linewidth}
        \centering
        \includegraphics[width=1.15\linewidth]{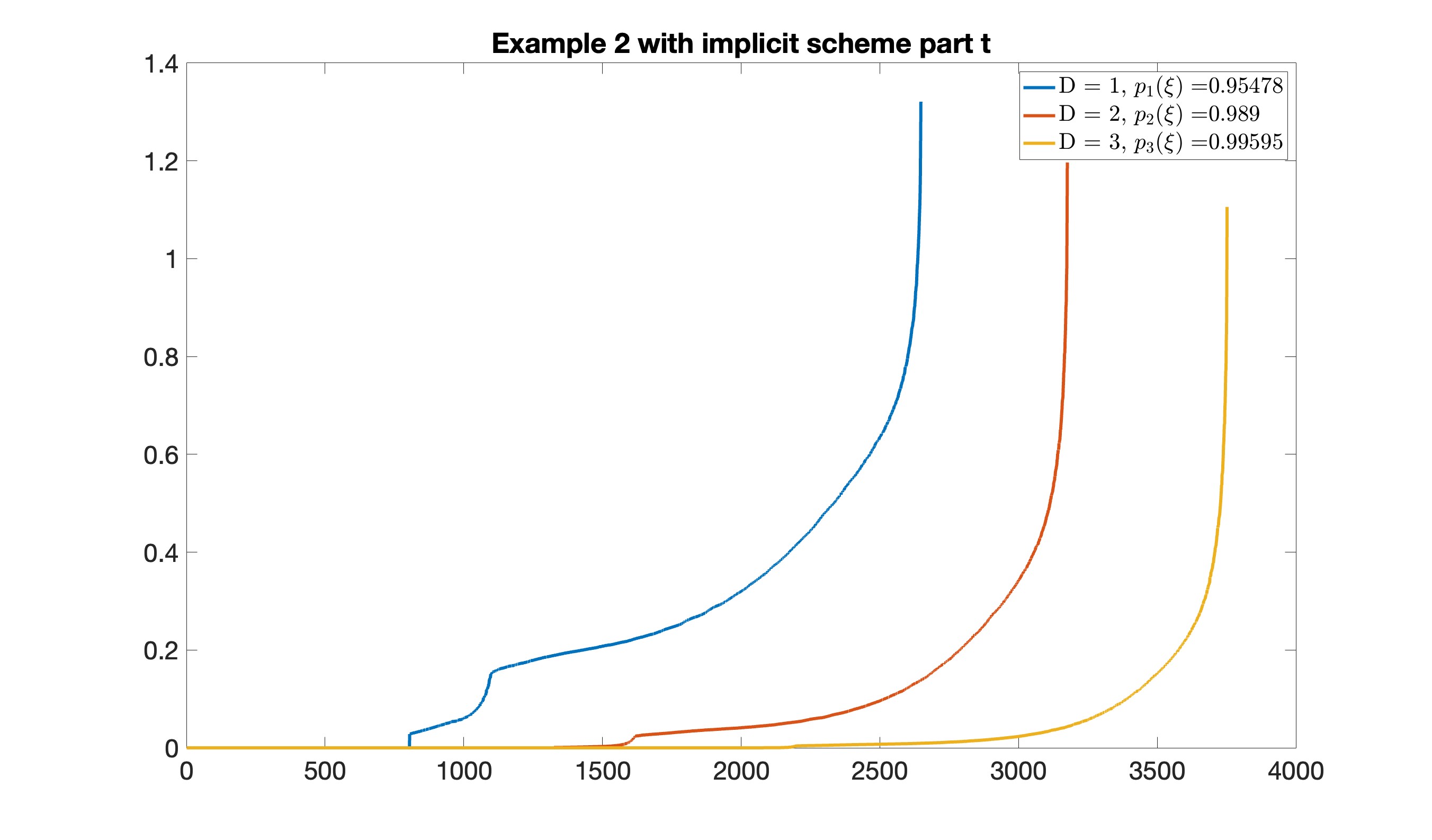} % second figure itself
    \end{minipage}
    \caption{Implicit scheme for computing the principal angles between $\operatorname{Ker}\mathbf{A}$ and $\Theta_D$ for $D=1,2,3$. Left: $t$ in whole time interval. Right: $t$ in part time interval.}
    \label{fig: Case 1 and 2}
\end{figure}

%%% explicit
\begin{figure}
    \centering
    \begin{minipage}{0.45\linewidth}
        \centering
        \includegraphics[width=1.15\linewidth]{with_drift_explicit.jpg} % first figure itself
    \end{minipage}\hfill
    \begin{minipage}{0.45\linewidth}
        \centering
        \includegraphics[width=1.15\linewidth]{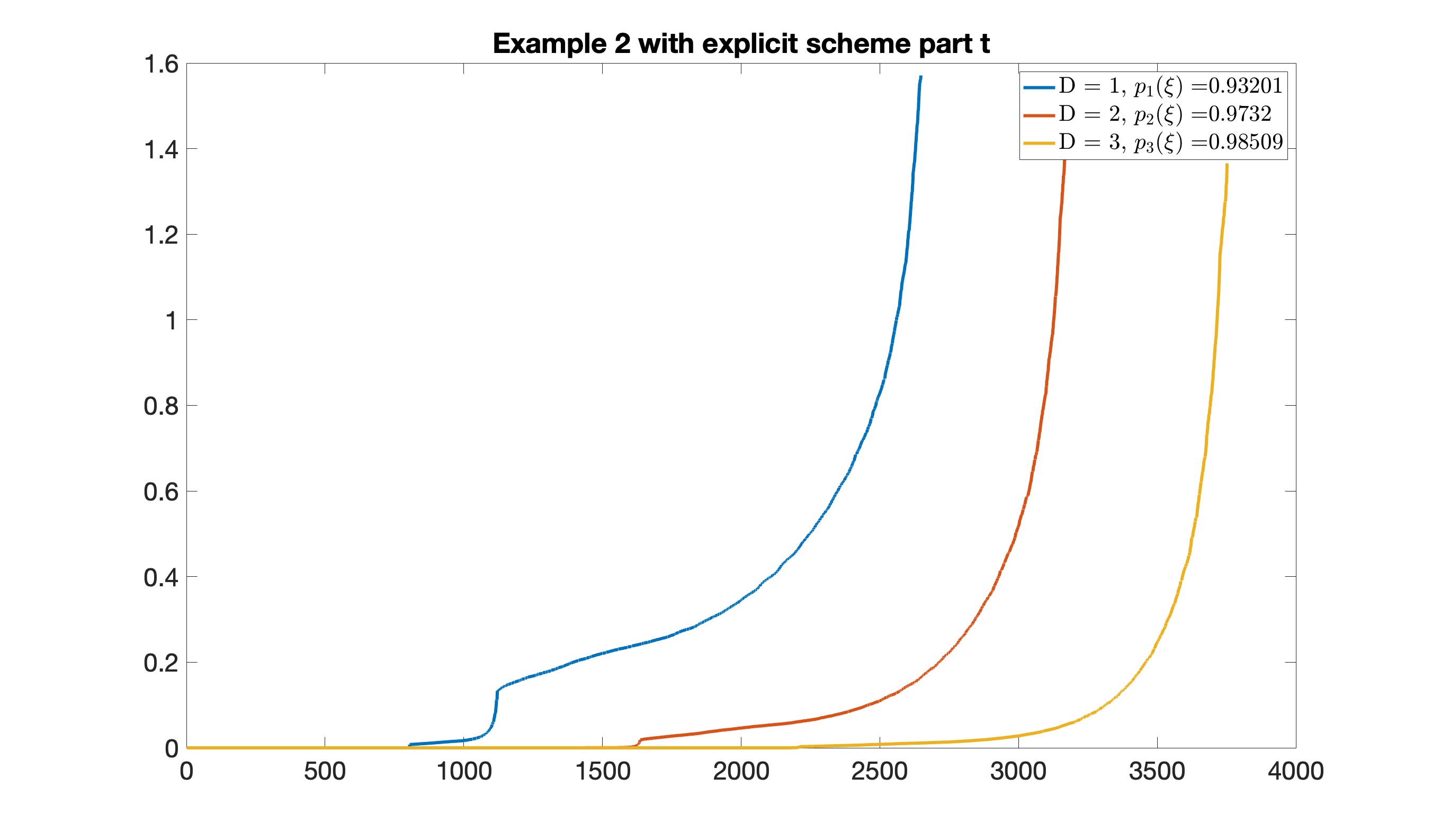} % second figure itself
    \end{minipage}
    \caption{Explicit scheme for computing the principal angles between $\operatorname{Ker}\mathbf{A}$ and $\Theta_D$ for $D=1,2,3$. Left: $t$ in whole time interval. Right: $t$ in part time interval.}
    \label{fig: explicit Case 1 and 2}
\end{figure}

%%% Crank-Nicolson
\begin{figure}
    \centering
    \begin{minipage}{0.45\linewidth}
        \centering
        \includegraphics[width=1.15\linewidth]{with_drift_Crank-Nicolson.jpg} % first figure itself
    \end{minipage}\hfill
    \begin{minipage}{0.45\linewidth}
        \centering
        \includegraphics[width=1.15\linewidth]{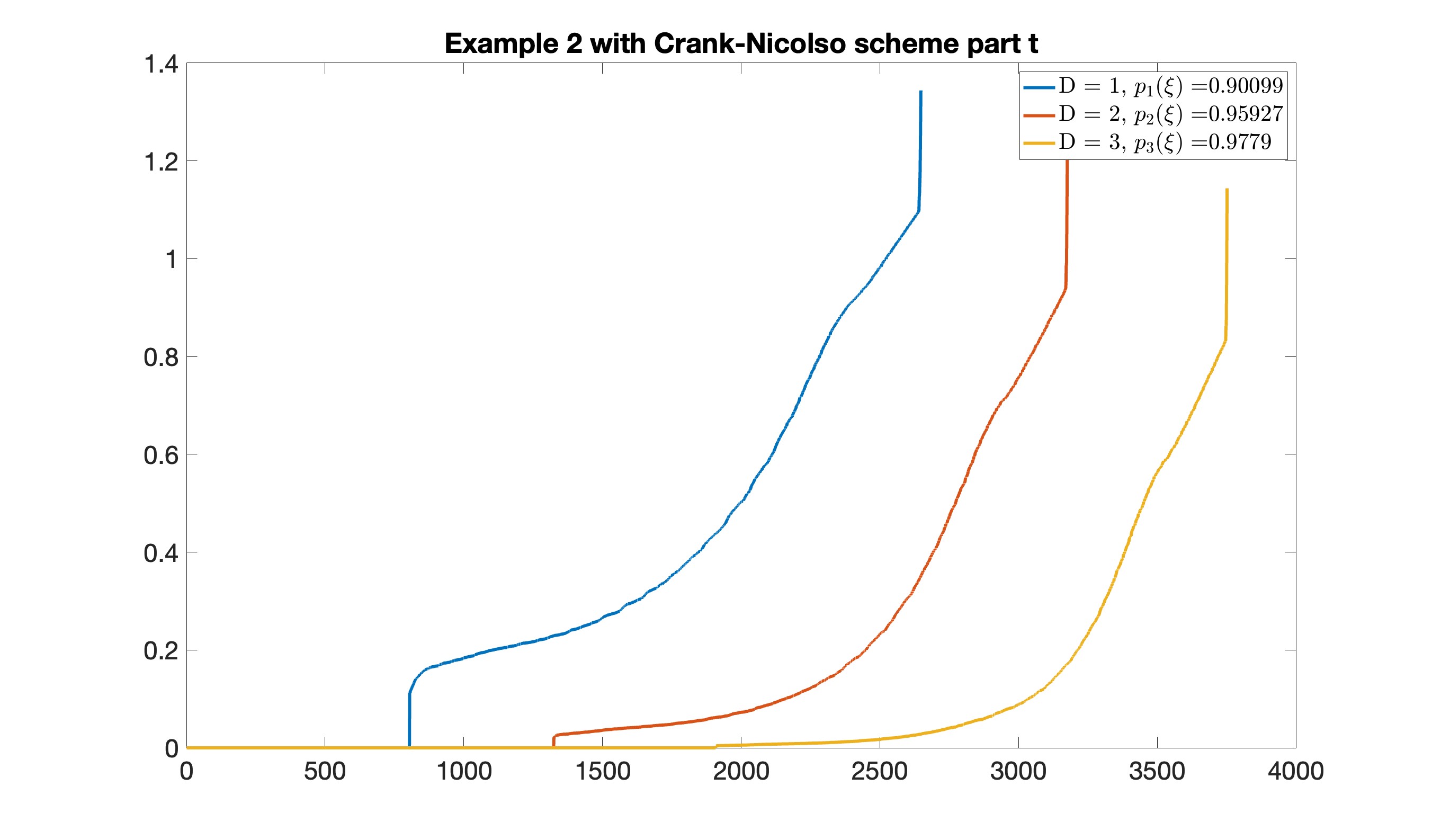} % second figure itself
    \end{minipage}
    \caption{Crank-Nicolson scheme for computing the principal angles between $\operatorname{Ker}\mathbf{A}$ and $\Theta_D$ for $D=1,2,3$. Left: $t$ in whole time interval. Right: $t$ in part time interval.}
    \label{fig: Crank-Nicolson Case 1 and 2}
\end{figure}

The time variable plays an important role in the periodic Fokker-Planck equation studied in this paper. Moreover, the Fokker-Planck equation only involves the first order derivative in the time variable, potentially meaning a less regularization effect in the optimization. Therefore, we need to examine whether the projection onto $\mathrm{Ker}(\mathbf{A})$ has a similar effect in the time direction as in the spatial direction. To check this, we compute two different cases for time $t$. In Case 1, we consider the whole time interval $[0,2\pi]$ as an example. Time $t$ should satisfy the periodic constrain. Then, matrix $\mathbf{A}$ is $(N-2)^2(L-1) \times N^2(L-1)$, and the dimension of $\operatorname{Ker}(\mathbf{A})$, $dk = (4N-4)(L-1)$. In Case 2, we randomly choose a small block within the whole time interval and focus only on that block. In other words, time $t$ in this small block can be totally regarded as a spatial variable. Matrix $\mathbf{A}$ is $(N-2)^2((L-1)-2) \times N^2(L-1)$, and the dimension of $\operatorname{Ker}(\mathbf{A})$, $dk = (4N-4)(L-1) + 2(N-2)^2$. In Figure \ref{fig: Case 1 and 2} and Figure \ref{fig: explicit Case 1 and 2}, we can see the mean weight of $\boldsymbol{\xi}$ projected to $\Theta_D$ in Case 2 (part time interval) is larger than the mean weight of $\boldsymbol{\xi}$ projected to $\Theta_D$ in Case 1 (whole time interval). Of course, on the other hand, $\mathrm{Ker}(\mathbf{A})$ has a higher dimension in Case 2, which means there are more zero principal angles. Those zero principal angles contribute to a larger mean weight $p(\boldsymbol{\xi})$. Nevertheless, this numerical experiment at least shows that the error term is still "pushed" to the boundary even when the numerical domain does not cover the entire time period. However, in Figure \ref{fig: Crank-Nicolson Case 1 and 2}, the mean weight of $\boldsymbol{\xi}$ projected onto $\Theta_D$ in Case 2 (part time interval) is smaller than the mean weight of $\boldsymbol{\xi}$ projected to $\Theta_D$ in Case 1 (whole time interval). The Crank-Nicolson method is not affected by the higher dimension of $\mathrm{Ker}(\mathbf{A})$ in the part time case. Thus, the Crank-Nicolson scheme is preferred to the implicit and explicit scheme.

The metric $p_D(\boldsymbol{\xi})$ does not show how the error concentrates at a particular section of the boundary. Figure \ref{fig: error distribution} provides an additional empirical test of the spatial distribution of error terms at each fixed time. The Fokker-Planck equation is taken from Example 3, and the corresponding ring density function is $u(x, y, t)=e^{-\left(x^2+y^2-f(t)\right)^2}$ with $f(t) = 5 + \sin t$. We choose a $50 \times 50 \times 50$ grid on $[-3,3] \times [-3,3] \times [0,2\pi]$ and solve the Fokker-Planck equation using a data-driven finite difference solver, with sample size $N = 10^6$ and time step $dt = 0.001$. $u_{\text{notp}}$ denotes the numerical solution without the periodic condition on the time variable $t$, and $u_{\text{p}}$ denotes the numerical solution with the periodic condition on the time variable $t$, i.e., $u_{\text{p}}(t=0) = u_{\text{p}}(t=2\pi)$. In comparison, we observe that at the temporal boundaries $t=0$ and $t=2\pi$, the error term $u_{\text{p}}-u^{\text{ext}}$ concentrates at the boundary layer, whereas the the error term $u_{\text{notp}}-u^{\text{ext}}$ is randomly distributed throughout the whole spatial domain. In other words, considering better accuracy at the time boundary, letting the numerical domain cover the entire period is still preferred because with the periodic boundary condition, the error term concentrates at the spatial boundary regardless of the time variable.

\begin{figure}[htbp]
    \centering

    % Row 1
    \begin{subfigure}[b]{0.45\textwidth}
        \includegraphics[width=1.25\textwidth]{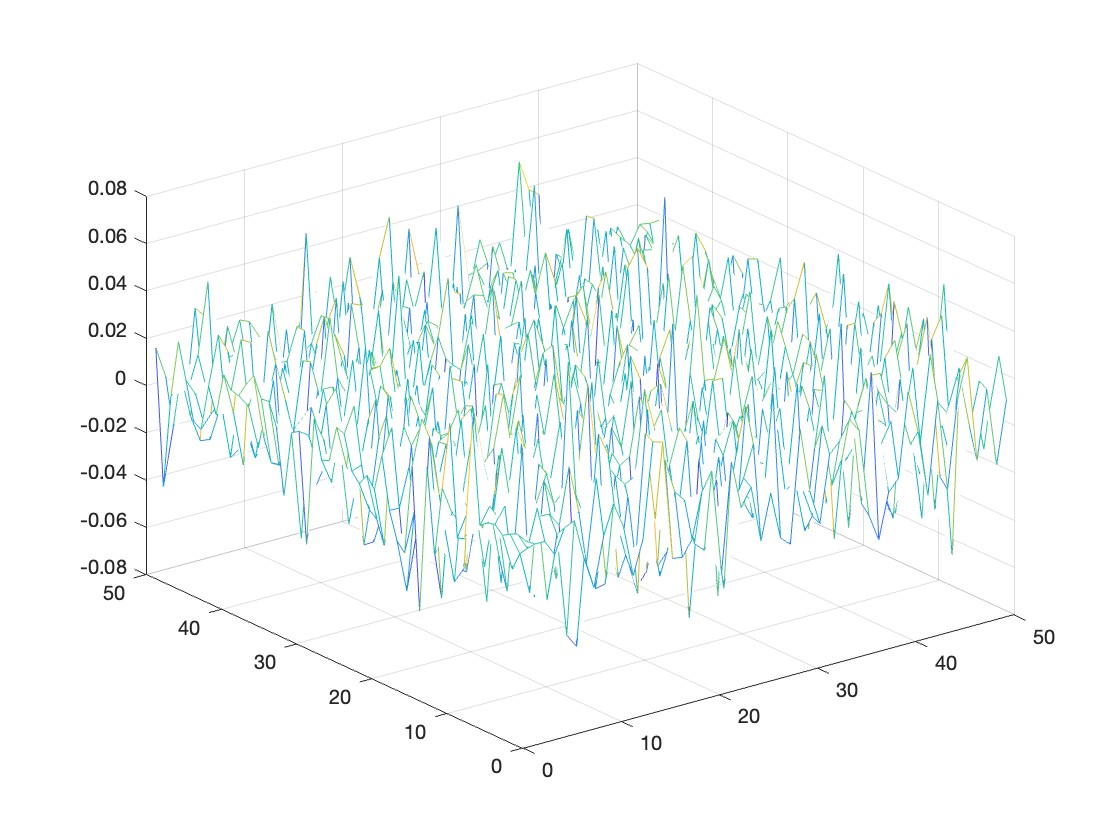}
        \caption{Error term $u_{\text{notp}}-u^{\text{ext}}$ at $t=0$}
    \end{subfigure}
    \hfill
    \begin{subfigure}[b]{0.45\textwidth}
        \includegraphics[width=1.25\textwidth]{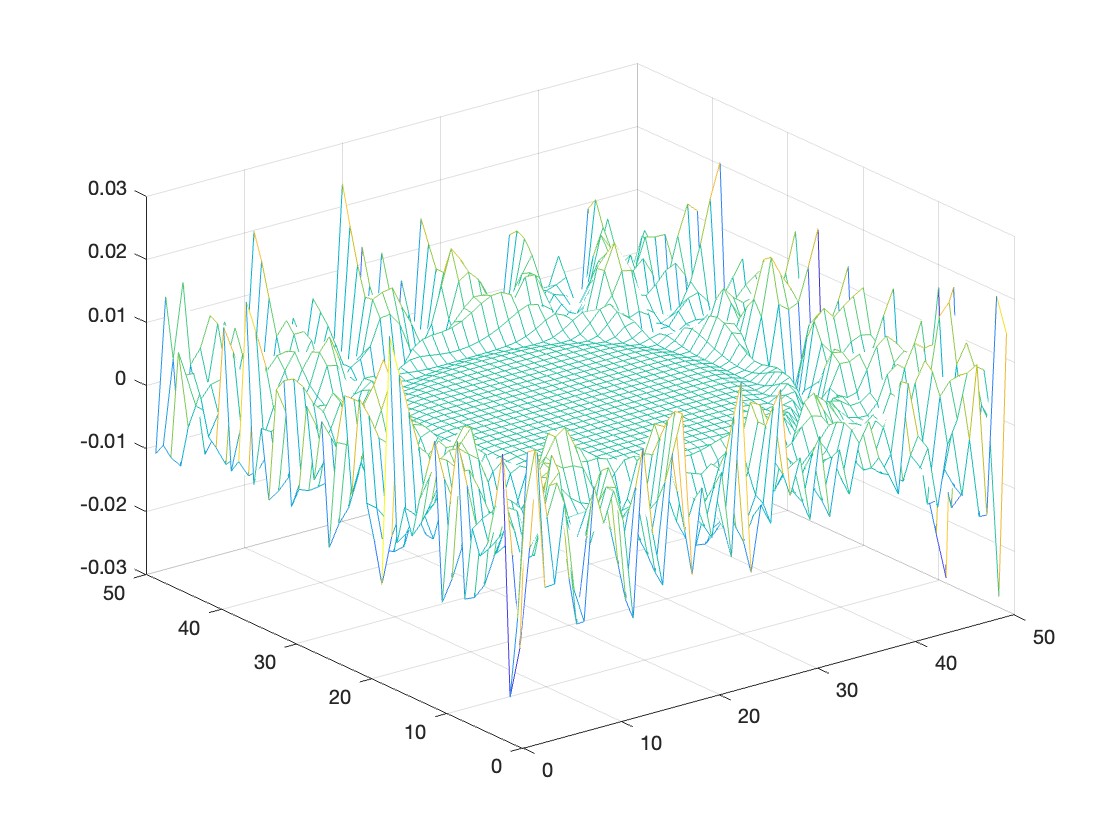}
        \caption{Error term $u_{\text{p}}-u^{\text{ext}}$ at $t=0$}
    \end{subfigure}

    \vskip\baselineskip

    % Row 2
    \begin{subfigure}[b]{0.45\textwidth}
        \includegraphics[width=1.25\textwidth]{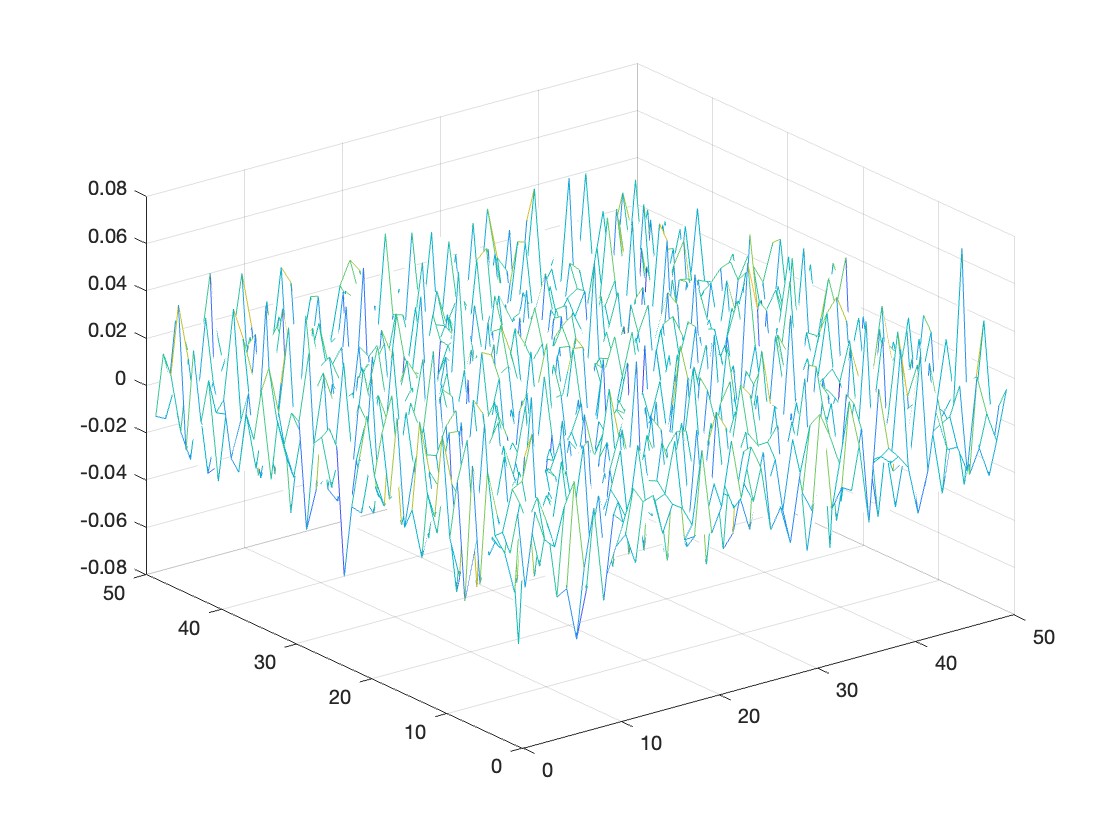}
        \caption{Error term $u_{\text{notp}}-u^{\text{ext}}$ at $t=2\pi$}
    \end{subfigure}
    \hfill
    \begin{subfigure}[b]{0.45\textwidth}
        \includegraphics[width=1.25\textwidth]{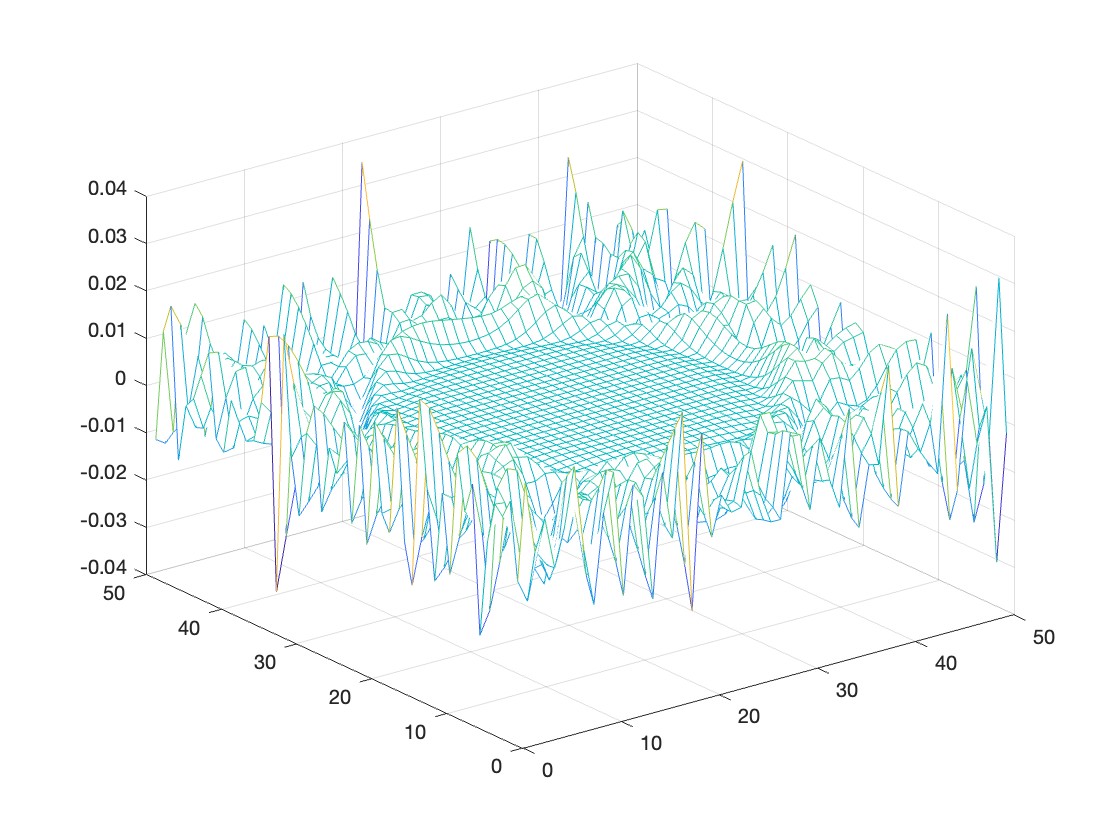}
        \caption{Error term $u_{\text{p}}-u^{\text{ext}}$ at $t=2\pi$}
    \end{subfigure}

    \vskip\baselineskip

    % Row 3
    \begin{subfigure}[b]{0.45\textwidth}
        \includegraphics[width=1.25\textwidth]{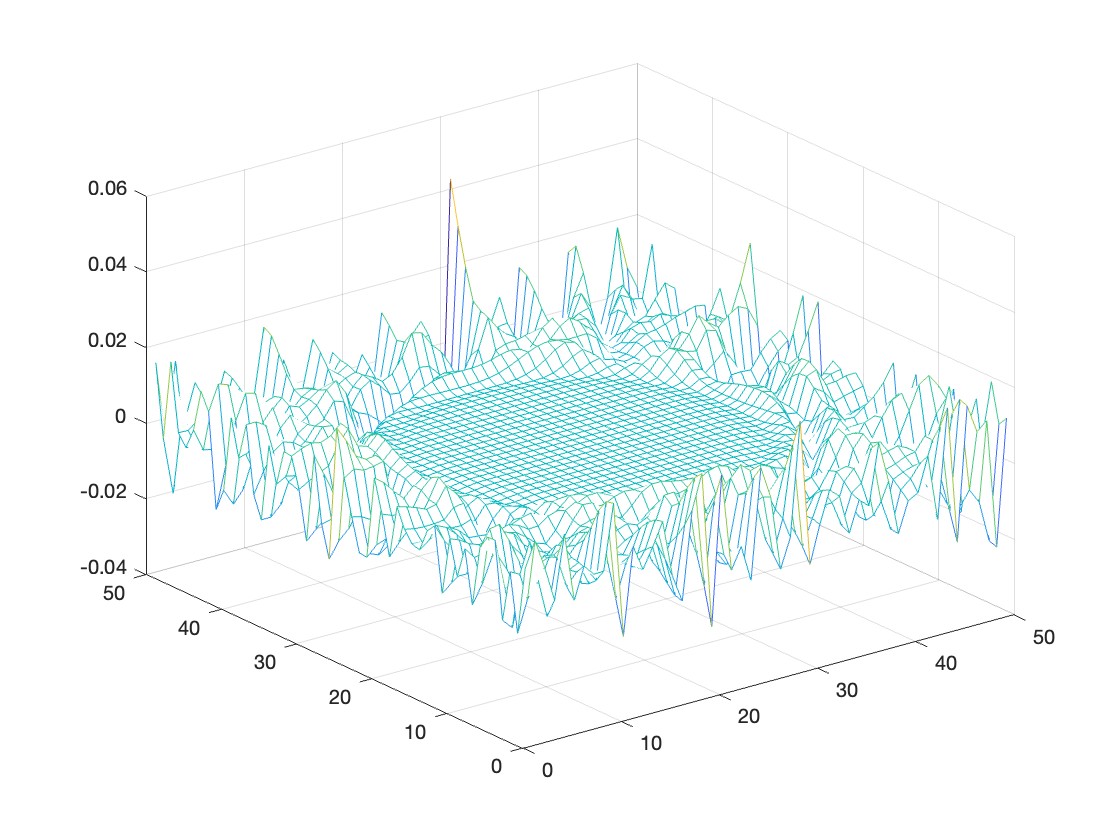}
        \caption{Error term $u_{\text{notp}}-u^{\text{ext}}$ at fixed time $t_0=\pi$}
    \end{subfigure}
    \hfill
    \begin{subfigure}[b]{0.45\textwidth}
        \includegraphics[width=1.25\textwidth]{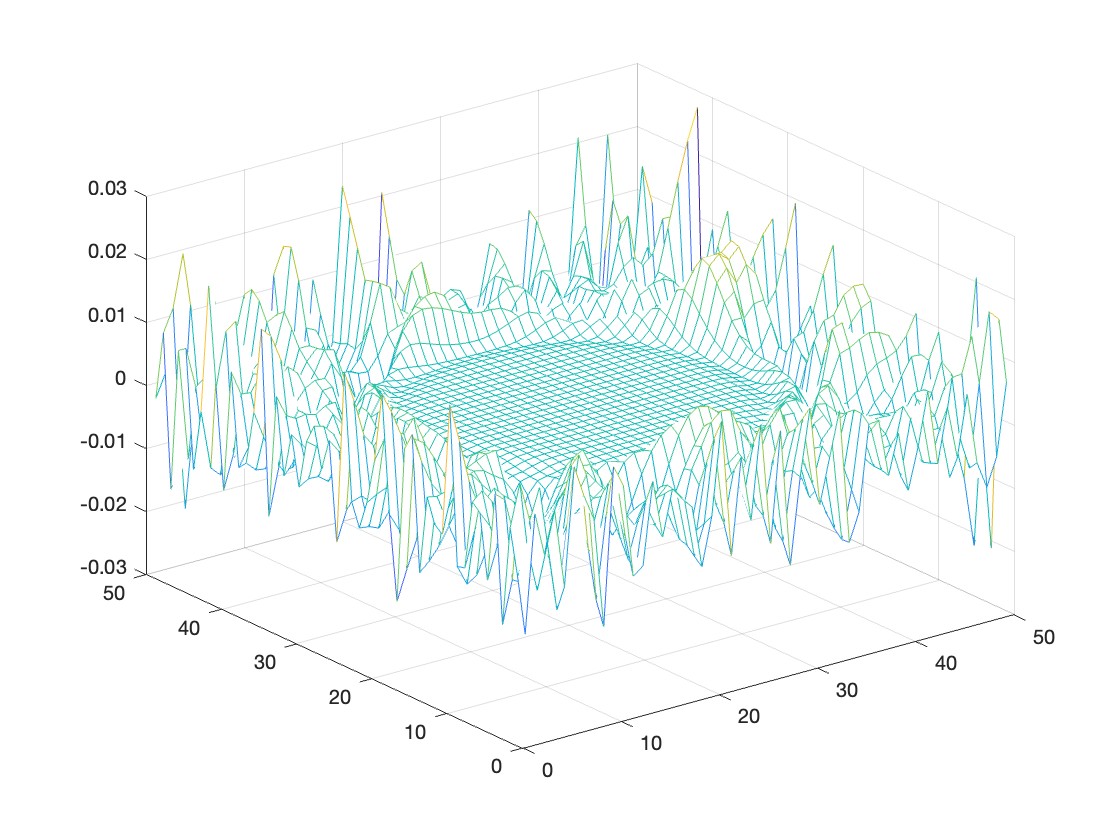}
        \caption{Error term $u_{\text{p}}-u^{\text{ext}}$ at fixed time $t_0 =\pi$}
    \end{subfigure}

    \caption{Empirical spatial distribution of error term for ring density function as in Section 6.1.}
    \label{fig: error distribution}
\end{figure}

\section{Data-driven deep neural network Fokker-Planck solver}
\subsection{Loss Function}
In the periodic case, we use the penalty method with penalty parameter 1 to convert the optimization problem in Equation (\ref{Eqn:optimization}) into the following unconstrained optimization problem
\begin{equation} \label{Eqn:new optimization}
\min _{\boldsymbol{u}}\|\boldsymbol{A} \boldsymbol{u}\|_2^2 + \|\boldsymbol{u}-\boldsymbol{v}\|_2^2,
\end{equation}
where $\boldsymbol{A}$ and $\boldsymbol{v}$ are the same as in Equation (\ref{Eqn:optimization}). As shown in \cite{zhai2022deep}, the new optimization problem (\ref{Eqn:new optimization}) has a similar effect to the original one in (\ref{Eqn:optimization}). %% error analysis %%

Next, we consider a mesh-free version of the optimization problem (\ref{Eqn:new optimization}), in which the variable $u$ is represented by an artificial neural network approximation. Unlike the finite difference method, the periodic boundary condition cannot be naturally incorporated into a standard artificial neural network. Thus, we must enforce the periodic boundary condition by adding an additional loss term. Let $\tilde{u}(\boldsymbol{x}, \boldsymbol{\theta})$ be an approximation of $\boldsymbol{u}$ represented by an artificial neural network with parameters $\boldsymbol{\theta}$. Inspired by Equation (\ref{Eqn:new optimization}), we define the following error loss function, with a slight modification of the norm used in the periodic penalty term.

\begin{equation} \label{Eqn:loss function}
    \begin{aligned}
        L(\boldsymbol{\theta}) &= \frac{1}{N^X} \sum_{i=1}^{N^X}\left(\mathcal{L} \tilde{u}\left(\boldsymbol{x}_i, \boldsymbol{\theta}\right)\right)^2
        + \frac{1}{N^Y} \sum_{j=1}^{N^Y}\left(\tilde{u}\left(\boldsymbol{y}_j, \boldsymbol{\theta}\right)-v\left(\boldsymbol{y}_j\right)\right)^2 \\
        &\quad + \frac{1}{N^Z} \sum_{k=1}^{N^Z} \left| \tilde{u}\left(\boldsymbol{z}_k, t_1, \boldsymbol{\theta} \right) - \tilde{u}\left(\boldsymbol{z}_k, t_1+T, \boldsymbol{\theta} \right) \right| \\
        &:= L_1(\boldsymbol{\theta}) + L_2(\boldsymbol{\theta}) + L_3(\boldsymbol{\theta}),
    \end{aligned}
\end{equation}
with respect to $\boldsymbol{\theta}$, where $\boldsymbol{x}_i \in \mathbb{R}^n \times \mathbb{R}, i=1,2, \ldots, N^X$ and $\boldsymbol{y}_j \in \mathbb{R}^n \times \mathbb{R}, j=1,2, \ldots, N^Y$ are collocation points sampled from numerical domain $D$, and $v\left(\boldsymbol{y}_j\right)$ is the Monte Carlo approximation from a numerical simulation of (\ref{Eqn:SDE}) at $\boldsymbol{y}_j$. Similarly, $\boldsymbol{z}_k \in \mathbb{R}^n, k=1,2, \ldots, N^Z$, denotes the spatial collocation points used to enforce the periodic boundary condition. These points are paired with the temporal values $t_1$ and $t_1+T$, and are sampled from the spatial projection of the domain $D$. 

Let $\mathfrak{X}:=\left\{\boldsymbol{x}_i ; i=1,2, \ldots, N^X\right\}$, $\mathfrak{Y}:=\left\{\boldsymbol{y}_j ; j=1,2, \ldots, N^Y\right\}$ and $\mathfrak{Z}:=\left\{\left(\boldsymbol{z}_k, t\right) \in \mathbb{R}^n \times \mathbb{R}; \boldsymbol{z}_k \in \mathbb{R}^n, t \in\left\{t_1, t_1+T\right\}, k=1,2, \ldots, N^Z\right\}$ be three training sets consisting of collocation points. To distinguish them, we refer to $\mathfrak{X}$ as the "training set", $\mathfrak{Y}$ as the "reference set" and $\mathfrak{Z}$ as the "boundary condition training set".

This loss function (\ref{Eqn:loss function}) is in fact the Monte Carlo integration of the following functional
\begin{equation}
    J(u)=\|\mathcal{L} u\|_{L^2(D)}^2 + \|u-v\|_{L^2(D)}^2 + \|u(t)-u(t+T)\|_{L^1(D)},
\end{equation}
which can be seen as the continuous version of the discrete optimization problem (\ref{Eqn:new optimization}).

The loss function (\ref{Eqn:loss function}) has three parts. Our goal is to obtain the minimizer parameter $\boldsymbol{\theta}^{*}$ that minimizes the loss function. For the first part of the loss function, the minimization of $L_1(\boldsymbol{\theta})$ helps generate the parameters $\boldsymbol{\theta}^{*}$ that guide the neural network representation $\tilde{u}(\boldsymbol{x}, \boldsymbol{\theta}^*)$ to fit the Fokker-Planck differential equation $\mathcal{L} \tilde{u}=0$ empirically at the training points $\boldsymbol{x}_i, i=1,2, \ldots, N^X$. The second term $L_2(\boldsymbol{\theta})$ of the loss function serves as a reference for the solution. It is the low-accuracy Monte Carlo approximation that guides the neural network training process to converge to the desired solution, namely the one satisfying the periodic Fokker-Planck equation (\ref{Eqn:density function equation}), as introduced in Section 3. As for the third term of the loss function, minimizing $L_3(\boldsymbol{\theta})$ enforces the temporal periodic boundary condition on $\tilde{u}(\boldsymbol{x},\boldsymbol{\theta}^*)$ and yields the optimal parameter $\boldsymbol{\theta}^*$. We emphasize that the periodic error term $L_3(\boldsymbol{\theta})$ is defined using the $\ell_1$-norm rather than the squared $\ell_2$-norm. This choice reflects the importance of periodicity in the problem and provides stronger and sharper penalization than a squared error term. In practice, this choice yields better enforcement of the temporal periodic boundary condition and a more accurate approximation of periodic steady-state behavior.

The neural network approximation minimizes the residual of the differential operator over collocation points, learns the probability density function from the reference data points, and enforces the temporal periodicity through the boundary condition points. As discussed in \cite{li2020numerical}, we use the "Alternating Adam" method (described in \cite{Kingma2014AdamAM}) to train the neural network, which runs the Adam optimizer to minimize loss functions $L_1$, $L_2$ and $L_3$ and updates the parameter $\boldsymbol{\theta}$ separately at each training step. This is because the magnitudes of the three loss functions can differ significantly, whereas the Adam optimizer is relatively scaling free. This method avoids the trouble of rebalancing the weight of $L_1$, $L_2$ and $L_3$ during the neural network training. 

%The following is the detailed implementation of the "double shuffling" method.

%To further accelerate the training process, we use the "double shuffling" method which samples a mini-batch in $\mathfrak{X}$, $\mathfrak{Y}$ and $\mathfrak{Z}$ independently in each iteration to update the parameter.

\begin{algorithm}[h]
\caption{Neural network training}
\KwIn{Training sets $ \mathfrak{X} $, $ \mathfrak{Y} $, and $ \mathfrak{Z} $}
\KwOut{Minimizer $ \boldsymbol{\theta}^* $ and $ \tilde{u}(\boldsymbol{x}, \boldsymbol{\theta}^*) $}
1: Initialize a neural network representation $ \tilde{u}(\boldsymbol{x}, \boldsymbol{\theta}) $ with undetermined parameters $ \boldsymbol{\theta} $\;
2: Run Monte Carlo simulation to get an approximate density $ v(\boldsymbol{y}_j) $ at each reference data point $ \boldsymbol{y}_j $, $ j = 1, 2, \dots, N^Y $\;
3: Pick a mini-batch in $ \mathfrak{X} $, compute the mean gradient of $ L_1 $, and use the mean gradient to update $ \boldsymbol{\theta} $\;
4: Pick a mini-batch in $ \mathfrak{Y} $, compute the mean gradient of $ L_2 $, and use the mean gradient to update $ \boldsymbol{\theta} $\;
5: Pick a mini-batch in $ \mathfrak{Z} $, compute the mean gradient of $ L_3 $, and use the mean gradient to update $ \boldsymbol{\theta} $\;
6: Repeat steps 3–5 until the losses of $ L_1 $, $ L_2 $, and $ L_3 $ are all sufficiently small\;
7: Return $ \boldsymbol{\theta}^* $ and $ \tilde{u}(\boldsymbol{x}, \boldsymbol{\theta}^*) $\;
\end{algorithm}

\subsection{Sampling collocation points and reference data}
For many stochastic dynamical systems, the invariant probability measure is concentrated near some small regions or low-dimensional manifolds, while the probability density function is close to zero farther away. Hence, samples of the collocation points in $\mathfrak{X}$ and $\mathfrak{Y}$ must effectively represent the concentration of the invariant probability density function. 

We adopt the method in \cite{zhai2022deep} to generate the collocation points and reference data. The idea is to capture both the high-density regions of the steady-state solution and the broader structure across the full domain \( D = [a_0, b_0] \times [a_1, b_1] \times [t_1, t_2] \).

First, we run a long numerical stochastic trajectory governed by the stochastic differential equation (\ref{Eqn:SDE}). To avoid bias from transient dynamics, we introduce an \emph{initial burn-in time} $s_0$, allowing the system to evolve for a sufficiently long time without recording any samples, so that it reaches a statistically representative regime before data collection begins. Correspondingly, we define the number of burn-in steps as $N_0 = \frac{s_0}{\Delta t}$, where \( \Delta t \) is the discrete time step used in the numerical simulation.

After the burn-in phase, we generate each collocation point $\boldsymbol{x}_j$ and $\boldsymbol{y}_j$ independently according to the following two-part strategy:

\begin{itemize}
    \item With probability \( \alpha \in (0,1) \), a point is sampled along the simulated trajectory. To introduce additional randomness and reduce correlations between successive samples, we evolve the system forward for a random amount of time before recording each point. Specifically, we draw a random evolution time \( T \sim \text{Uniform}(0, T_{\max}) \), where \( T_{\max} > 0 \) is a prescribed maximum evolution time. The system is then evolved forward over time \( T \). The resulting spatial-temporal position \( (x, y, t) \) is then recorded.
    
    \item With probability \( 1-\alpha \), a point is sampled \emph{uniformly} from domain \( D \) so that the network can learn small values from it.
\end{itemize}

Since the concentration part  preserves more information of the invariant distribution density, we usually set $\alpha = 0.5 \sim 0.9$. See Algorithm 2 for full details.\\

\begin{algorithm}[H]
\caption{Data collocation sampling}
\KwIn{Rate $\alpha \in [0.5, 0.9]$.}
\KwOut{Training collocation points $\boldsymbol{x}_i$, $i=1,2,\dots,N^X$ (or $\boldsymbol{y}_j$, $j=1,2,\dots,N^Y$).}

Initialize $\boldsymbol{X}_0$\;
Run a numerical trajectory of (\ref{Eqn:SDE}) until time $s_0$ (or after steps $N_0$) to ``burn in''\;
Choose a maximum evolution time $T_{max}$\;
\For{$i=1$ to $N^X$}{
    Generate a random number $c_i \sim U([0,1])$.\;
    \eIf{$c_i \leq \alpha$}{
        Generate a random evolution time $T_i \sim U(0,T_{max})$\;
        Run the numerical trajectory of (\ref{Eqn:SDE}) up to time $T_i$\;
        Let $\boldsymbol{x}_i = \boldsymbol{X}_{T_i}$\;
    }{
        Generate a random point $\boldsymbol{x}_i \sim U(D)$\;
    }
}
Return $\boldsymbol{x}_i$, $i=1,2,\dots,N^X$\;
\end{algorithm}
\vspace{1em}

As for set $\mathfrak{Z}:=\left\{\left(\boldsymbol{z}_k, t\right) \in \mathbb{R}^n \times \mathbb{R}; \boldsymbol{z}_k \in \mathbb{R}^n, t \in\left\{t_1, t_1+T\right\}, k=1,2, \ldots, N^Z\right\}$, to generate the spatial points $\boldsymbol{z}_k$, we uniformly sample from the spatial domain $[a_0, b_0] \times [a_1, b_1]$ within $D$. Specifically, for each $k = 1, \dotsc, N^Z$, we draw independent random variables for each spatial coordinate from a uniform distribution over the corresponding interval. The associated temporal coordinates are fixed at either $t_1$ or $t_1 + T$ depending on the enforcement of the periodic constraint.

\section{Estimation of the convergence speed}
In addition to numerically computing the periodic probability measure, we are also interested in estimating the speed of convergence of the law of $\boldsymbol{X}_t$, which satisfies the SDE (\ref{Eqn:SDE}), toward the periodic invariant distribution. In this chapter, we apply the coupling method to numerically compute the geometric ergodicity, i.e., the geometric convergence rate, especially in higher-dimensional spaces.

The main idea of the coupling method (also called maximal coupling in \cite{li2020numerical}) is to use the exponential tail of the coupling time distributions to numerically estimate the geometric convergence rate of a stochastic process.

\subsection{Geometric ergodicity and coupling of Markov processes}

Throughout this chapter, we consider the coupling of two periodic Markov processes, $\boldsymbol{X}_t$ and $\boldsymbol{Y}_t$
satisfying the same SDE (\ref{Eqn:SDE}), with the common transition probability ${P^{t}}$ and the unique periodic invariant probability distribution $\mu_t$. Then, similar to the coupling of measures (Definition \ref{def: coupling of measure}), the coupling of processes $(\boldsymbol{X}_t, \boldsymbol{Y}_t)$ also satisfies a coupling inequality.

\begin{theorem} \textbf{(Coupling Inequality for Markov Process)} \label{thm: coupling inequality for Markov process}
    Let $\nu_1$ and $\nu_2$ be the initial distributions of $\boldsymbol{X}_t$ and $\boldsymbol{Y}_t$, respectively. Then, for the coupling $(\boldsymbol{X}_t, \boldsymbol{Y}_t)$, we have the following \textit{coupling inequality}
\begin{equation} \label{eqn:coupling inequality for Markov process}
    \left\|\nu_1 P^t-\nu_2 P^t\right\|_{T V} \leq 2 \mathbb{P}\left[\tau_c>t\right].
\end{equation}
where total variance $\|\nu_1-\nu_2\|_{T V}=2 \sup _{A \in \mathcal{B}}|\nu_1(A)-\nu_2(A)|$.
\end{theorem}

The coupling inequality (\ref{eqn:coupling inequality for Markov process}) can be easily proved from the definition of \textbf{coupling time} $\tau_c:=\inf _{t \geq 0}\left\{\boldsymbol{X}_t=\boldsymbol{Y}_t\right\}$ and the coupling inequality for probability measures (\ref{eqn: inequality of measure}). A coupling $(\boldsymbol{X}_t, \boldsymbol{Y}_t)$ is said to be the \textit{optimal coupling} if the equality in (\ref{eqn:coupling inequality for Markov process}) holds for any $t > 0$. 

We can numerically estimate the rate of geometric ergodicity of $\boldsymbol{X}_t$ (or $\boldsymbol{Y}_t$) via Theorem (\ref{thm: coupling inequality for Markov process}). From (\ref{thm: coupling inequality for Markov process}) and Markov inequality, we have
\begin{equation}
    \label{eqn:tvnorm}
    \left\|\nu_1 P^t-\nu_2 P^t\right\|_{T V} \leq 2 \mathbb{P}\left[\tau_c>t\right] \leq 2 \mathbb{E}\left[e^{r_0 \tau_c}\right] e^{-r_0 t}.
\end{equation}

The finiteness of $\mathbb{E}\left[e^{r_0 \tau_c}\right]$ is sufficient to prove the geometric contraction/ergodicity. Note that since $\mu$ is the unique invariant probability measure of $X_t$, $X_t$ must be an irreducible Markov process with respect to $\mu$ in the sense of $\psi$-irreducibility. (We refer to Chapter 3 in \cite{li2020numerical} for details on the irreducibility of Markov processes.) Theoretical arguments (Lemma 2.2 - Lemma 2.4) in \cite{li2020numerical} ensure the finiteness of $\mathbb{E}\left[e^{r_0 \tau_c}\right]$ from one initial value can be extended to $\mu$-almost all initial values. Therefore, it is sufficient to show that there exists a pair of initial values $(x_0,y_0)$ satisfying $\mathbb{E}\left[e^{r_0 \tau_c}\right] < \infty$. However, the moment generating function $\mathbb{E}\left[e^{r_0 \tau_c}\right]$ is difficult to compute, especially when $r_0$ is close to the critical value $\sup \left\{r>0: \mathbb{E}\left[e^{r \tau_c}\right]<\infty\right\}$. The following lemma shows that instead of computing the moment generating function, we can turn to estimate the exponential tail of $\mathbb{P}\left[\tau_c>t\right]$.

\begin{lemma}
\label{lem52}
    For any initial distributions $\nu_1$ and $\nu_2$, assume that for $r_0>0$,
    \begin{equation} \label{eqn: exp tail}
        \limsup _{t \rightarrow \infty} \frac{1}{t} \log \mathbb{P}_{(\nu_1, \nu_2)}\left[\tau_c>t\right] \leq-r_0.
    \end{equation}
    Then for any $\varepsilon \in\left(0, r_0\right)$, it holds that
    \[
    \mathbb{E}_{(\nu_1, \nu_2)}\left[e^{\left(r_0-\varepsilon\right) \tau_c}\right]<\infty.
    \]
\end{lemma}
The proof of Lemma \ref{lem52} is identical to that of Lemma 2.5 in \cite{li2020numerical}.

In summary, the existence of an exponential tail from one initial pair implies $\mathbb{E}_{\nu_1, \nu_2}\left[e^{(r_0 - \varepsilon) \tau_c}\right] < \infty$. The irreducibility ensures that $\mathbb{E}\left[e^{r_0 \tau_c}\right] < \infty$ holds for $\mu$-almost all initial values. Then the geometric ergodicity follows from Equation \eqref{eqn:tvnorm}. This gives the following proposition.

\begin{proposition} \label{prop: exp tail and rate}
    Let $(\boldsymbol{X}_t, \boldsymbol{Y}_t)$ be a coupling of periodic processes $\boldsymbol{X}_t$ and $\boldsymbol{Y}_t$ satisfying the Fokker-Planck equation (\ref{Eqn:FPE}).
    \begin{itemize}
        \item[(i)] Assume that there exists an initial pair $\left(\boldsymbol{x}_0, \boldsymbol{y}_0\right) \in \mathbb{R}^n \times \mathbb{R}^n$ at initial time $s>0$ and $r_0>0$ such that
        \begin{equation} \label{eqn: exp tail_contracting}
            \limsup _{t \rightarrow \infty} \frac{1}{t} \log \mathbb{P}_{\left(\boldsymbol{x}_0, \boldsymbol{y}_0\right)}\left[\tau_c>t\right] \leq-r_0.
        \end{equation}
        Then for any $\varepsilon \in\left(0, r_0\right), \boldsymbol{X}_t$ (or $\boldsymbol{Y}_t$) is geometrically contracting with rate $\left(r_0-\varepsilon\right)$;
        \item[(ii)] Assume that there exists $\boldsymbol{x}_0 \in \mathbb{R}^n$ at initial time $s$ and $r_0>0$ such that
        \begin{equation} \label{eqn: exp tail_ergodic}
            \limsup _{t \rightarrow \infty} \frac{1}{t} \log \mathbb{P}_{\left(\boldsymbol{x}_0, \mu_0 \right)}\left[\tau_c>t\right] \leq-r_0.
        \end{equation}
        Then for any $\varepsilon \in\left(0, r_0\right), \boldsymbol{X}_t$ (or $\boldsymbol{Y}_t$) is geometrically ergodic with rate $\left(r_0-\varepsilon\right)$.
    \end{itemize}
\end{proposition}

\subsection{Description of algorithm}
The main idea of the algorithm is to use the exponential tail of coupling time distributions to numerically estimate the geometric ergodicity of the periodic process $\mathbf{X}_t$. Unlike time-homogeneous stochastic differential equation, the coupling time of time-periodic process $X_t$ usually gives 
\begin{equation} \label{eqn: tail analysis}
    \mathbb{P}_{\boldsymbol{x}_0, \boldsymbol{y}_0}\left[\tau_c>t\right] \approx C(t) e^{-r t}, \quad \forall t \gg 1 
\end{equation} 
instead of
$$
\mathbb{P}_{\boldsymbol{x}_0, \boldsymbol{y}_0}\left[\tau_c>t\right] \approx C(\boldsymbol{x}_0, \boldsymbol{y}_0) e^{-r t}, \quad \forall t \gg 1 \,.
$$
This can be justified by the following lemma.
\begin{lemma}
    Let $P_{s,t}(\boldsymbol{y}, \boldsymbol{x})$ be the time-periodic transition kernel starting from $\boldsymbol{y}$ at time $s$ and ending at $\boldsymbol{x}$ at time $t$, satisfying $P_{s+T,\,t+T}(\boldsymbol{y}, \boldsymbol{x})=P_{s,t}(\boldsymbol{y}, \boldsymbol{x})$. Assume the time-$T$ transition kernel $P_{0, T}(\boldsymbol{y}, \boldsymbol{x})$ has a spectral decomposition
    $$
        P_{0, T}(\boldsymbol{y}, \boldsymbol{x}) = \sum_{i = 0}^\infty p_i(\boldsymbol{x})q_i(\boldsymbol{y})\lambda_i 
    $$
    for a sequence of bi-orthonormal eigenfunction pairs $\{(p_i, q_i)\}$. The family $\{(p_i, q_i)\}$ satisfies $\int_{\mathbb{R}^n}p_i(\boldsymbol{z})\,q_j(\boldsymbol{z}) \mathrm{d}\boldsymbol{z} = \delta_{ij}$ and is complete in the sense that $\sum_{i=0}^{\infty}p_i(\boldsymbol{x})\,q_i(\boldsymbol{y})=\delta(\boldsymbol{x}-\boldsymbol{y})$. Then there exists a $T$-periodic kernel $R(\boldsymbol{x}, \boldsymbol{y}, t)$ satisfying
    \begin{equation}
        \label{floquet}
        P_{0, t}(\boldsymbol{y}, \boldsymbol{x}) = \sum_{i = 0}^\infty \left (\int_{\mathbb{R}^n} R(\boldsymbol{x}, \boldsymbol{z}, t) \, p_i( \boldsymbol{z}) \mathrm{d} \boldsymbol{z} \right )q_i(\boldsymbol{y})e^{\mu_i t} \,,
    \end{equation}
    where $\mu_i = \frac{1}{T}\ln \lambda_i$. 
\end{lemma}
\begin{proof}
For consistency, we denote $P_{s,t}(\boldsymbol{y}, \boldsymbol{x})$ as $P(\boldsymbol{x},t;\boldsymbol{y},s)$ in this proof.

First, we consider
\begin{equation}
\label{eqn: Q def}
    Q(\boldsymbol{x},t;\boldsymbol{y},s) := \sum_{i=0}^{\infty}p_i(\boldsymbol{x})\,q_i(\boldsymbol{y})\,e^{(t-s)\mu_i},
\end{equation}
where $\mu_i = \frac{1}{T}\ln \lambda_i$. 
% Then $Q$ has semi-group property and $Q(x,T;y,0) = P(x,T;y,0)$.
We first show that $Q$ has semi-group property, i.e. for any $t,s\geq0$,
\begin{equation}
    Q(\boldsymbol{x},t+s;\boldsymbol{y},0) = \int_{\mathbb{R}^n}Q(\boldsymbol{x},t;\boldsymbol{z},0)\,Q(\boldsymbol{z},s;\boldsymbol{y},0)\mathrm{d}\boldsymbol{z}.
\end{equation}
From the definition (\ref{eqn: Q def}) of $Q$, we have
\begin{equation*}
    \begin{aligned}
        \int_{\mathbb{R}^n}Q(\boldsymbol{x},t;\boldsymbol{z},0)\,Q(\boldsymbol{z},s;\boldsymbol{y},0)\mathrm{d}\boldsymbol{z} 
        &=\int_{\mathbb{R}^n}\left(\sum_{i=0}^{\infty}p_i(\boldsymbol{x})\,q_i(\boldsymbol{z})\,e^{t\mu_i}\right) \, \left(\sum_{j=0}^{\infty}p_j(\boldsymbol{z})\,q_j(\boldsymbol{y})\,e^{s\mu_j}\right)\mathrm{d}\boldsymbol{z}\\
        &=\int_{\mathbb{R}^n}\sum_{i,j=0}^{\infty}p_i(\boldsymbol{x})\,q_i(\boldsymbol{z})\,p_j(\boldsymbol{z})\,q_j(\boldsymbol{y})\,e^{t\mu_i+s\mu_j}\mathrm{d}\boldsymbol{z}\\
        &=\sum_{i,j=0}^{\infty}p_i(\boldsymbol{x})\,q_j(\boldsymbol{y})\,e^{t\mu_i+s\mu_j}\int_{\mathbb{R}^n} q_i(\boldsymbol{z})\,p_j(\boldsymbol{z}) \mathrm{d}\boldsymbol{z}\\
        &=\sum_{i=0}^{\infty}p_i(\boldsymbol{x})\,q_i(\boldsymbol{y})\,e^{(t+s)\mu_i}\\
        &=Q(\boldsymbol{x},t+s;\boldsymbol{y},0)\,.\\
    \end{aligned}
\end{equation*}
Also, at time $T$, we have
\begin{equation}
    Q(\boldsymbol{x},T;\boldsymbol{y},0) = P(\boldsymbol{x},T;\boldsymbol{y},0)\,.
\end{equation}
This can also be obtained from the definition,
\begin{equation*}
    \begin{aligned}
        Q(\boldsymbol{x},T;\boldsymbol{y},0) 
        &= \sum_{i=0}^{\infty}p_i(\boldsymbol{x})\,q_i(\boldsymbol{y})\,e^{T\mu_i}\\
        &= \sum_{i=0}^{\infty}p_i(\boldsymbol{x})\,q_i(\boldsymbol{y})\,e^{T(\frac{1}{T}\ln \lambda_i)}\\
        &= \sum_{i=0}^{\infty}p_i(\boldsymbol{x})\,q_i(\boldsymbol{y})\,\lambda_i\\
        &= P(\boldsymbol{x},T;\boldsymbol{y},0)\,.
    \end{aligned}
\end{equation*}

Second, let
\begin{equation}
\label{eqn: R def}
    R(\boldsymbol{x},\boldsymbol{y},t) := \int_{\mathbb{R}^n}P(\boldsymbol{x},t;\boldsymbol{z},0) \, Q(\boldsymbol{z},-t;\boldsymbol{y},0) \,\mathrm{d}\boldsymbol{z}, \quad t\geq 0\,.
\end{equation}
We claim that $R$ is $T$-periodic in $t$:
\begin{equation}
\label{eqn: R is T-periodic}
    R(\boldsymbol{x},\boldsymbol{y},t+T) = R(\boldsymbol{x},\boldsymbol{y},t).
\end{equation}
Using the Chapman-Kolmogorov equation, the $T$-periodic property of the transition kernel $P(\boldsymbol{x},t;\boldsymbol{y},0)$, and definitions (\ref{eqn: Q def}) and (\ref{eqn: R def}), we can prove the $T$-periodic property of $R$ as follows
\begin{equation}
\label{eqn: R periodic proof}
    \begin{aligned}
        R(\boldsymbol{x},\boldsymbol{y},t+T)
        &= \int_{\mathbb{R}^n} P(\boldsymbol{x},t+T;\boldsymbol{z},0) \, Q(\boldsymbol{z},-(t+T);\boldsymbol{y},0) \,\mathrm{d}\boldsymbol{z}\\
        &= \int_{\mathbb{R}^n} \left( \int_{\mathbb{R}^n} P(\boldsymbol{x},t+T;\boldsymbol{w},T) \, P(\boldsymbol{w},T;\boldsymbol{z},0) \,\mathrm{d}\boldsymbol{w} \right) \,Q(\boldsymbol{z},-(t+T);\boldsymbol{y},0) \,\mathrm{d}\boldsymbol{z}\\
        %&= \int_{\mathbb{R}^n} \int_{\mathbb{R}^n} P(x,t;w,0) \, P(w,T;z,0) \,\mathrm{d}w\, Q(z,-t-T;y,0) \,\mathrm{d}z\\
        &= \iint_{\mathbb{R}^n \times \mathbb{R}^n} P(\boldsymbol{x},t;\boldsymbol{w},0) \, P(\boldsymbol{w},T;\boldsymbol{z},0) \, Q(\boldsymbol{z},-t-T;\boldsymbol{y},0) \,\mathrm{d}\boldsymbol{w}\,\mathrm{d}\boldsymbol{z}\\
        &= \int_{\mathbb{R}^n} P(\boldsymbol{x},t;\boldsymbol{w},0)\left(\int_{\mathbb{R}^n}  P(\boldsymbol{w},T;\boldsymbol{z},0) \, Q(\boldsymbol{z},-t-T;\boldsymbol{y},0)\,\mathrm{d}\boldsymbol{z}\right) \,\mathrm{d}\boldsymbol{w}\,. \\
    \end{aligned}
\end{equation}
We consider the $\boldsymbol{z}$-integral in the right hand side of the last line in equation (\ref{eqn: R periodic proof}). From the properties and definition of $Q$,
\begin{equation*}
\begin{aligned}
    &\quad \;\;\int_{\mathbb{R}^n}  P(\boldsymbol{w},T;\boldsymbol{z},0) \, Q(\boldsymbol{z},-t-T;\boldsymbol{y},0)\,\mathrm{d}\boldsymbol{z}\\
    &= \int_{\mathbb{R}^n}  Q(\boldsymbol{w},T;\boldsymbol{z},0) \, Q(\boldsymbol{z},-t-T;\boldsymbol{y},0)\,\mathrm{d}\boldsymbol{z}\\
    &= \int_{\mathbb{R}^n} \left(\sum_{i=0}^{\infty} p_i(\boldsymbol{w})\,q_i(\boldsymbol{z})\,e^{T\mu_i}\right)
    \left( \sum_{j=0}^{\infty} p_j(\boldsymbol{z})\,q_j(\boldsymbol{y})\,e^{(-t-T)\mu_j}\right)\,\mathrm{d}\boldsymbol{z}\\
    &= Q(\boldsymbol{w},-t;\boldsymbol{y},0)\,.\\
\end{aligned}
\end{equation*}
Then, equation (\ref{eqn: R periodic proof}) becomes
\begin{equation*}
    R(\boldsymbol{x},\boldsymbol{y},t+T) = \int_{\mathbb{R}^n} P(\boldsymbol{x},t;\boldsymbol{w},0) \, Q(\boldsymbol{w},-t;\boldsymbol{y},0) \,\mathrm{d}\boldsymbol{w} = R(\boldsymbol{x},\boldsymbol{y},t)\,.
\end{equation*}

Next, we claim that for all $t \geq 0$,
\begin{equation}
\label{eqn: P=RQ}
    P(\boldsymbol{x},t;\boldsymbol{y},0) = \int_{\mathbb{R}^n} R(\boldsymbol{x},\boldsymbol{z},t) \, Q(\boldsymbol{z},t;\boldsymbol{y},0) \, \mathrm{d}\boldsymbol{z} \,.
\end{equation}
Substituting the definition (\ref{eqn: R def}) of $R(\boldsymbol{x},\boldsymbol{z},t)$ into equation (\ref{eqn: P=RQ}), we get the right hand side of equation (\ref{eqn: P=RQ})
\begin{equation}
\begin{aligned}
    \text{RHS}
    &= \int_{\mathbb{R}^n} \left( \int_{\mathbb{R}^n} P(\boldsymbol{x},t;\boldsymbol{w},0) \, Q(\boldsymbol{w},-t;\boldsymbol{z},0) \,\mathrm{d}\boldsymbol{w} \right)\, Q(\boldsymbol{z},t;\boldsymbol{y},0) \, \mathrm{d}\boldsymbol{z}\\
    &= \int_{\mathbb{R}^n} P(\boldsymbol{x},t;\boldsymbol{w},0)\left( \int_{\mathbb{R}^n} Q(\boldsymbol{w},-t;\boldsymbol{z},0)\, Q(\boldsymbol{z},t;\boldsymbol{y},0)\mathrm{d}\boldsymbol{z} \right) \, \mathrm{d}\boldsymbol{w}\\
    &= \int_{\mathbb{R}^n} P(\boldsymbol{x},t;\boldsymbol{w},0)\left( \sum_{i=0}^{\infty} p_i(\boldsymbol{w})\,q_i(\boldsymbol{y}) \right) \, \mathrm{d}\boldsymbol{w}\\
    &= \int_{\mathbb{R}^n} P(\boldsymbol{x},t;\boldsymbol{w},0) \, \delta(\boldsymbol{w}-\boldsymbol{y}) \,\mathrm{d}\boldsymbol{w}\\
    &= P(\boldsymbol{x},t;\boldsymbol{y},0) = \text{LHS}
\end{aligned}
\end{equation}
Finally, from the definition of $Q$, we have
\begin{equation*}
\begin{aligned}
    P(\boldsymbol{x},t;\boldsymbol{y},0) 
    &= \int_{\mathbb{R}^n} R(\boldsymbol{x},\boldsymbol{z},t) \, Q(\boldsymbol{z},t;\boldsymbol{y},0) \, \mathrm{d}\boldsymbol{z}\\
    &= \int_{\mathbb{R}^n} R(\boldsymbol{x},\boldsymbol{z},t) \, \left( \sum_{i=0}^{\infty}p_i(\boldsymbol{z})\,q_i(\boldsymbol{y})\,e^{\mu_i t}\right) \, \mathrm{d}\boldsymbol{z}\\
    &= \sum_{i = 0}^\infty \left (\int_{\mathbb{R}^n} R(\boldsymbol{x}, \boldsymbol{z}, t) \, p_i( \boldsymbol{z}) \mathrm{d} \boldsymbol{z} \right )q_i(\boldsymbol{y})e^{\mu_i t}\\
\end{aligned}
\end{equation*}

\end{proof}

Therefore, if the coupling is close to optimal, the coupling time distribution also has a $T$-periodic pre-factor. As a result, we need a different approach from \cite{li2020numerical} to estimate the coupling time distribution. When $t\rightarrow \infty$, it follows from Proposition (\ref{prop: exp tail and rate}) that for almost every pair of initial values $(\boldsymbol{x},\boldsymbol{y})$,
\[
\limsup _{t \rightarrow \infty} \frac{1}{t} \log \left(\left\|P^t(\boldsymbol{x}, \cdot)-P^t(\boldsymbol{y}, \cdot)\right\|_{T V}\right)<-r.
\]
The numerical verification of geometric ergodicity can also be carried out by replacing $\boldsymbol{y}$ with a sample from the periodic invariant distribution $\mu_t$.

\subsubsection{Estimation of $r$ and $C(t)$}
Taking logarithms on both sides of (\ref{eqn: tail analysis}) yields:
\[
\log \mathbb{P}\left[\tau_c>t\right] \approx \log C(t)-r t
\]
This expression motivates the decomposition of $\log \mathbb{P}\left[\tau_c > t\right]$ into a linear term and a bounded oscillating component. We describe below the methodology for estimating both the convergence rate $r$ and the periodic function $C(t)$.

First, to estimate the convergence rate $r$, we exploit the periodicity of $C(t)$ and consider multiple sequences of time samples spaced by the period $T$. Let $k$ be a fixed integer denoting the number of periods (typically $k = 5 \sim 8$). For each initial offset $t_i \in [0, T)$, $i=1,2,...,N$ for some large $N$, we construct a sequence:
\[
t_i, t_i+T, t_i+2 T, \ldots, t_i+k T .
\]
For each such time sequence, the values $\log C\left(t_i+j T\right)$ for $j=0, \ldots, k$ are all equal becasue of the periodic condition $C(t+T)=C(t)$. Therefore, we can regard the $\log C(t)$ term as a constant, and perform a linear regression for each such sequence $\left\{t_i+j T\right\}_{j=0}^k$ to obtain the local slope $-r_i$. After fitting all such sequences starting at different offsets $t_i \in[0, T)$, we average the resulting slopes to obtain the approximation value of $-r$.

Next, we estimate the oscillating term $\log C(t)$. With the convergence rate $r$ estimated, we define the residual
\[
\log C(t) = \log \mathbb{P}\left[\tau_c > t\right] + r t.
\]
This residual isolates the periodic component. To approximate $\log C(t)$, we perform a truncated Fourier expansion
\[
\log C(t) \approx a_0 + \sum_{j=1}^{k} \left( a_j \sin(j \omega t) + b_j \cos(j \omega t) \right) ,
\]
where $\omega = 2\pi/T$ is the fundamental frequency and $k$ is the number of periods as introduced before. The coefficients $a_j, b_j$, $j=1,2,...,k$ can be determined using least-squares regression. 

After obtaining the convergence rate $r$ and oscillating term $\log C(t)$, we compare the result with the $95\%$ confidence interval of $\log \mathbb{P}\left[\tau_c>t\right]$. Figure \ref{fig: tail_analysis} shows the estimation result for the coupling time distribution of the following SDE
\begin{equation}
\label{eqn: ring SDE}
\left\{\begin{array}{l}
d X_t=\left(-4X_t \,(X_t^2+Y_t^2-(1+0.5\sin t))\right)+Y_t \,\mathrm{d} t+\varepsilon_1 \, \mathrm{d} W_t^x \\
d Y_t=\left(-4Y_t \,(X_t^2+Y_t^2-(1+0.5\sin t))\right)-X_t \, \mathrm{d} t+\varepsilon_2 \, \mathrm{d} W_t^y
\end{array}\right. \, ,
\end{equation}
with $\varepsilon_1 = \varepsilon_2 = 1$. In Figure \ref{fig: tail_analysis}, we run Monte Carlo estimation with $10^6$ samples and time step size $h=0.001$, using the mix of independent coupling method and maximal coupling method to get the distribution of coupling time $\mathbb{P}\left[\tau_c>t\right]$. These coupling methods are introduced in the next subsection. The number of $2\pi$-periods is $k=6$. We construct $N=200$ time sequences $\{t_i, t_i+2\pi, \ldots,t_i+6\times2\pi\}$, $i=1,2,\ldots,N=200$, to get the estimation of $-r = \frac{1}{N}\sum_{i=1}^{N} (-r_i)$ and then the $2\pi$-periodic oscillating term $\log C(t)$ with $\omega = 1$ in the Fourier expansion. The purple line represents the approximated curve $\log C(t)-rt$. The estimated value of slope $-r = -0.21455$. The yellow line and red line are respectively the lower confidence bound and upper confidence bound in $95\%$ confidence level. 
%(Do we also need to explain the small tail at the beginning of $\log C(t)$?)

\begin{figure}[htbp]
    \centering
    \includegraphics[width=1.0\textwidth]{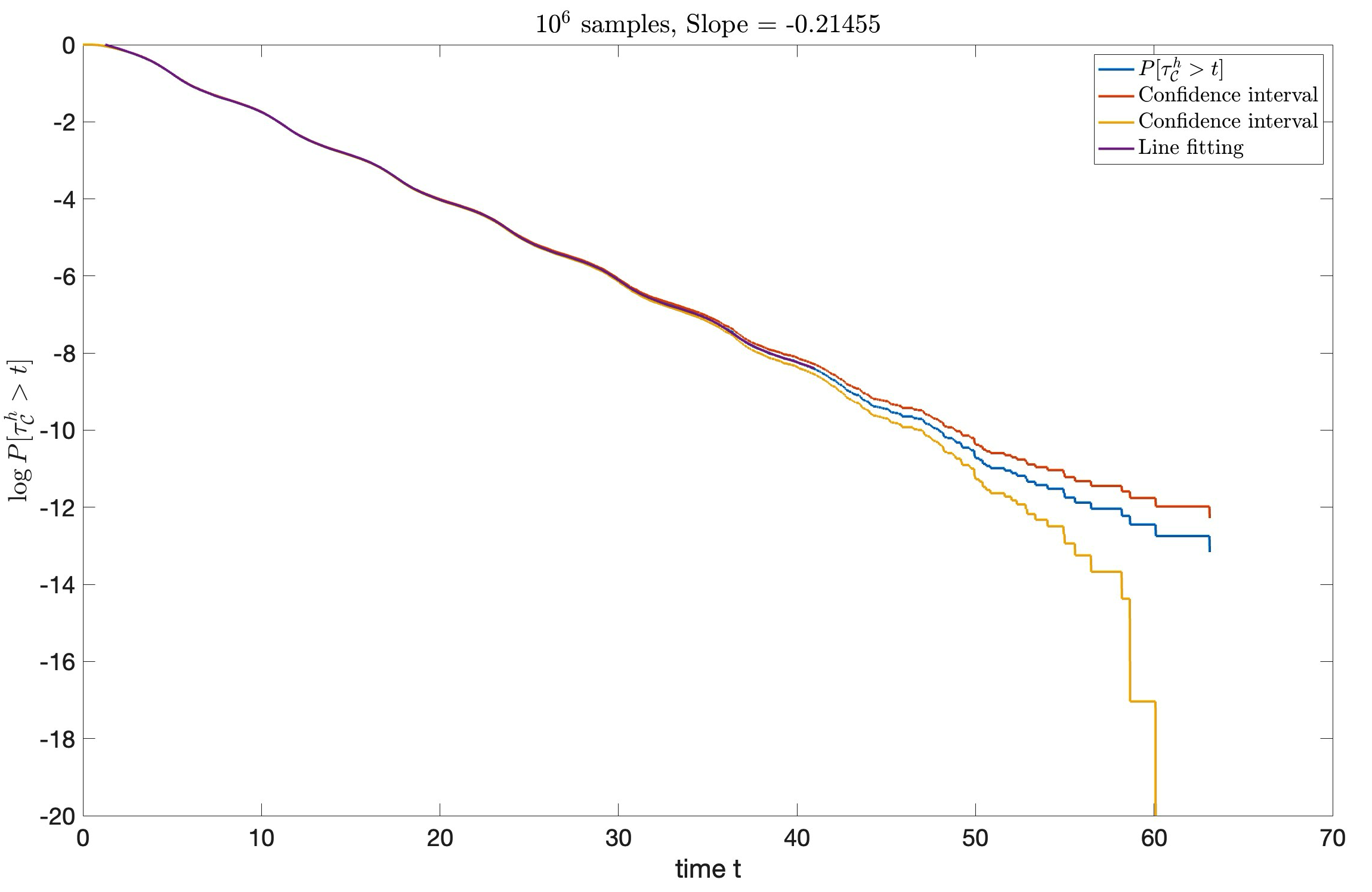}
    \caption{Curve approximation of $\log \mathbb{P}[\tau_c > t] \approx \log C(t) -rt$ for SDE in equation (\ref{eqn: ring SDE}) in $95\%$ confidence level with $10^6$ samples and time step size $h=0.001$.}
    \label{fig: tail_analysis}
\end{figure}

\subsubsection{Coupling Method}
It remains to discuss how the coupling time distribution $\mathbb{P}\left[\tau_c > t\right]$ is estimated. This is done by running two coupled numerical trajectories of Equation \eqref{Eqn:SDE} starting from $\boldsymbol{x}_0$ and $\boldsymbol{y}_0$ respectively. The numerical coupling strategies are introduced in \cite{li2020numerical}. We briefly review them here for the sake of completeness of the paper. Take the following two processes $\boldsymbol{X}_t$ and $\boldsymbol{Y}_t$ as an example:
\[
\left(\boldsymbol{X}_{n+1}, \boldsymbol{Y}_{n+1}\right)=\left(f\left(\boldsymbol{X}_n\right)+\zeta_n^1,\, f\left(\boldsymbol{Y}_n\right)+\zeta_n^2\right).
\]

At the next step, $\boldsymbol{X}_{n+1}$ and $\boldsymbol{Y}_{n+1}$ can be coupled in the following ways.
\begin{itemize}
    \item {\bf Independent coupling:} $\zeta_1$ and $\zeta_2$ are independent variables.
    \item {\bf Synchronous coupling:} $\zeta_1 = \zeta_2$
    \item {\bf Reflection coupling: }  $\zeta_1$ and $\zeta_2$ satisfy a mirror symmetry against the middle hyperplane orthogonal to the direction connecting $\boldsymbol{X}_t$ and $\boldsymbol{Y}_t$.
    \item {\bf Maximal coupling: } a joint distribution on $(\boldsymbol{X}_{n+1}, \boldsymbol{Y}_{n+1})$ is constructed such that
\[
\mathbb{P}[\boldsymbol{X}_{n+1} = \boldsymbol{Y}_{n+1}] = 1 - \frac{1}{2}\|\mu - \nu\|_{\text{TV}}.
\]
\end{itemize}

%See \cite{} for the detailed implementation of the maximal coupling. 

Each of these coupling methods has advantages depending on the nature of the drift and diffusion terms in the SDE. The general strategy is to use synchronous coupling or reflection coupling until two trajectories are sufficiently close, and then switch to the maximal coupling. For further details, including construction techniques and analytical properties, we refer the reader to Chapter 3 of \cite{li2020numerical}. We run each trajectory of the numerical coupling method until $\boldsymbol{X}_t = \boldsymbol{Y}_t$, which gives a sample of the coupling time $\tau_c$. A Monte Carlo simulation of many coupled trajectories until their coupling times gives the desired coupling time distribution $\mathbb{P}\left[\tau_c > t\right]$. 

We can see the logarithm of coupling time distribution $\mathbb{P}\left[\tau_c > t\right]$ of SDE (\ref{eqn: ring SDE}) in Figure \ref{fig: Monte Carlo simulation for coupling time distribution}. Three pictures in Figure \ref{fig: Monte Carlo simulation for coupling time distribution} are obtained by running Monte Carlo simulation using synchronous coupling first and then maximal coupling when two trajectories are sufficiently close, with time step size $h=0.001$ and respectively $10^2$ samples, $10^4$ samples and $10^6$ samples. The yellow line and red line are respectively the lower confidence bound and upper confidence
bound in $95\%$ confidence level.

\begin{figure}[htbp]
    \centering
    \begin{subfigure}[b]{0.32\textwidth}
        \includegraphics[width=\textwidth]{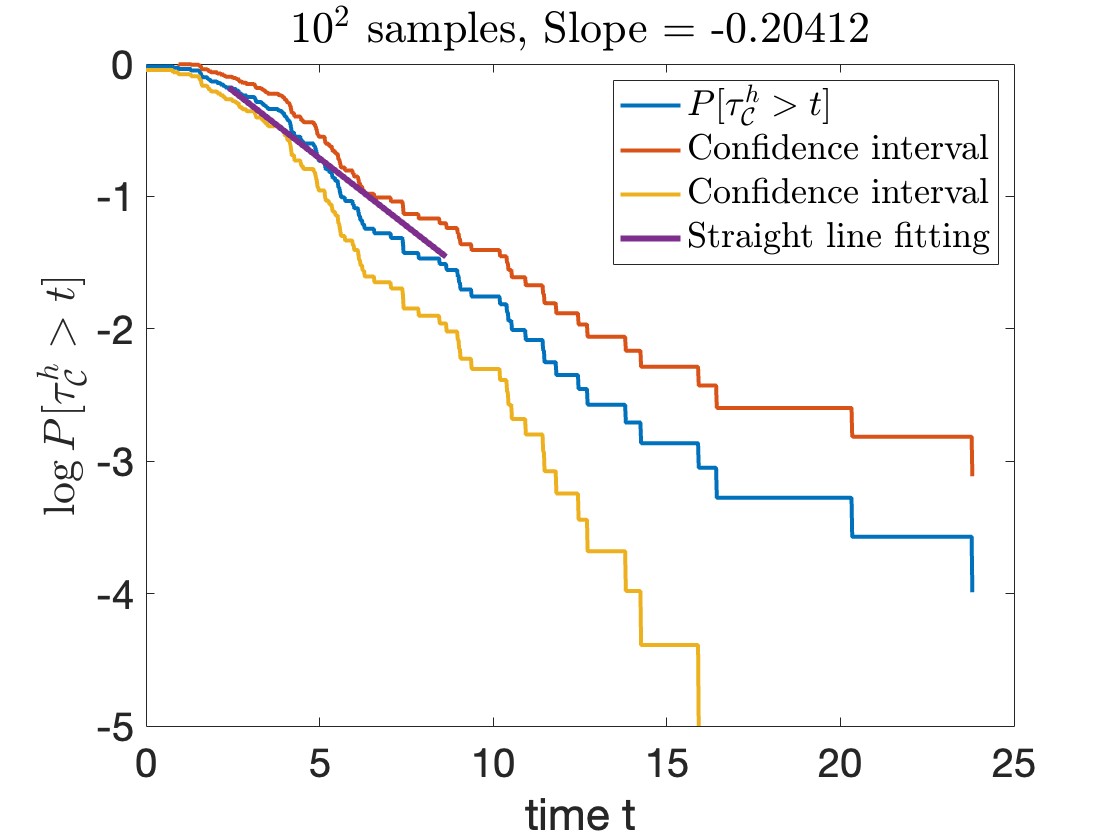}
        %\caption{Caption 1}
        %\label{fig:sub1}
    \end{subfigure}
    \hfill
    \begin{subfigure}[b]{0.32\textwidth}
        \includegraphics[width=\textwidth]{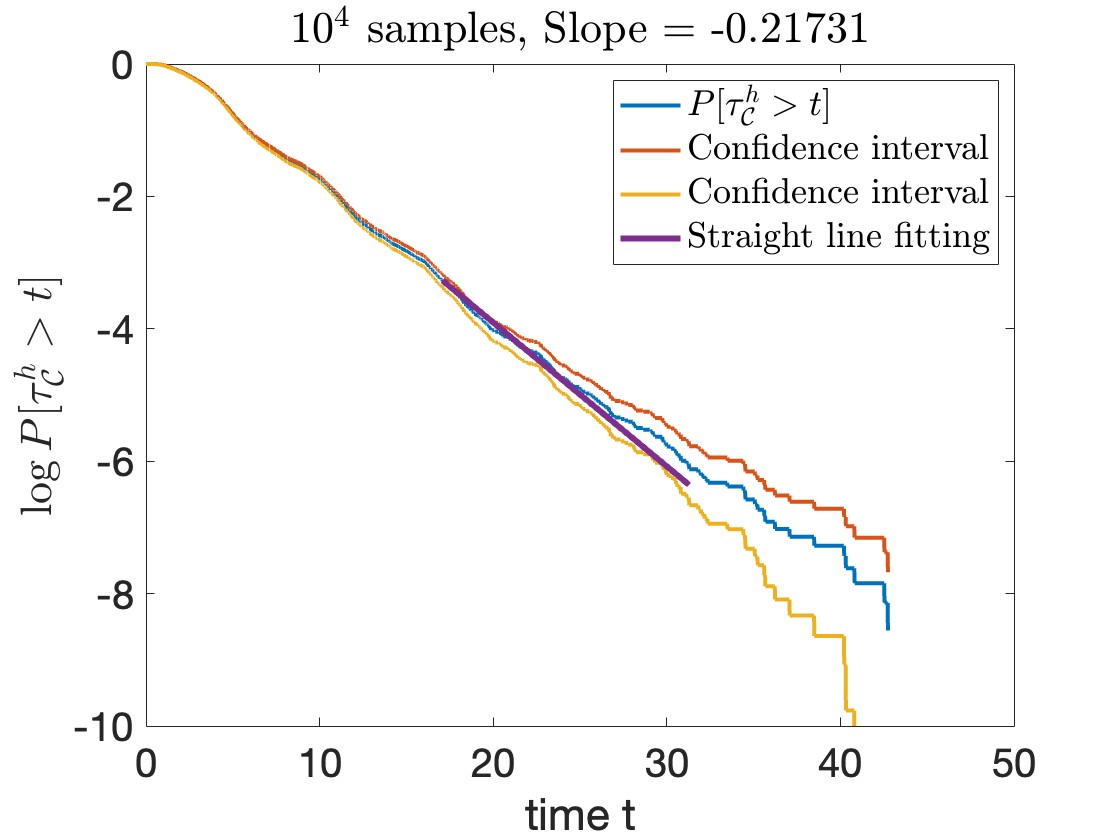}
        %\caption{Caption 2}
        %\label{fig:sub2}
    \end{subfigure}
    \hfill
    \begin{subfigure}[b]{0.32\textwidth}
        \includegraphics[width=\textwidth]{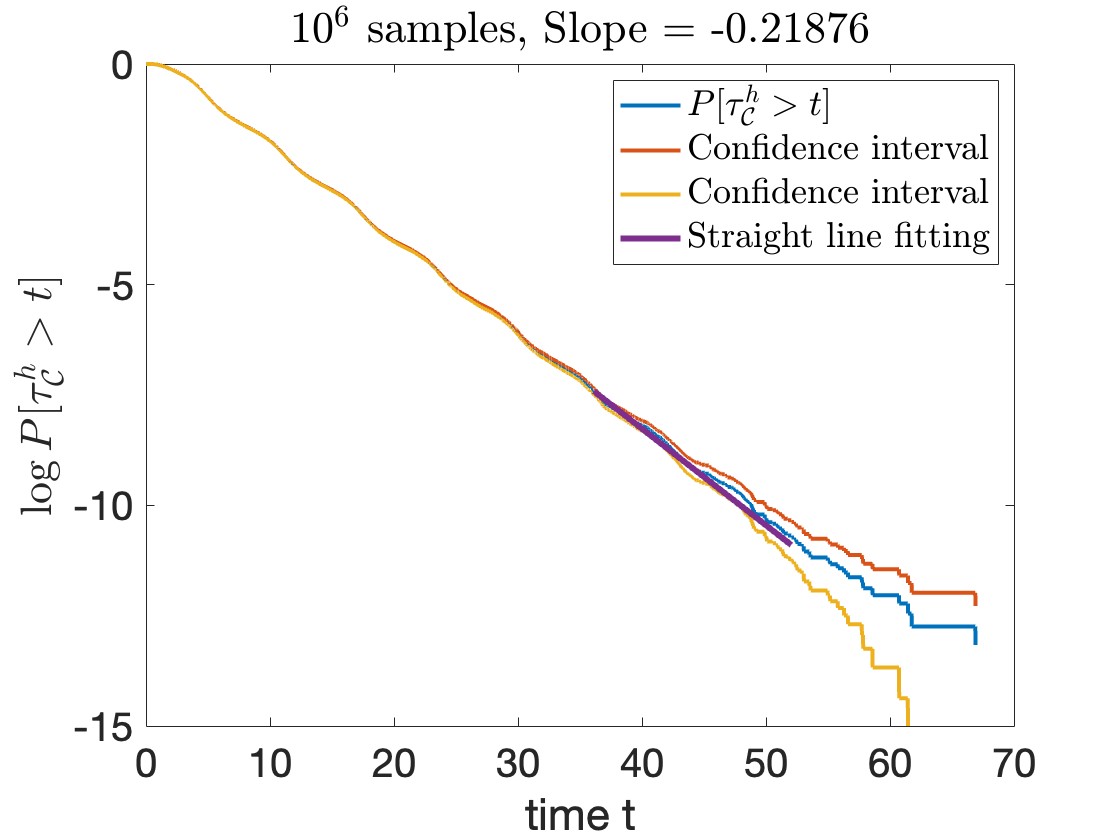}
        %\caption{Caption 3}
        %\label{fig:sub3}
    \end{subfigure}
    \caption{Logarithm of coupling time distribution $\mathbb{P}\left[\tau_c > t\right]$ of SDE (\ref{eqn: ring SDE}) in $95\%$ confidence interval. Left: Monte Carlo simulation with $10^2$ samples. Middle: Monte Carlo simulation with $10^4$ samples. Right: Monte Carlo simulation with $10^6$ samples.}
    \label{fig: Monte Carlo simulation for coupling time distribution}
\end{figure}

\section{Numerical examples}

\subsection{Toy example: Stuart–Landau oscillator}
Consider a two dimensional periodic stochastic system
\begin{equation} \label{eqn: example 6.1}
\left\{\begin{array}{l}
\mathrm{d} X_t = \left( -4 X_t\left(X_t^2+Y_t^2-(c+\sin t) \right)+ \frac{X_t \cos t}{2\left(X_t^2+Y_t^2\right)}\right) \mathrm{d} t+\varepsilon \mathrm{d} W_t^x, \\
\mathrm{d} Y_t = \left( -4 Y_t\left(X_t^2+Y_t^2-(c+\sin t) \right)+\frac{Y_t \cos t}{2\left(X_t^2+Y_t^2\right)} \right) \mathrm{d} t+\varepsilon \mathrm{d} W_t^y,
\end{array}\right.
\end{equation}
\noindent where $W_t^x$ and $W_t^y$ are independent Wiener processes, the diffusion coefficient $\varepsilon = \sqrt 2$, and $c$ is a constant whose value must be chosen carefully. The drift part of equation (\ref{eqn: example 6.1}) is a time-periodic gradient flow of the potential function
\[
V(x,y,t) = (x^2+y^2-(c+\sin t))^2
\]
Hence, the probability density function of the periodic invariant measure of (\ref{eqn: example 6.1}) can be explicitly written as
\begin{equation}
\label{61explicit}
u(x,y,t) = \frac{1}{Z(t)} e^{-V(x,y,t)},
\end{equation}
where $Z(t) = \int _{-\infty}^{+\infty} \int _{-\infty}^{+\infty} \exp\left( - (x^2 + y^2 - (c + \sin t))^2 \right) \mathrm{d}x \mathrm{d}y$  the time-dependent normalization factor.

Notice that being different from the "ring example" used in \cite{zhai2022deep}, equation \eqref{eqn: example 6.1} has a singularity near the origin, causing many theoretical and computational troubles. To prevent this, we choose a large $c$ to make the probability of hitting the small neighborhood of the origin extremely small. This means we are effectively solving a quasi-stationary distribution (QSD) in the domain without a very small neighborhood of the origin, with an extremely small killing rate. After some tests, we find that the probability of hitting the origin becomes negligible when $c = 5$. Since the primary role of this example is to compare the output of our algorithm to the explicit solution, we let $c = 5$ and compare our numerical solution with equation \eqref{61explicit} from now on. 

\begin{figure}[htbp]
    \centering
    \begin{subfigure}[b]{0.3\textwidth}
        \includegraphics[width=\textwidth]{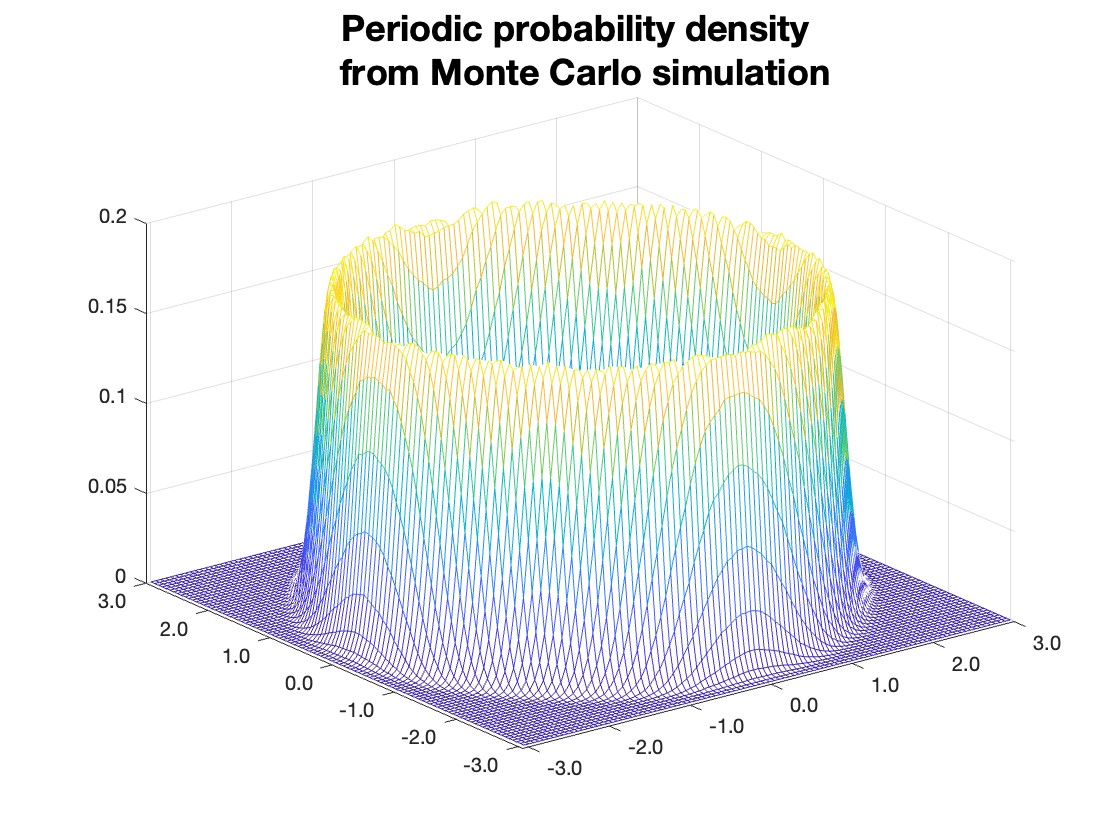}
        % \caption{Caption 1}
    \end{subfigure}
    \hfill
    \begin{subfigure}[b]{0.3\textwidth}
        \includegraphics[width=\textwidth]{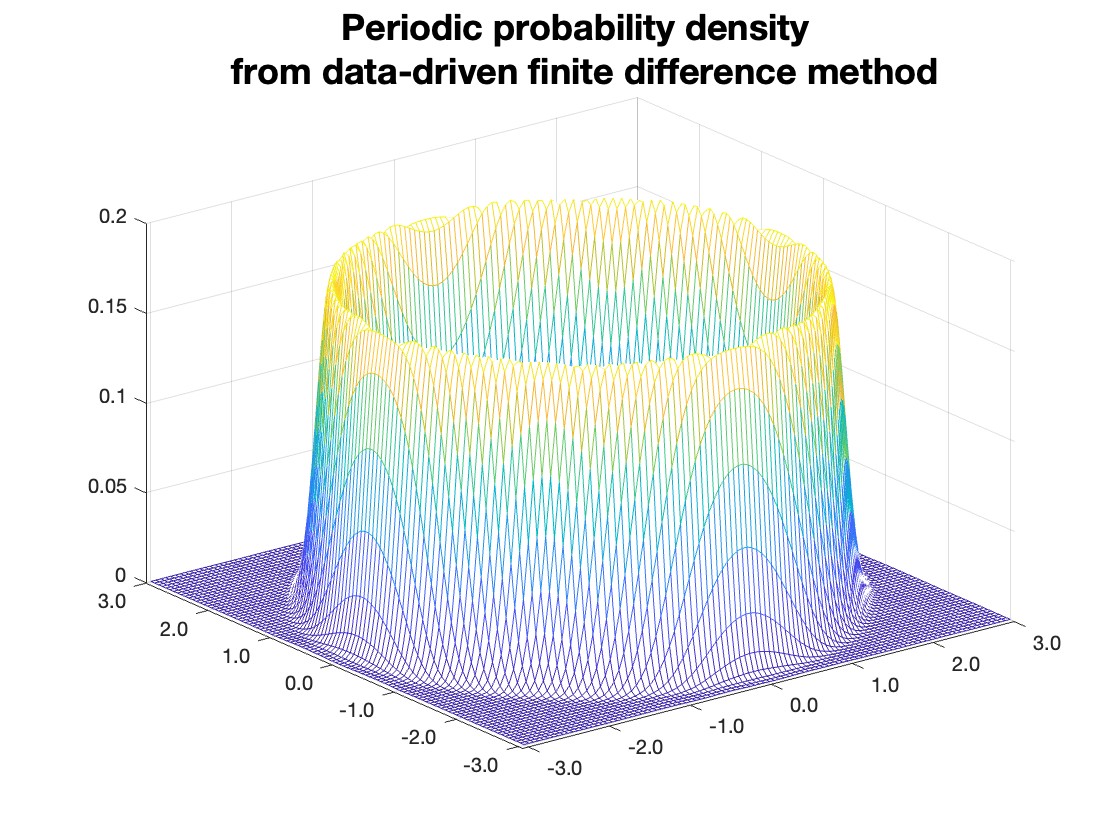}
        % \caption{Caption 2}
    \end{subfigure}
    \hfill
    \begin{subfigure}[b]{0.3\textwidth}
        \includegraphics[width=\textwidth]{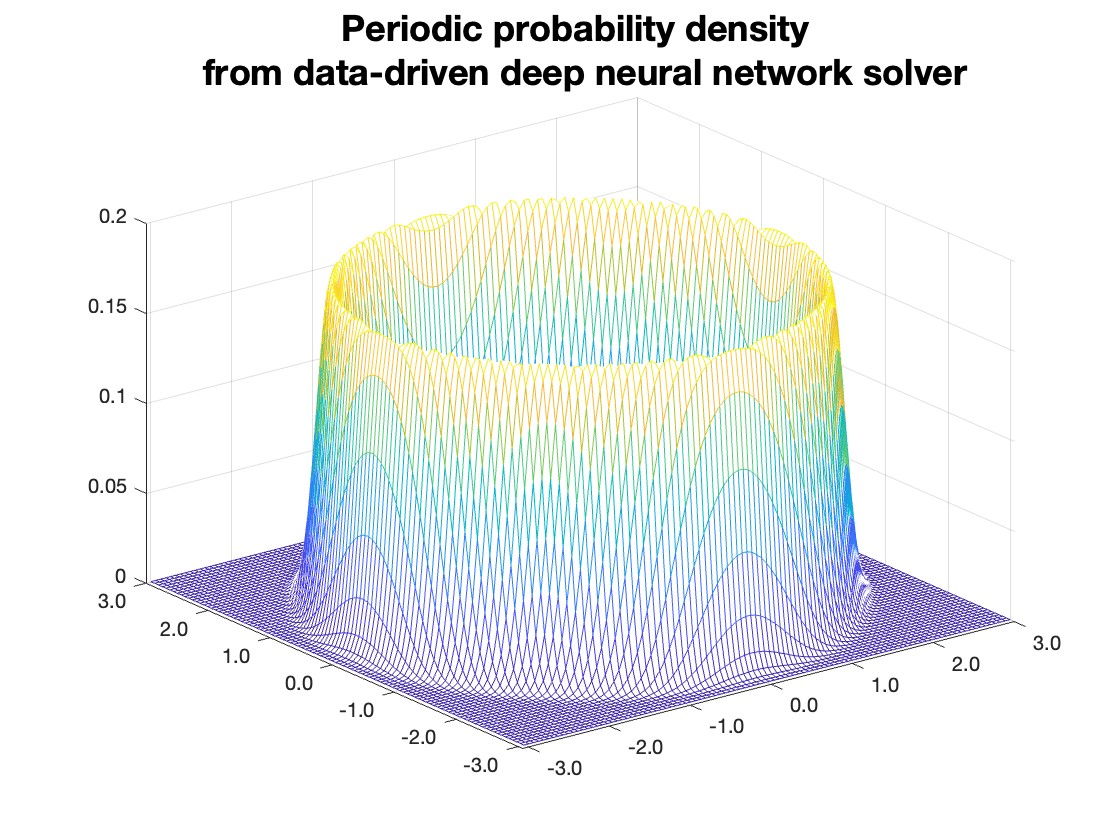}
        % \caption{Caption 3}
    \end{subfigure}
    
    \vspace{0.3cm}
    
    \begin{subfigure}[b]{0.3\textwidth}
        \includegraphics[width=\textwidth]{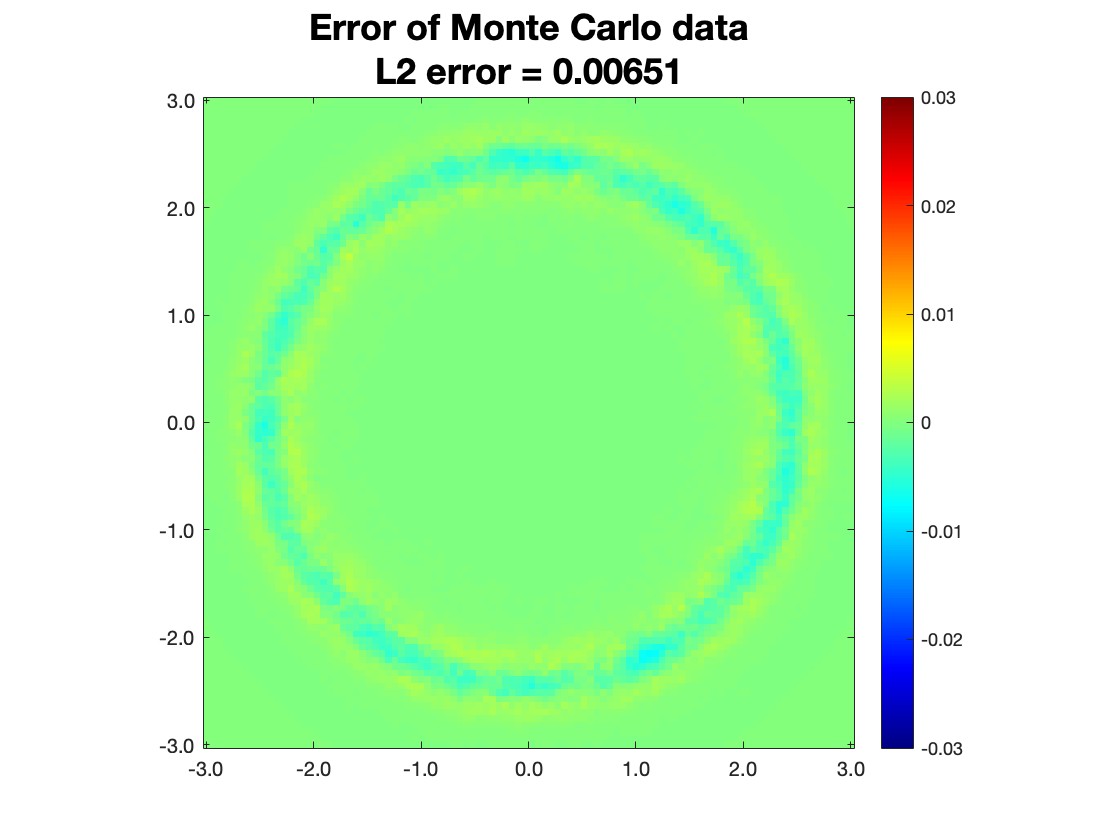}
        % \caption{Caption 4}
    \end{subfigure}
    \hfill
    \begin{subfigure}[b]{0.3\textwidth}
        \includegraphics[width=\textwidth]{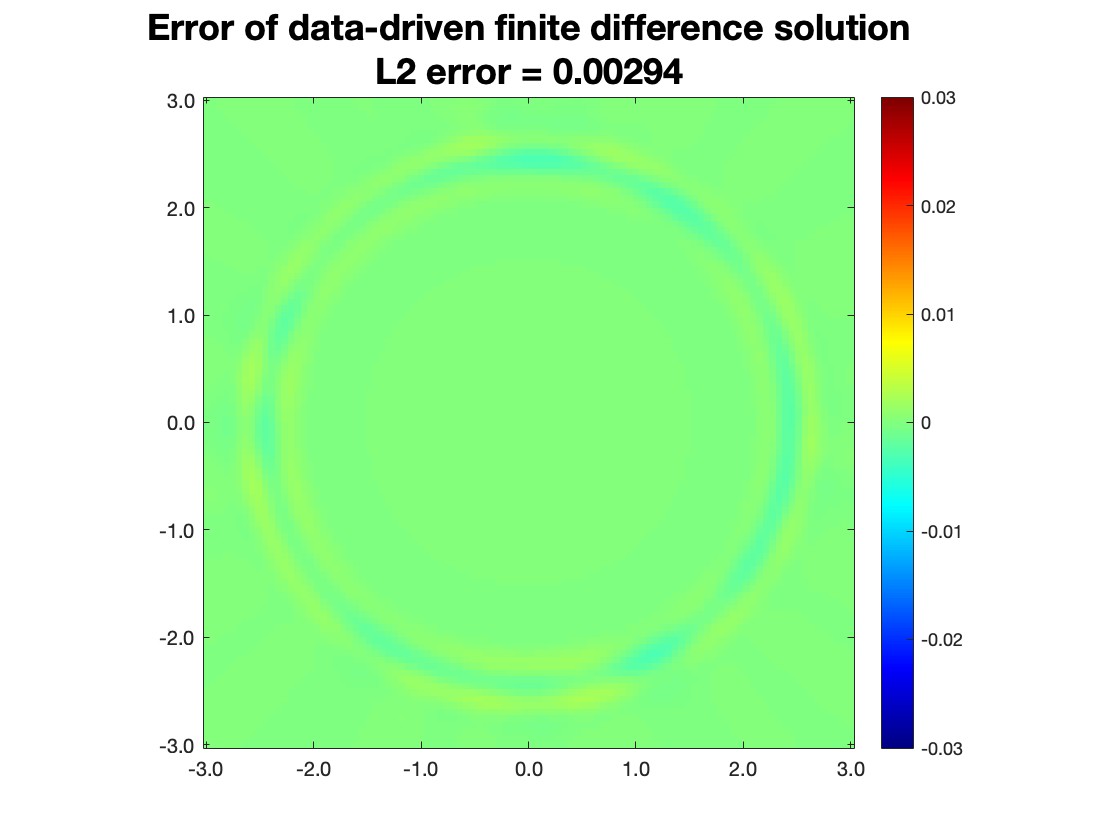}
        % \caption{Caption 5}
    \end{subfigure}
    \hfill
    \begin{subfigure}[b]{0.3\textwidth}
        \includegraphics[width=\textwidth]{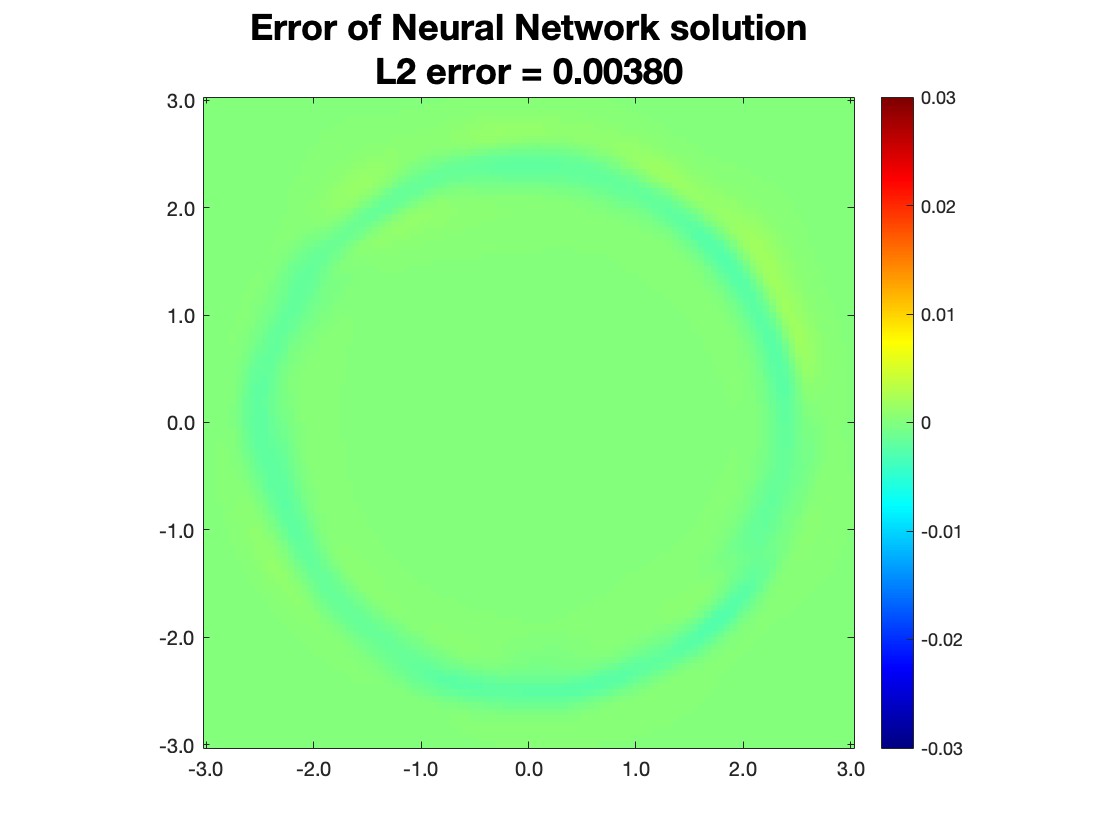}
        % \caption{Caption 6}
    \end{subfigure}
    
    \caption{A comparison of periodic probability density function of the periodic invariant measure obtained by Monte Carlo simulation, Data-driven finite difference method, and optimization through an artificial neural network at time $t=\frac{\,17\,}{\,25\,}\pi$. First row: periodic probability density functions. Second row: distribution of error against the exact solution. (Grid size = $100\times 100\times 100$. Sample size of Monte Carlo: $10^9$.)}
    \label{fig:Example 6.1 comparision of solutions}
\end{figure}

%% Neural network
%\begin{figure}[htbp]
%    \centering
%    \includegraphics[width=1.0\textwidth]{Example 6.1/gird size 50_50_50/Neural Network_t=35.png}
%    \caption{Neural network result.}
%    \label{fig: Example 6.1 Neural network}
%\end{figure}

%% convergence rate and curve fitting
\begin{figure}[htbp]
    \centering
    \includegraphics[width=1.0\textwidth]{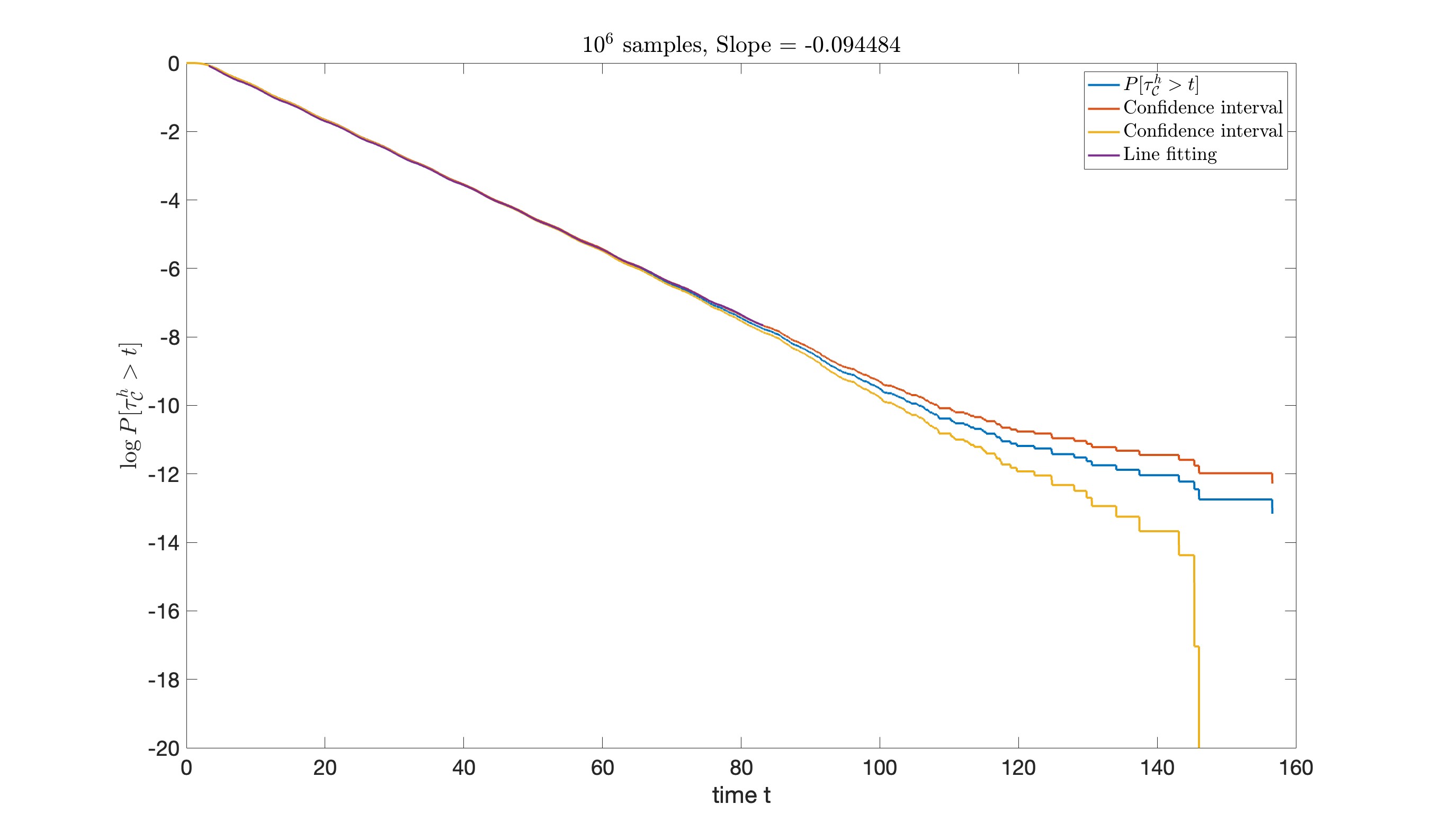}
    \caption{Approximation of $\log P(\tau_c > t)$ and the related confidence interval.}
    \label{fig: Example 6.1 convergence rate}
\end{figure}

In Figure \ref{fig:Example 6.1 comparision of solutions}, the first column gives a Monte Carlo approximation and its error distribution. The Monte Carlo solution is obtained by running an Euler-Maruyama numerical scheme with $10^9$ steps and calculating the empirical probability on a $100 \times 100 \times 100$ mesh of the origin $D = [-3,3]\times[-3,3]\times[0,2\pi]$. The time step size is $0.001$. In the second column of Figure \ref{fig:Example 6.1 comparision of solutions}, the optimization problems (\ref{Eqn:new optimization}) are solved by linear algebra solvers. The last column of Figure \ref{fig:Example 6.1 comparision of solutions} shows the neural network training with loss function (\ref{Eqn:loss function}) using all probability densities at grid points obtained by the same Monte Carlo simulation. The architecture of the artificial neural network is a small feed-forward neural network with 6 hidden layers, each of which contains $64$, $128$, $128$, $128$, $64$, $16$ neurons respectively, to approximate the solution to the Fokker-Planck equation of the example. The output layer has one neuron. The number of neurons in the input layer is 3 (for $x$, $y$ and $t$). The activation function is sigmoid function. The number of training points is $10000$. The original learning rate for loss functions $L_1$, $L_2$ and $L_3$ in (\ref{Eqn:loss function}) is $0.001$, $0.002$ and $0.002$ respectively, and the learning rate for $L_2$ and $L_3$ will be reduced by $0.9$ after every 40 epochs, while the learning rate for $L_1$ remains unchanged. There are 200 training epochs in total. The trained neural network is evaluated on a $50\times50\times50$ grid. And the minimum value of loss functions are $L_1 = 0.08458879476$, $L_2 = 0.00013776$, $L_3 = 0.00225265$ respectively. And the discrete $L^2$ errors of Monte Carlo data, data-driven difference solution and Neural Network solution against the exact solution are computed respectively and demonstrated in the title of each sub-figures. This demonstrates the effectiveness of both the finite difference solver and the neural network solver. 

Next, we estimate the convergence rate to the periodic invariant probability measure using the coupling method. Figure \ref{fig: Example 6.1 convergence rate} is obtained by using the coupling method discussed in Section 5.2 with $N=10^6$ samples and time step size $h=0.001$. Before the two trajectories are sufficiently close to each other, we use the independent couplings. The convergence rate $r=0.094484$, and we choose $k=6$ for the fitting of the periodic function $C(t)$. As seen in Figure \ref{fig: Example 6.1 convergence rate}, the probability distribution $P[\tau^h_c > t]$ does have a $T$-periodic pre-factor, which is consistent with our discussion in Section 5.2.

\subsection{Chaotic forced Van der Pol equation}
We consider a chaotic forced Van der Pol equation with a strange attractor
\begin{equation} \label{eq:forced_VDP}
\begin{cases}
\mathrm{d}X_t = Y_t \, \mathrm{d}t + \varepsilon \, \mathrm{d}W_t^x , \\
\mathrm{d}Y_t = \left[ \mu (1 - X_t^2) Y_t - X_t + A \sin(\omega t) \right] \mathrm{d}t + \varepsilon \, \mathrm{d}W_t^y,
\end{cases}
\end{equation}
where nonlinear damping parameter $\mu = 2.0$, amplitude of external forcing $A = 1.2$, forcing frequency $\omega = 0.2 \pi$, the diffusion coefficient $\varepsilon = 0.25$, $W_t^x$ and $W_t^y$ are independent Wiener processes.

\begin{figure}[htbp]
    \centering
    
    % Row 1
    \begin{subfigure}{0.45\textwidth}
        \centering
        \includegraphics[width=\linewidth]{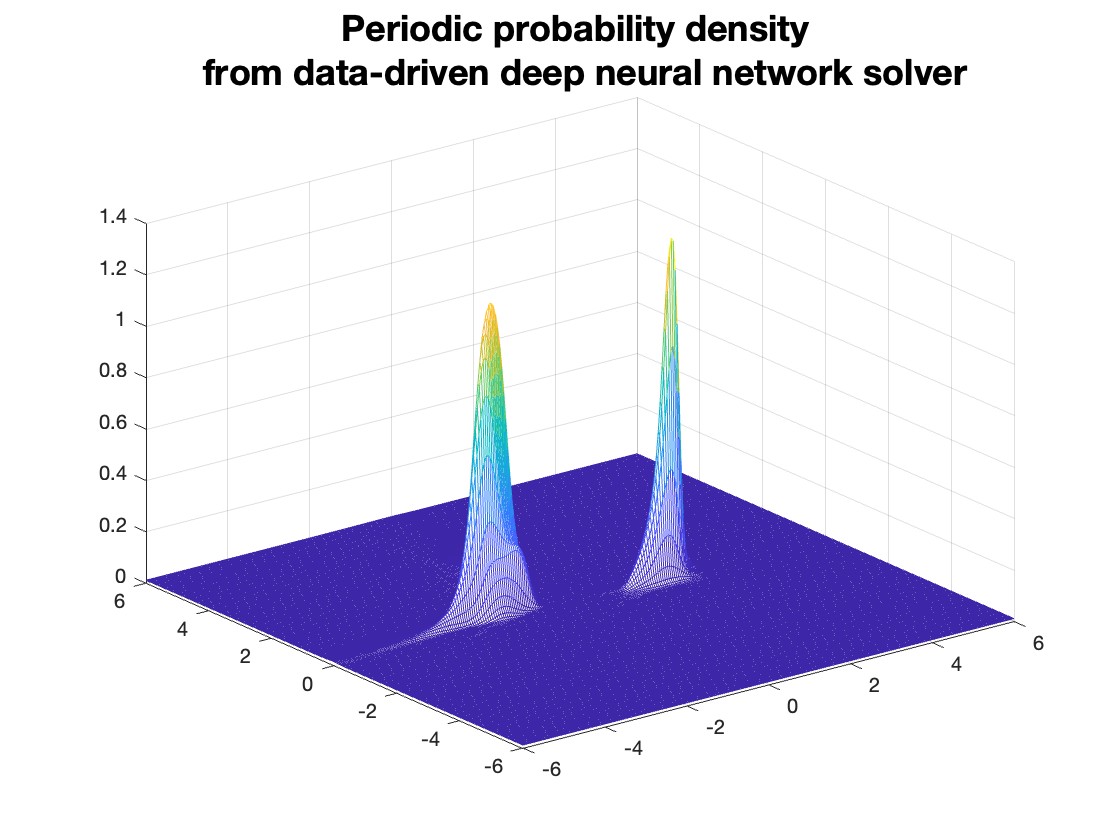}
        %\caption{First}
    \end{subfigure}
    \hfill
    \begin{subfigure}{0.45\textwidth}
        \centering
        \includegraphics[width=\linewidth]{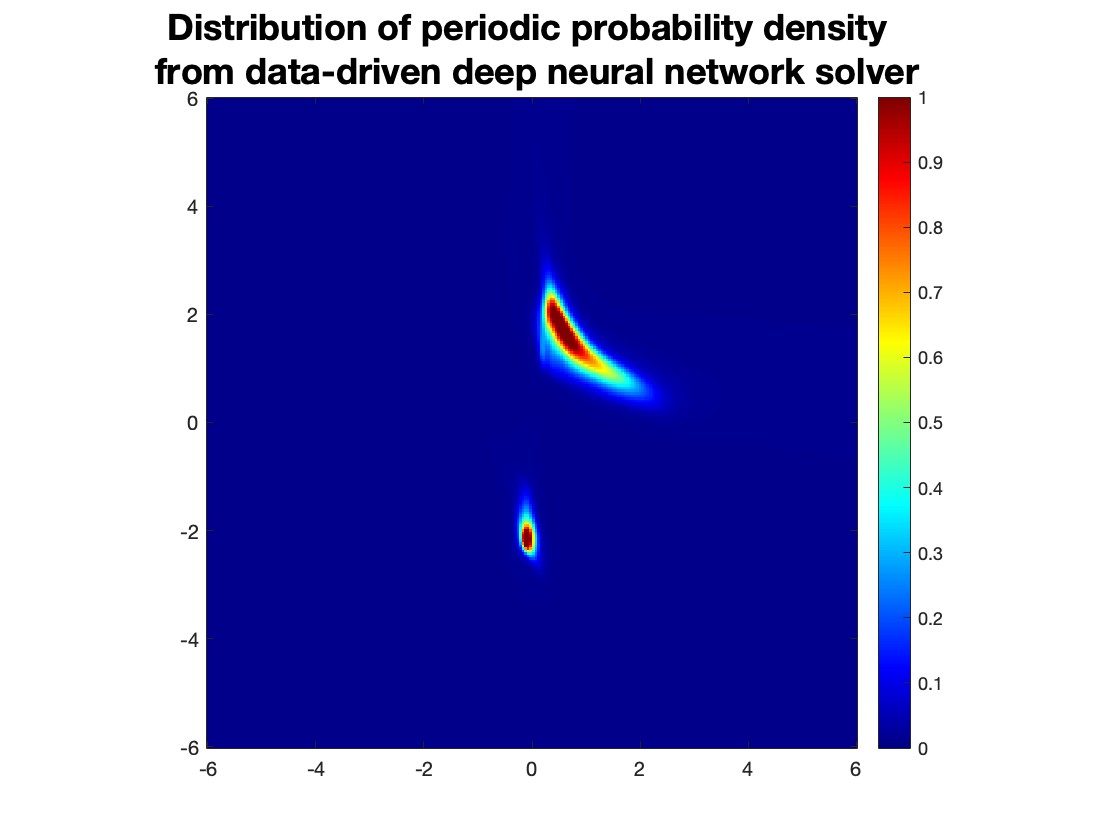}
        %\caption{Second}
    \end{subfigure}
    
    % Row 2
    \par\bigskip
    \begin{subfigure}{0.45\textwidth}
        \centering
        \includegraphics[width=\linewidth]{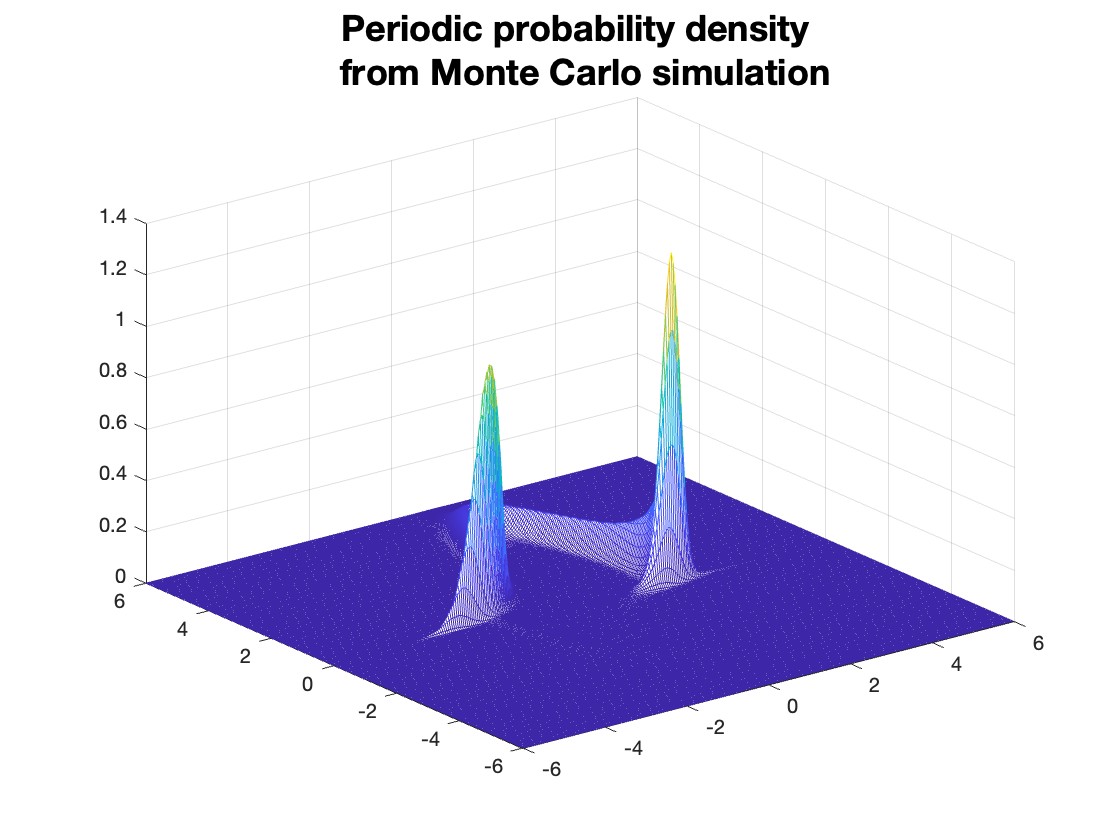}
        %\caption{Third}
    \end{subfigure}
    \hfill
    \begin{subfigure}{0.45\textwidth}
        \centering
        \includegraphics[width=\linewidth]{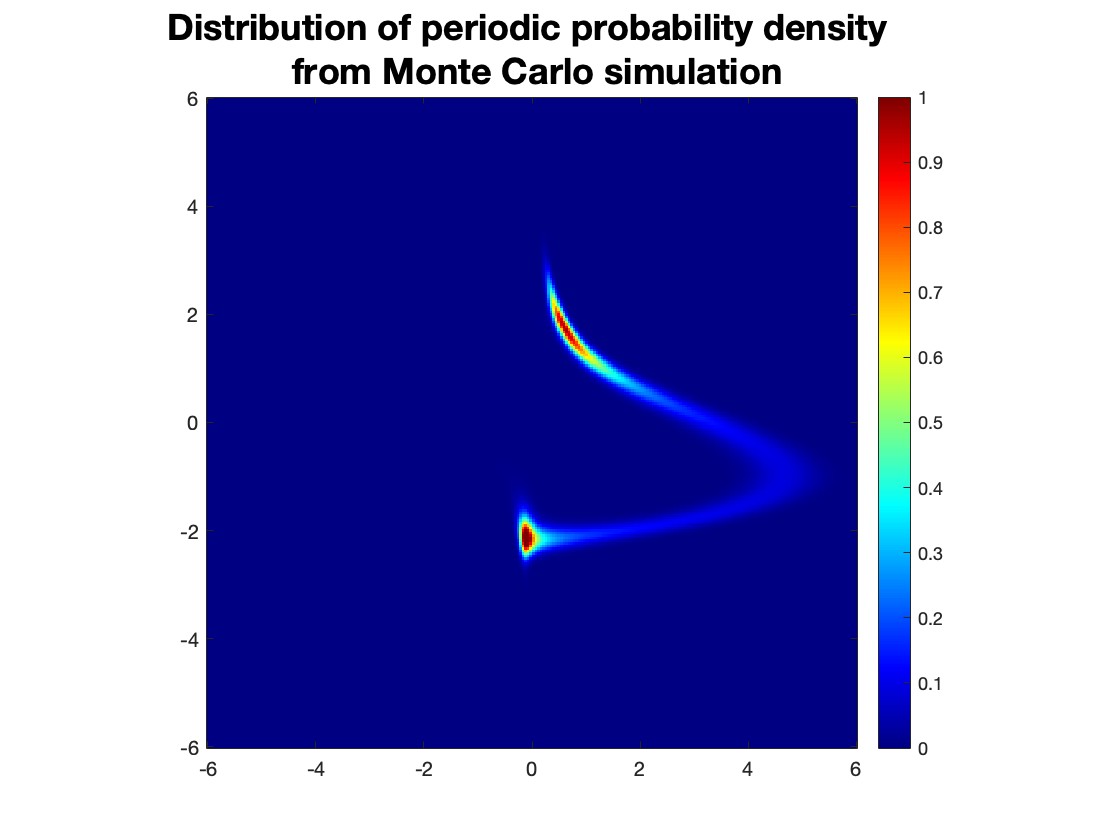}
        %\caption{Fourth}
    \end{subfigure}
    
    %\caption{1e10,ratio=0.5}
    \caption{First row: Neural Network training result at time $t = 48\cdot\frac{10}{256} = 1.875$, with grid size $= 256\times 256\times 256$. The left one is the periodic probability density function in 3D, and the right one is the distribution of probability density in 2D. Second row: Monte Carlo simulation at time $t = 48\cdot\frac{10}{256} = 1.875$, with grid size $= 256\times 256\times 256$. The left one is the periodic probability density function in 3D, and the right one is the distribution of probability density in 2D.}
    
    \label{fig: Example 6.2 neural network training result v.s. Monte Carlo}
\end{figure}

In Figure \ref{fig: Example 6.2 neural network training result v.s. Monte Carlo}, the numerical domain is $D = [-6,6] \times [-6,6] \times [0,10] \subset \mathbb{R}^3$. A direct Monte Carlo simulation that uses the Euler-Maruyama scheme with sample size $N = 10^9$, time step size $h=0.001$ and $256 \times 256 \times 256$ mesh, is used to generate subplots in the second row of Figure \ref{fig: Example 6.2 neural network training result v.s. Monte Carlo}. We use Algorithm 2 to sample $10^4$ reference points and $10^4$ training points in the domain $D$. The neural network training uses all probability densities at grid points obtained by the same Monte Carlo simulation. Then we train the artificial neural network with loss function (\ref{Eqn:loss function}), and the neural network training result is in the first row of Figure \ref{fig: Example 6.2 neural network training result v.s. Monte Carlo}. The architecture of the artificial neural network is a small feed-forward neural network with three input neurons and one output neuron, and six hidden layers, each of which contains $64$, $128$, $128$, $128$, $64$, $16$ neurons respectively. The activation function is sigmoid function. The original learning rate for loss functions $L_1$, $L_2$ and $L_3$ in (\ref{Eqn:loss function}) is $0.001$, $0.002$ and $0.002$ respectively, and the learning rate for $L_2$ and $L_3$ will be reduced by $0.9$ after every 40 epochs, while the learning rate for $L_1$ remains unchanged. There are 200 training epochs in total. The trained neural network is evaluated on a $100\times100\times100$ grid. The minimum loss functions are $L_1 = 0.06969284$, $L_2 = 0.04186701$, $L_3 = 0.00170170$ respectively. %The distribution of probability density in 2D with minimum loss function is showed in Figure (\ref{fig: Example 6.2 neural network training result}).

% \begin{figure}[htbp]
%     \centering
%     % First subfigure
%     \begin{subfigure}[b]{0.3\textwidth}
%         \centering
%         \includegraphics[width=\textwidth]{Example 6.2/eps_025_NX_256_result/Neural_training_3D_k=48.jpg}
%         %\caption{Caption 1}
%         %\label{fig:sub1}
%     \end{subfigure}
%     % Second subfigure
%     \begin{subfigure}[b]{0.3\textwidth}
%         \centering
%         \includegraphics[width=\textwidth]{Example 6.2/eps_025_NX_256_result/Neural_training_2D_k=48.jpg}
%         %\caption{Caption 2}
%         %\label{fig:sub2}
%     \end{subfigure}
%     % Third subfigure
%     \begin{subfigure}[b]{0.3\textwidth}
%         \centering
%         \includegraphics[width=\textwidth]{Example 6.2/eps_025_NX_256_result/min_loss_025_256.png}
%         %\caption{Caption 3}
%         %\label{fig:sub3}
%     \end{subfigure}
    
%     \caption{Neural Network training result at time $t = 48*\frac{10}{256} = 1.875$, with grid size $= 256\times 256\times 256$. Left one: periodic probability density function in 3D. Middle one: distribution of probability density in 2D. Right one: distribution of probability density in 2D with minimum loss function.}
%     \label{fig: Example 6.2 neural network training result}
% \end{figure}

In addition to the result of neural network training at time $t=1.875$, we also obtain the periodic probability density at time $t=1.25$, $t=2.5$, $t=6.875$ and $t=8.75$, obtained from neural network training using the same architecture. The periodic probability density and related distribution are all shown in Figure \ref{Example 6.2: neural network result at t=4,6,8,10}.

\begin{figure}[htbp]
    \centering
    
    % Row 1 (t=4)
    \begin{subfigure}{0.45\textwidth}
        \centering
        \includegraphics[width=\linewidth]{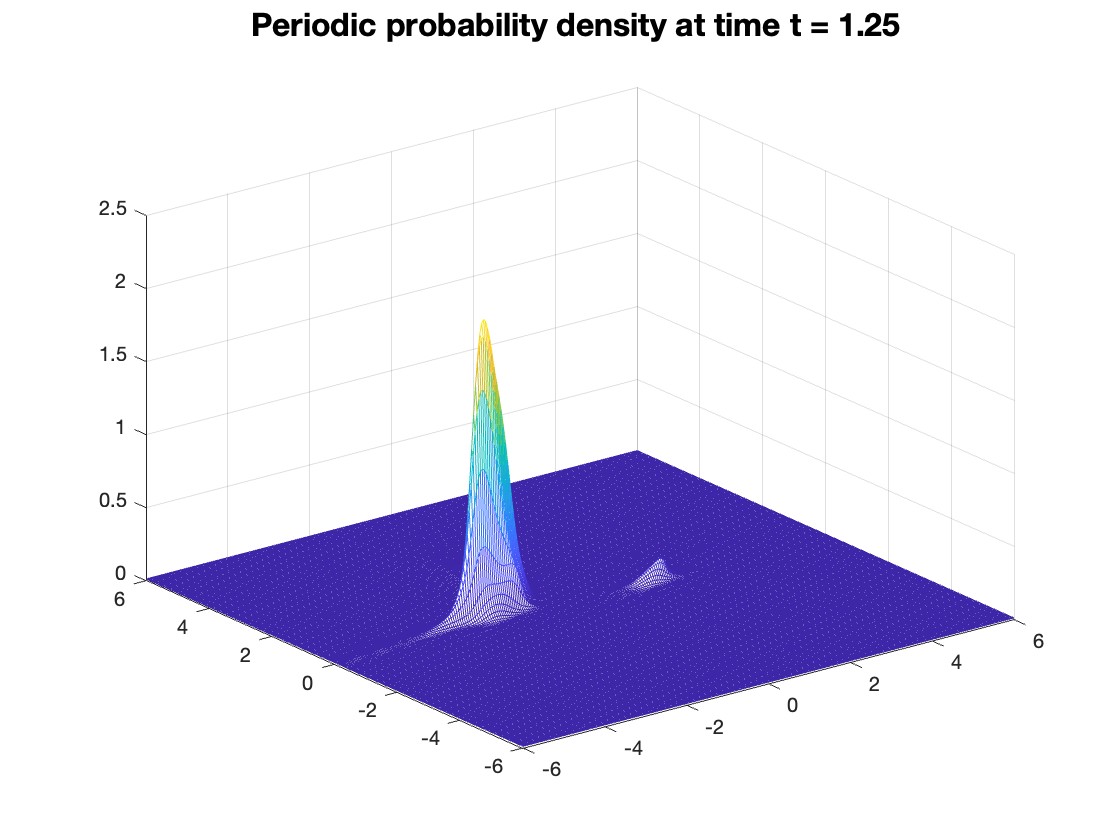}
        %\caption{A}
    \end{subfigure}
    \hfill
    \begin{subfigure}{0.45\textwidth}
        \centering
        \includegraphics[width=\linewidth]{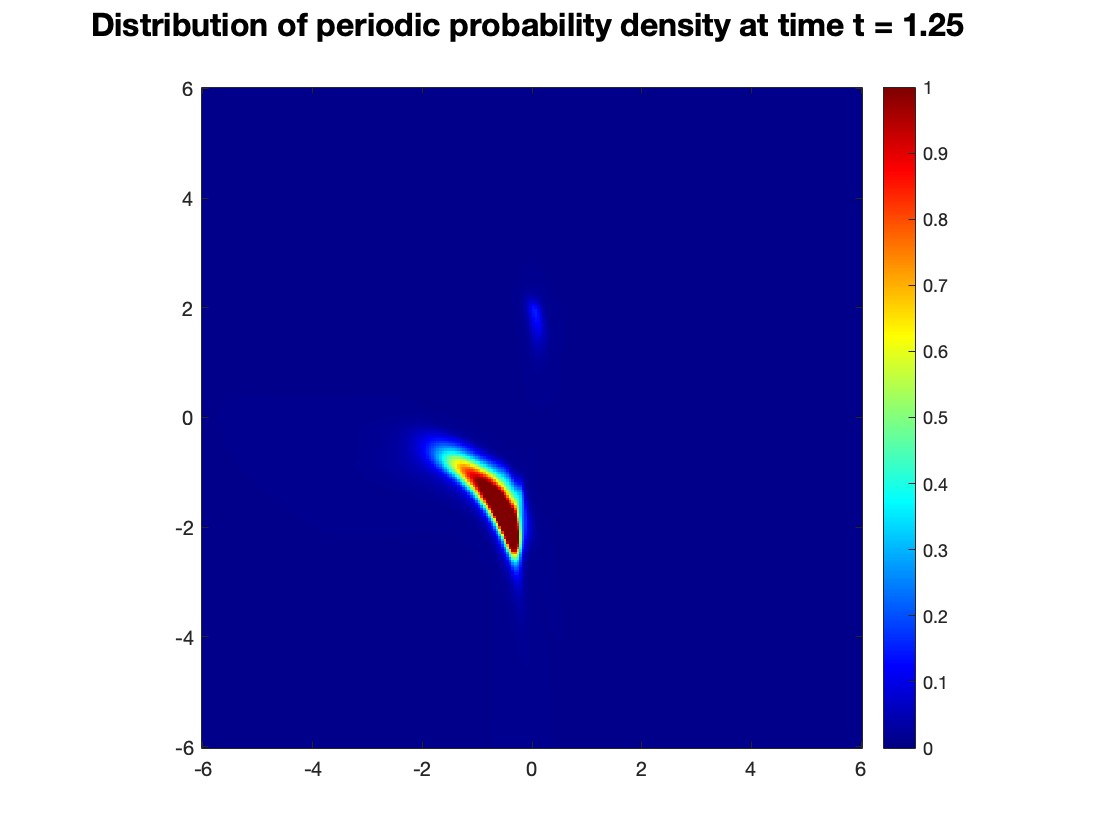}
        %\caption{B}
    \end{subfigure}

    % Row 2 (t=6)
    \par\bigskip
    \begin{subfigure}{0.45\textwidth}
        \centering
        \includegraphics[width=\linewidth]{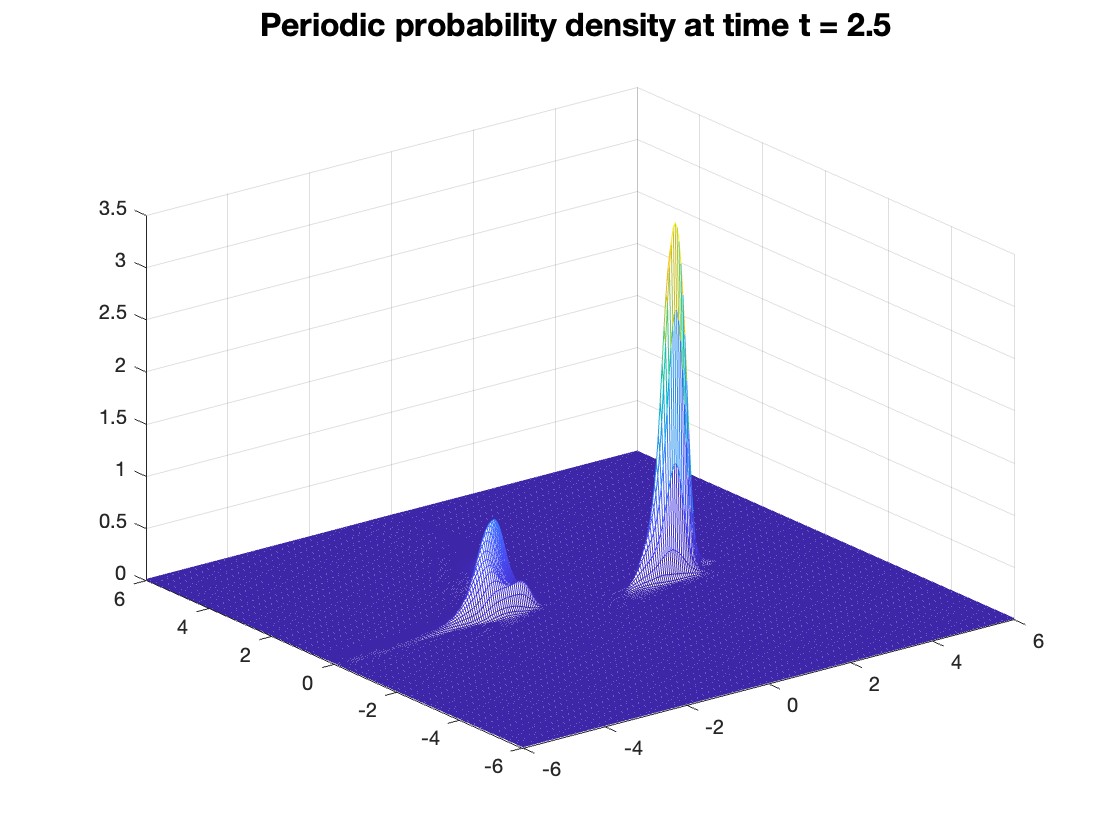}
        %\caption{C}
    \end{subfigure}
    \hfill
    \begin{subfigure}{0.45\textwidth}
        \centering
        \includegraphics[width=\linewidth]{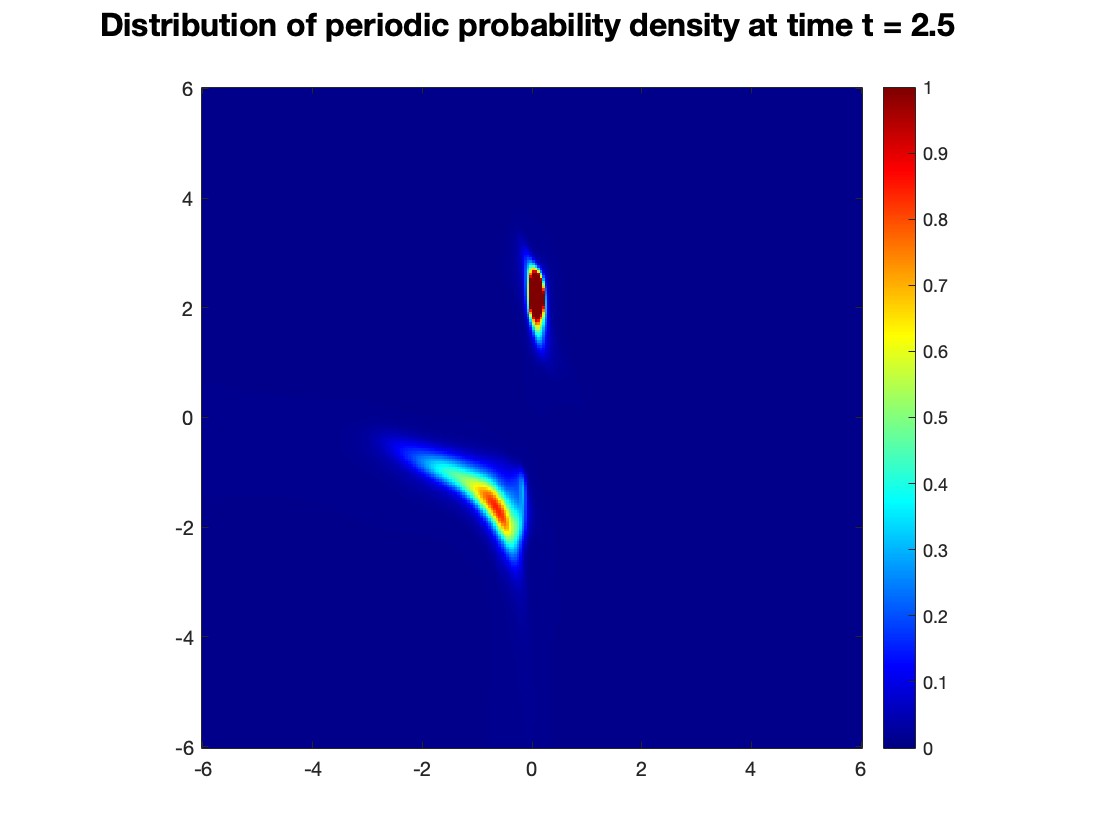}
        %\caption{D}
    \end{subfigure}
    
    % Row 3 (t=8)
    \par\bigskip
    \begin{subfigure}{0.45\textwidth}
        \centering
        \includegraphics[width=\linewidth]{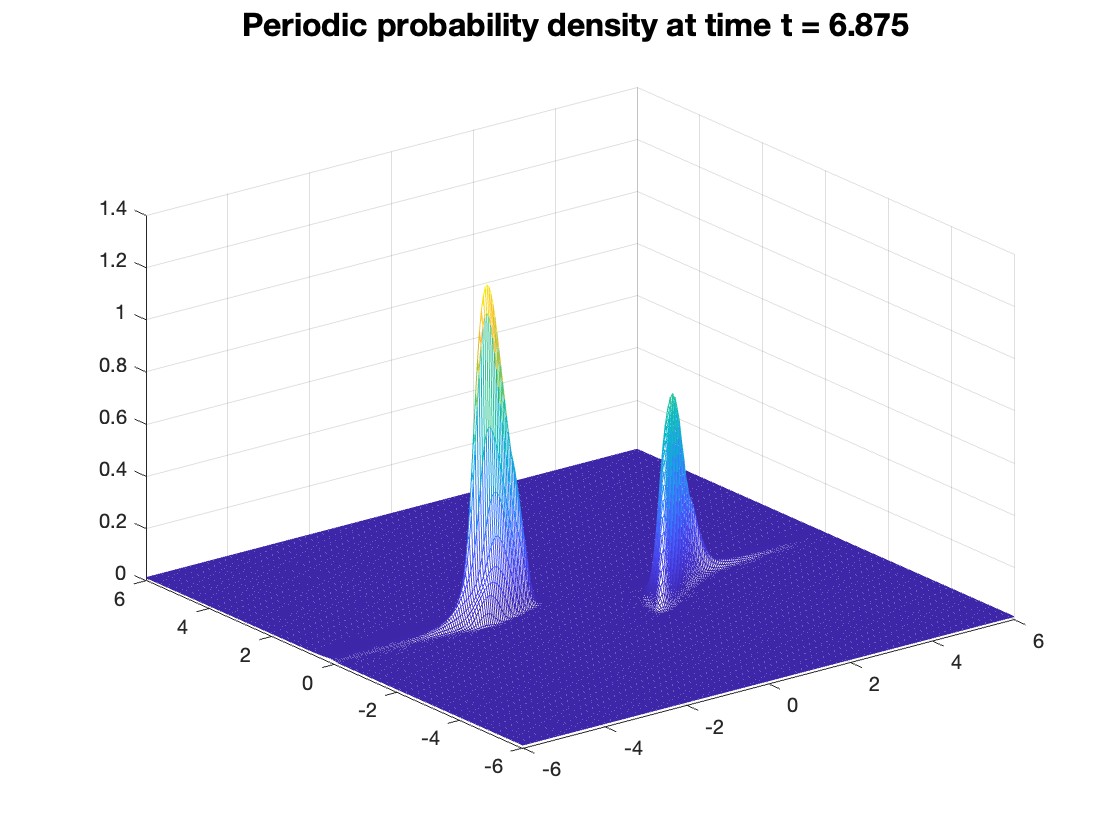}
        %\caption{E}
    \end{subfigure}
    \hfill
    \begin{subfigure}{0.45\textwidth}
        \centering
        \includegraphics[width=\linewidth]{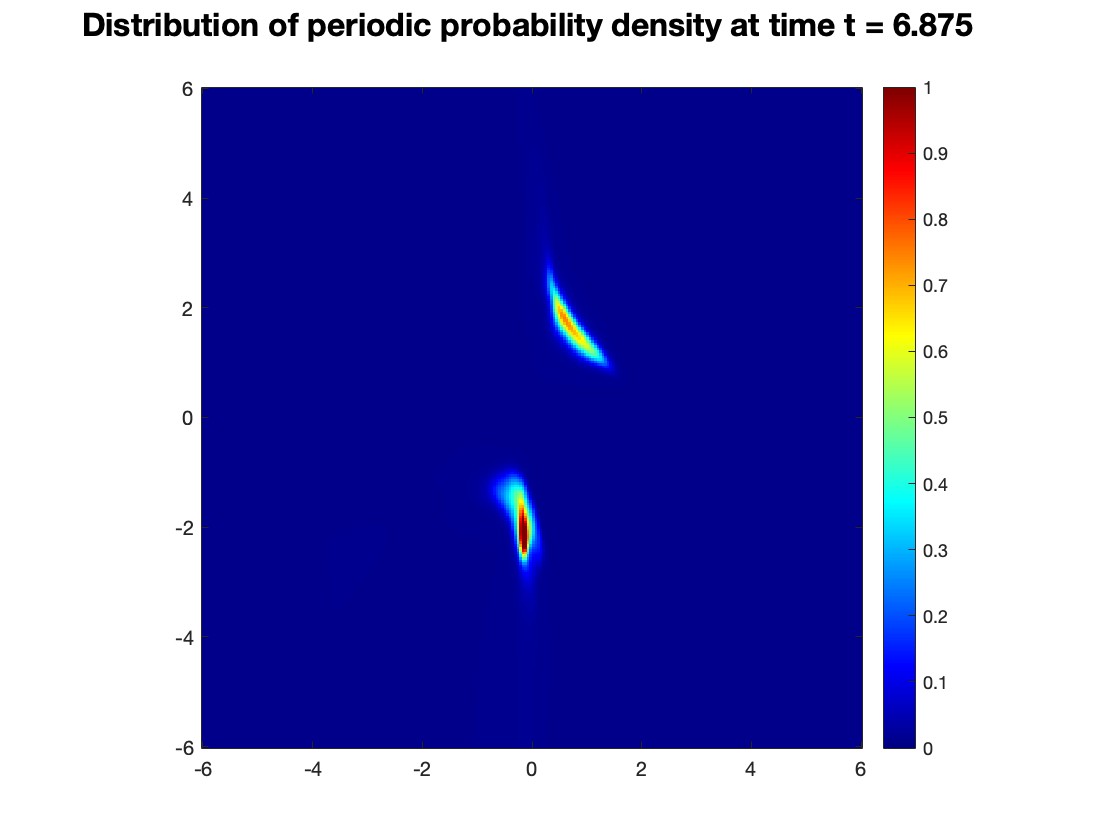}
        %\caption{F}
    \end{subfigure}

    % Row 4 (t=10)
    \par\bigskip
    \begin{subfigure}{0.45\textwidth}
        \centering
        \includegraphics[width=\linewidth]{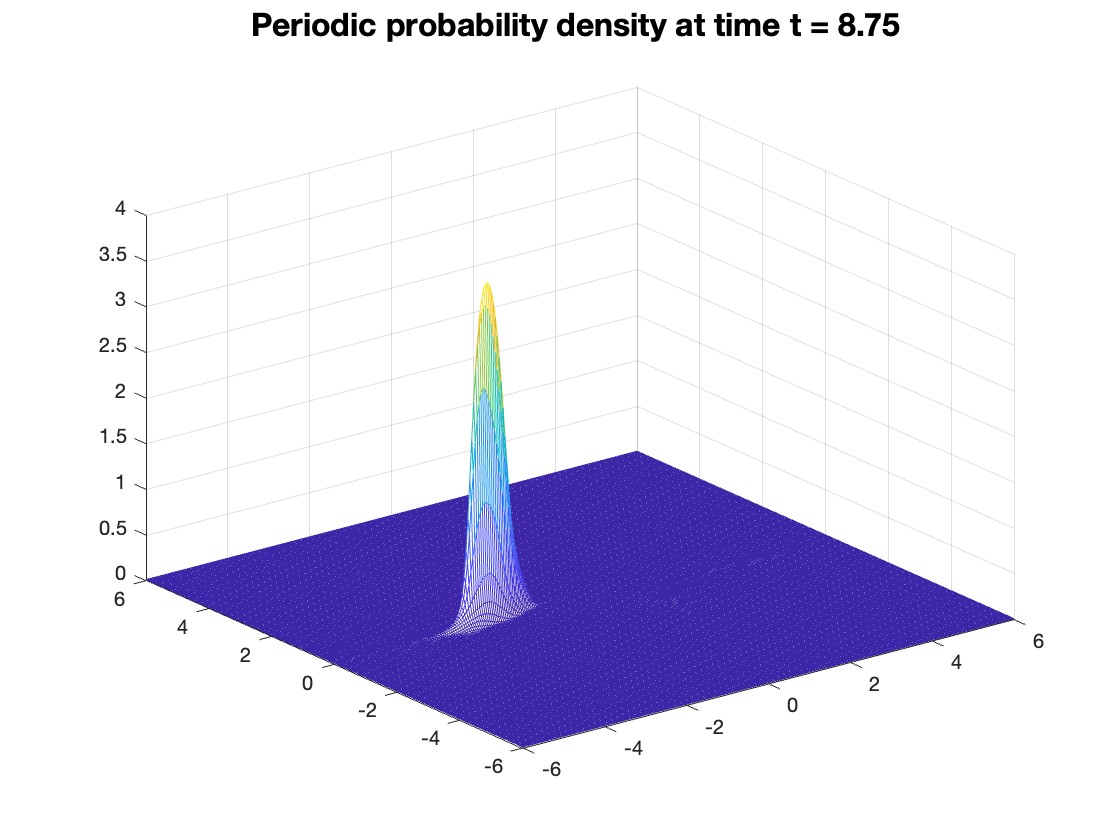}
        %\caption{G}
    \end{subfigure}
    \hfill
    \begin{subfigure}{0.45\textwidth}
        \centering
        \includegraphics[width=\linewidth]{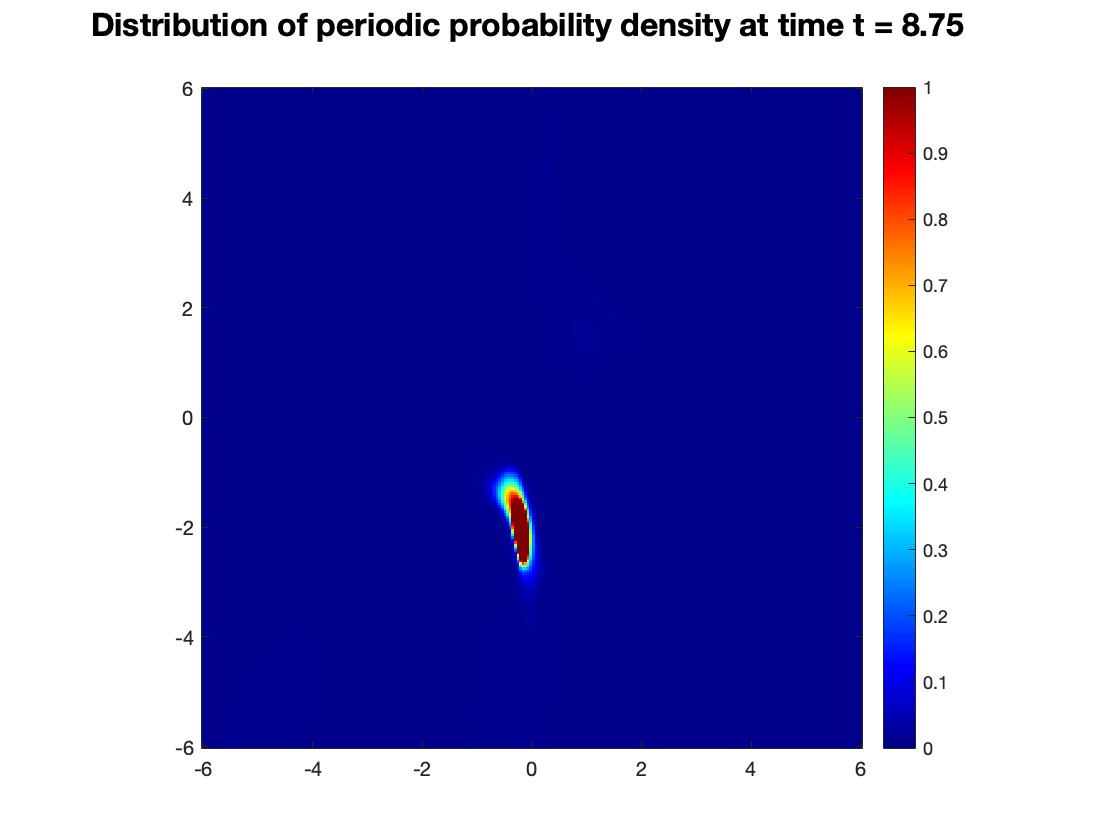}
        %\caption{H}
    \end{subfigure}

    \caption{Neural Network training result in 3D and distribution in 2D of Example 6.2 at time slice $t = 1.25,\,2.5,\,6.875,\,8.75$ respectively with grid size $= 256\times 256\times 256$.}
   \label{Example 6.2: neural network result at t=4,6,8,10}
\end{figure}

%% Insert convergence rate result picture, convergence rate and curve fitting
\begin{figure}[htbp]
    \centering
    \includegraphics[width=1.0\textwidth]{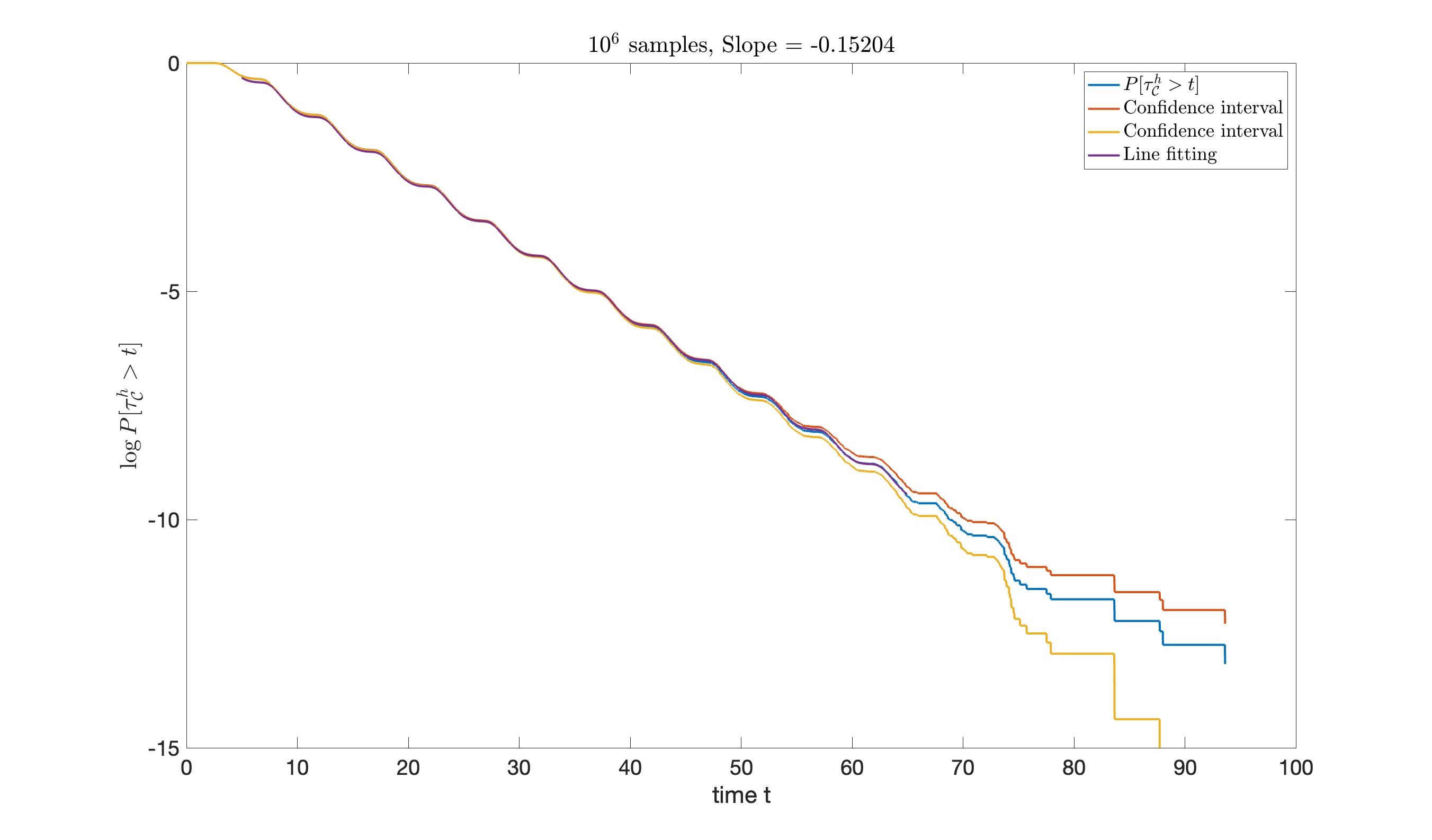}
    \caption{Approximation of $\log P(\tau_c > t)$ and the related confidence interval.}
    \label{fig: Example 6.2 convergence rate}
\end{figure}

In Figure \ref{fig: Example 6.2 convergence rate}, the convergence rate to the periodic invariant probability measure is obtained by using the coupling method in Section 5.2 with $N=10^6$ samples and time step size $h=0.001$. The initial position for two points are $(-2.0, \,0.0)$ and $(2.0, \,0.0)$. We use the reflection couplings before using maximal couplings. The convergence rate $r = 0.15204$. We choose $k=6$ in the algorithm of estimating $C(t)$. Also, Figure \ref{fig: Example 6.2 convergence rate} shows the $T$-periodic pre-factor of the probability distribution $P[\tau^h_c > t]$, which is again consistent with the analysis in Section 5.2.

\subsection{Periodically driven neuronal oscillators}
We consider the following periodically driven neuronal oscillators, which is one of the generalization of the periodic Van der Pol oscillators in higher dimensional space.
\begin{equation} \label{eq:coupled_VDP}
\begin{cases}
\mathrm{d}X_t = \left[ \mu(t) \left( X_t - \frac{1}{3} X_t^3 - Y_t \right) + \gamma (Z_t - X_t) \right] \mathrm{d}t + \varepsilon \, \mathrm{d}W_t^x, \\
\mathrm{d}Y_t = \frac{1}{\mu(t)} X_t \, \mathrm{d}t + \varepsilon \, \mathrm{d}W_t^y, \\
\mathrm{d}Z_t = \left[ \mu(t) \left( Z_t - \frac{1}{3} Z_t^3 - U_t \right) + \gamma (X_t - Z_t) \right] \mathrm{d}t + \varepsilon \, \mathrm{d}W_t^z, \\
\mathrm{d}U_t = \frac{1}{\mu(t)} Z_t \, \mathrm{d}t + \varepsilon \, \mathrm{d}W_t^u,
\end{cases}
\end{equation}
where $\mu(t) = 1.2 + 0.2 \sin(0.5 t)$ is time-periodic with period $T = 4\pi$, $\gamma$ is coupling strength, diffusion coefficient $\varepsilon = 0.5$, and $W_t^x$, $W_t^y$, $W_t^z$ and $W_t^w$ are four independent Wiener processes. 

When $\gamma = 0$, the equation (\ref{eq:coupled_VDP}) is a decoupled Van der Pol oscillator. From a computational point of view, the decoupled Van der Pol oscillator is indeed more challenging as it admits a 4D invariant manifold in a 5D phase space. If $\gamma \neq 0$, the equation becomes a coupled Van der Pol oscillator. In this example we use $\gamma = 0.03$ in coupled Van der Pol oscillators. 

We first compute the decoupled Van der Pol system with $\gamma = 0$. The numerical domain is $[-3,3] \times [-3,3] \times [-3,3] \times [-3,3] \times [-2 \pi,2\pi]\subset \mathbb{R}^5$. A direct Monte Carlo simulation that uses $10^{11}$ steps of the Euler-Maruyama scheme is used to estimate the probability density at $40000$ collocation points. We use the mesh-free algorithm developed in \cite{dobson2019efficient} to reduce the memory requirement. Then we generate $40000$ reference points to minimize the Fokker-Planck operator residual. The architecture of the artificial neural network is a feed-forward neural network with five input neurons and one output neuron, and six hidden layers, each of which contains 64, 256, 256, 256, 64, 16
neurons respectively. The trained neural network is evaluated at four $(x,y,t)$ 3D-cubes for $(z,u) = (1.95, -0.39)$, $(1.77, 0.15)$, $(-1.65, -0.57)$ and $(-1.83, -0.15)$ respectively. For each 3D-cube, we present two $(x,y)$-slices at $t = -\pi$ and $t = \pi$. The heat maps of 8 $(x,y)$-slices of the decoupled Van der Pol oscillator are presented in Figure \ref{VDP4d}. From Figure \ref{VDP4d}, we can see that the probability density functions at all $8$ $(x,y)$-slices are similar to each other, which is consistent with the fact that the periodic invariant probability measure at each time is a product measure of two 2D probability measures.

%We will discuss the Neural Network estimation result and convergence rate of decoupled and coupled periodic Van der Pol oscillators separately.

\begin{figure}
    \centering
    \includegraphics[width=\linewidth]{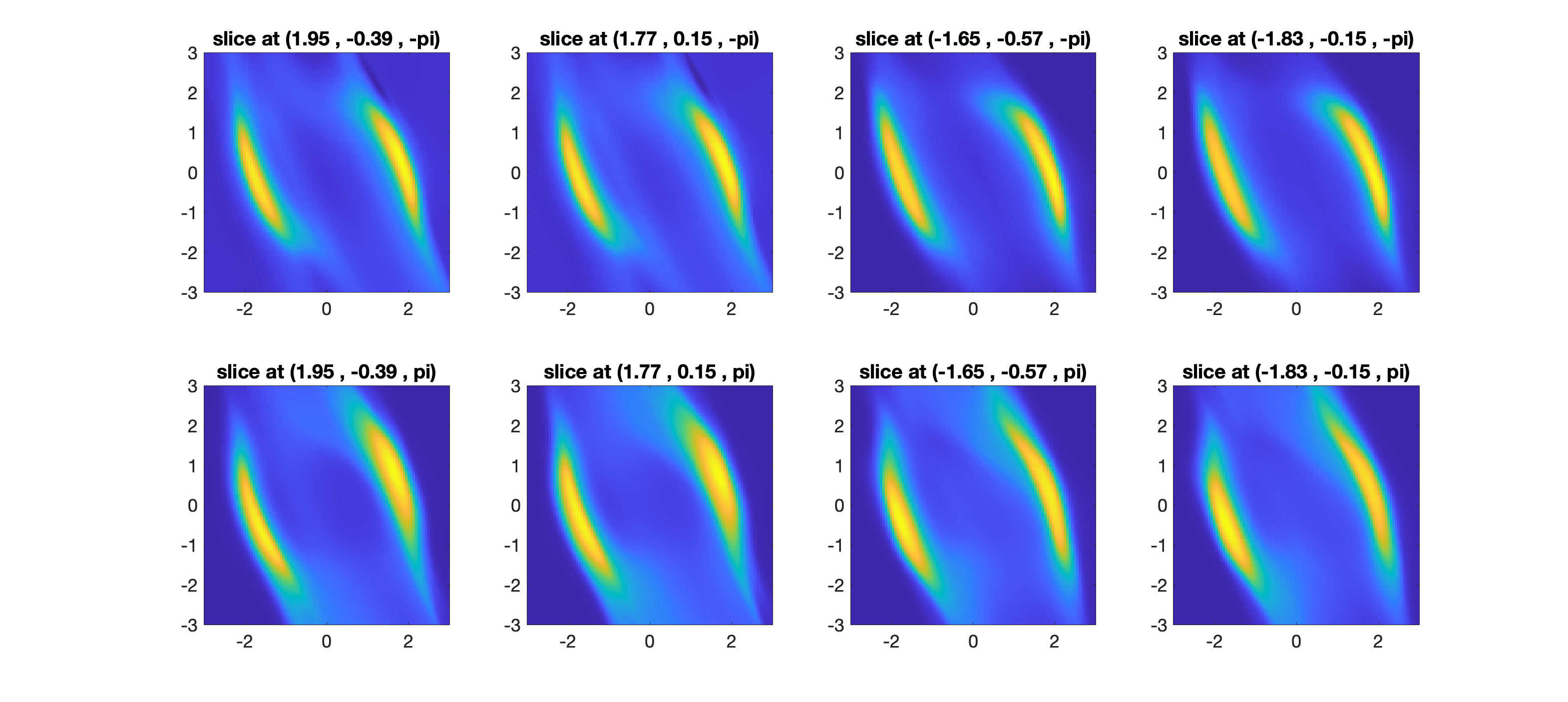}
    \caption{Heat maps of neural network solution of the decoupled Van der Pol oscillator at 8 $(x,y)$-slices. The first line: $(x,y)$-slices at $t = -\pi$ for $(z,u) = (1.95, -0.39)$, $(1.77, 0.15)$, $(-1.65, -0.57)$ and $(-1.83, -0.15)$. The second line: $(x,y)$-slices at $t = \pi$ for $(z,u) = (1.95, -0.39)$, $(1.77, 0.15)$, $(-1.65, -0.57)$ and $(-1.83, -0.15)$.}
    \label{VDP4d}
\end{figure}

Then, we compute the periodic invariant probability measure of the coupled Van der Pol equation at the same 5D domain with coupling strength $\gamma = 0.03$. 

We choose the same four $(x,y,t)$ 3D-cubes as in decoupled Van der Pol oscillators, for $(z,u) = (1.95, -0.39)$, $(1.77, 0.15)$, $(-1.65, -0.57)$ and $(-1.83, -0.15)$ respectively. For each 3D-cube, we present one $(x,y)$-slice at $t = 0$. The heat maps of 4 $(x,y)$-slices of the coupled Van der Pol oscillators with $\gamma = 0.03$ are shown in Figure \ref{VDP4c}. Since there is no ground truth, we use large scale Monte Carlo simulation to estimate the probability density function at one 2D-slice, $(x,y,1.95,-0.39,0)$, as a reference. We can see that the probability density function at 2D slices becomes asymmetric due to coupling. As seen in Figure \ref{VDP4c}, the neural network solution is consistent with the Monte Carlo simulation result. In addition, we demonstrate the periodic invariant probability measure at the 2D slice $(x, 0.15, z, 0.20, 0)$ in the bottom middle panel of Figure \ref{VDP4c}. We can see that although the coupling strength is weak, $x$-variable and $z$-variable still become highly correlated. 

\begin{figure}
    \centering
    \includegraphics[width=\linewidth]{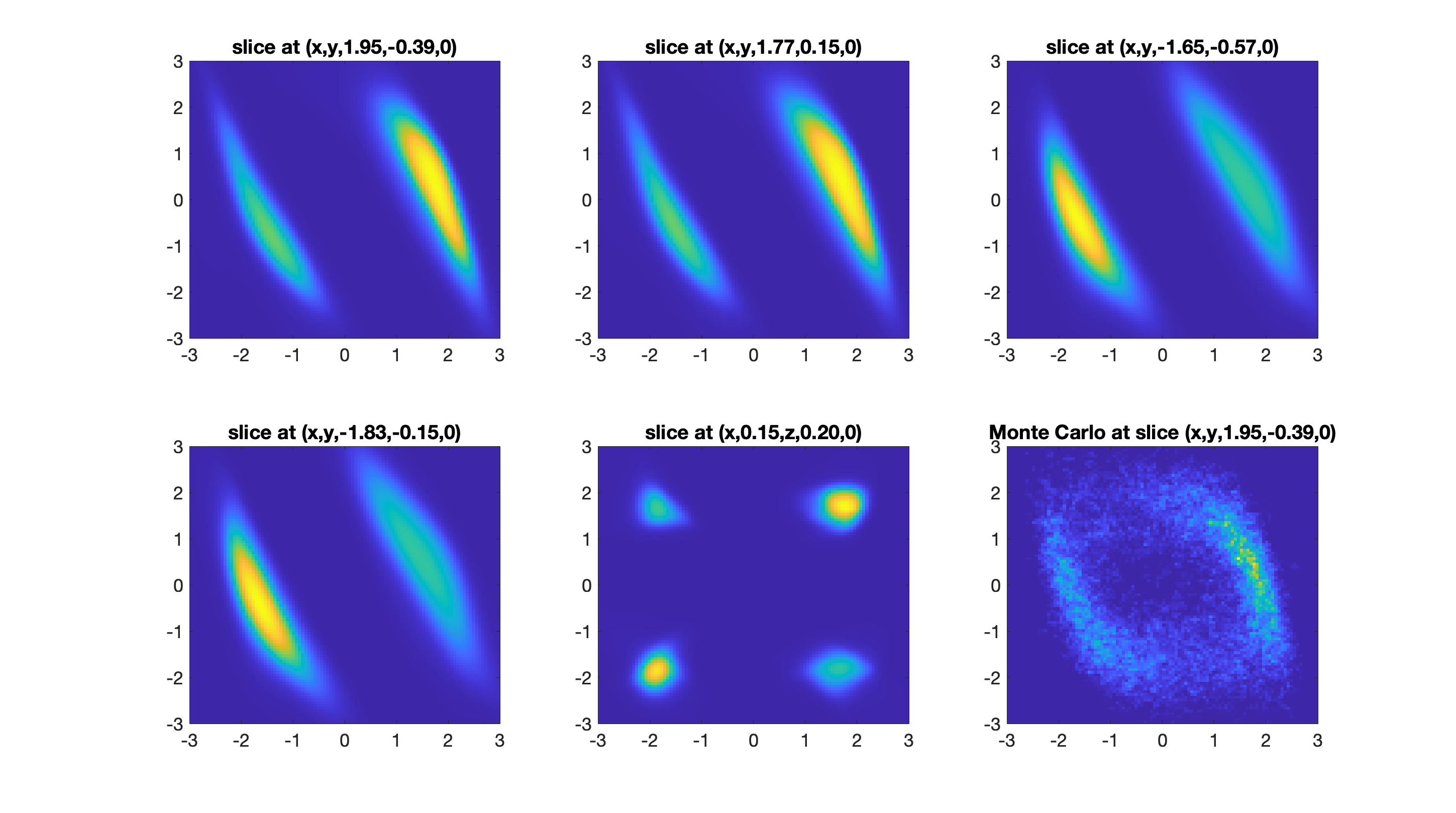}
    \caption{Heat maps for neural network training result and Monte Carlo simulation. The top panel and the left bottom panel: heat maps for neural network solution at 4 $(x,y)$-slices at time $t=0$ with $(z,u) = (1.95, -0.39)$, $(1.77, 0.15)$, $(-1.65, -0.57)$ and $(-1.83, -0.15)$. The middle bottom panel: heat maps of periodic invariant probability measure at 2D slice $(x, 0.15, z, 0.20, 0)$. The right bottom panel: Monte Carlo simulation of probability density function at 2D-slice $(x,y,1.95,-0.39,0)$.}
    \label{VDP4c}
\end{figure}

\begin{figure}[htbp]
    \centering
    \includegraphics[width=1.0\textwidth]{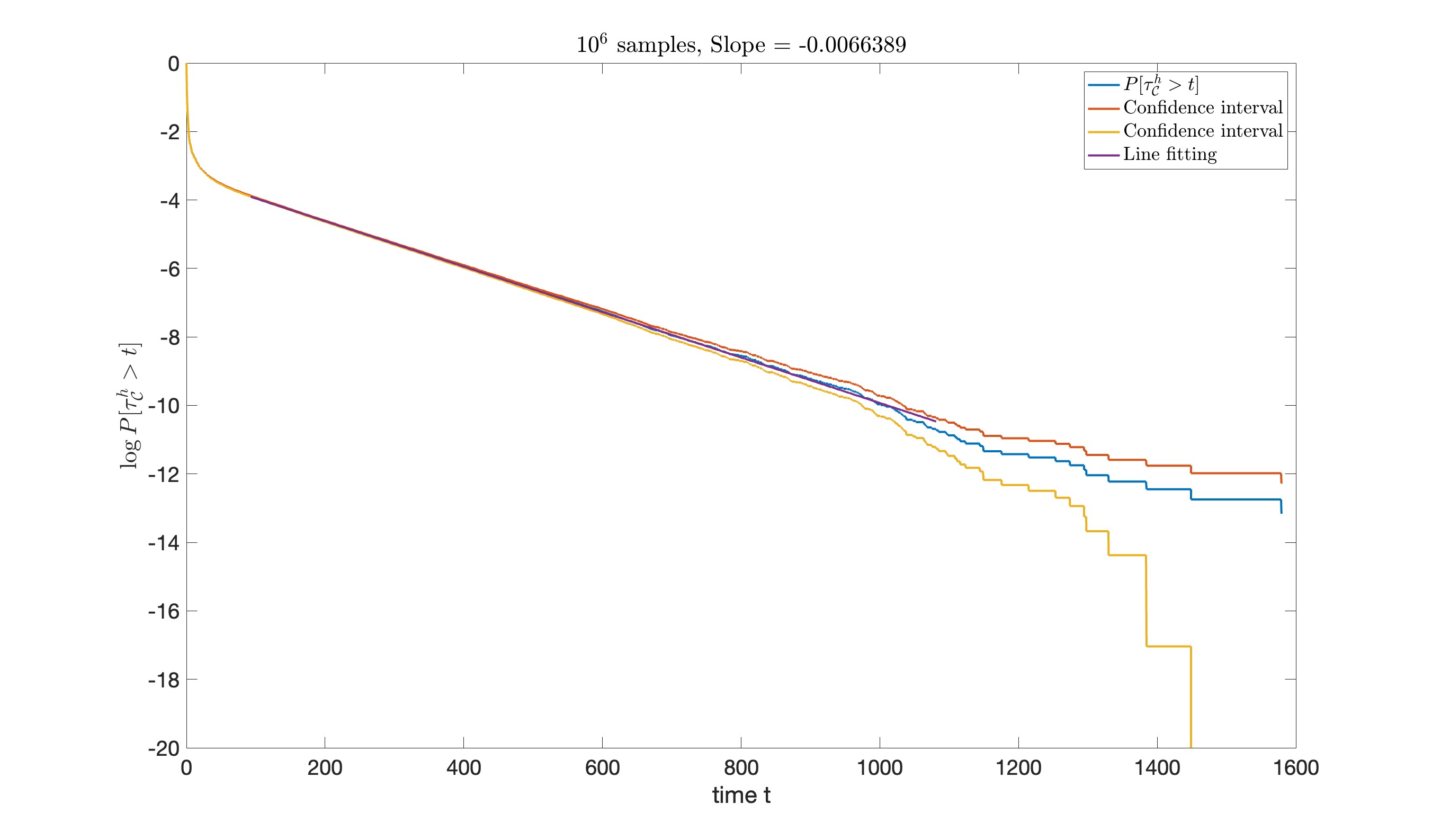}
    \caption{Approximation of $\log P(\tau_c > t)$ and the related confidence interval.}
    \label{fig: Example 6.3 convergence rate}
\end{figure}

Then, we use the coupling method to estimate the speed of convergence to the periodic invariant probability measure of coupled Van der Pol oscillators with coupling rate $\gamma=0.03$. In Figure \ref{fig: Example 6.3 convergence rate}, the convergence rate to the periodic invariant probability measure is obtained by using the coupling method in Section 5.2 with $N=10^6$ samples and time step size $h=0.001$.The initial position for two points are $(2.0, 1.0, 1.0, -2.0)$ and $(2.5, 1.0, 1.0, 2.0)$. We use the reflection couplings first, and when two trajectories are sufficiently close to each other, we use the maximal couplings. We choose $k=8$ in the algorithm of estimating $C(t)$, and the convergence rate $r = 0.0066389$, which is a very slow convergence rate. It is reasonable because two trajectories are more difficult to be coupled successfully in 4-dimensional space.

\section{Conclusion and future works}
We proposed a data-driven neural network method for solving the time-periodic Fokker-Planck equations that describe the time-periodic invariant probability measures. The time-periodic case is not a trivial extension of the case of the stationary (time-independent) Fokker-Planck equation studied in \cite{dobson2019efficient} and \cite{zhai2022deep}. In fact, we need to consider the specific periodic boundary condition when constructing the discretization of operator $\mathcal{L}$ in the optimization problem in Chapter 3, as well as the new periodic penalty term in the loss function of the neural network solver in Chapter 4. In addition, in the analysis of the convergence speed of the periodic process $X_t$, we introduced a Floquet-type decomposition of the distribution of coupling time, which is different from the estimation approach used for time-homogeneous stochastic differential equations. These features make the time-periodic case significantly different from the time-independent case, and introduce new complexity in both analysis and numerical computation.

In this paper, we studied three numerical examples in high dimensions. Despite the success of estimation of the periodic probability density in 3D, there are some challenges in higher-dimensional systems, such as in dimension five or higher. First, effective sampling of the state space to generate collocation points for neural network training remains difficult. Even though the mesh-free sampler helps avoid the issues associated with grid-based discretization of operator $\mathcal{L}$ in Monte Carlo simulation, sampling in regions with low probability density is still challenging. This is because the Fokker-Planck equation is about the {\it probability density}, which is a very local property. Sampling probability densities in high dimensions is intrinsically challenging. Second, importance sampling becomes critical because the solution can be highly concentrated in the domain. For the same reason, the relative volume of those regions is exponentially smaller as the dimension increases. As a result, in the region where the solution changes the most rapidly, there are only a small fraction of points carrying significant information. In this situation, the data-driven neural network solver may produce a trivial solution, since $u = 0$ also satisfies the Fokker-Planck equation. 
 
Based on the common challenges for real high dimension cases, it is important to improve sampling methods to obtain more accurate numerical estimation of the invariant periodic probability solution. In future work, on one hand, we need to further advance the collocation sampling method proposed in \cite{wang2024deep} to "highlight" areas where the solution has the most rapid changes when constructing the training set. On the other hand, we must develop a true high-dimensional sampler that does not rely on spatial discretizations at all. This will require an adaptation of the kernel density estimator (KDE) for sampling the invariant probability density of stochastic differential equations, as well as the use of the "forward-reverse sampler" proposed in \cite{milstein2004transition} to correct the error introduced by KDE. We will address these two points in our subsequent work. 

\bibliographystyle{plain}
\bibliography{reference} 

\begin{thebibliography}{10}

\bibitem{bogachev2022fokker}
Vladimir~I Bogachev, Nicolai~V Krylov, Michael R{\"o}ckner, and Stanislav~V
  Shaposhnikov.
\newblock {\em Fokker--Planck--Kolmogorov Equations}, volume 207.
\newblock American Mathematical Society, 2022.

\bibitem{dobson2019efficient}
Matthew Dobson, Yao Li, and Jiayu Zhai.
\newblock An efficient data-driven solver for fokker-planck equations:
  algorithm and analysis.
\newblock {\em arXiv preprint arXiv:1906.02600}, 2019.

\bibitem{ermentrout2010mathematical}
Bard Ermentrout and David~Hillel Terman.
\newblock {\em Mathematical foundations of neuroscience}, volume~35.
\newblock Springer, 2010.

\bibitem{hairer2010convergence}
Martin Hairer.
\newblock Convergence of markov processes.
\newblock {\em Lecture notes}, 18(26):11, 2010.

\bibitem{huang2015integral}
Wen Huang, Min Ji, Zhenxin Liu, and Yingfei Yi.
\newblock Integral identity and measure estimates for stationary fokker-planck
  equations.
\newblock {\em The Annals of Probability}, 43(4):1712--1730, 2015.

\bibitem{ji2019existence}
Min Ji, Weiwei Qi, Zhongwei Shen, and Yingfei Yi.
\newblock Existence of periodic probability solutions to fokker-planck
  equations with applications.
\newblock {\em Journal of Functional Analysis}, 277(11):108281, 2019.

\bibitem{ji2021convergence}
Min Ji, Weiwei Qi, Zhongwei Shen, and Yingfei Yi.
\newblock Convergence to periodic probability solutions in fokker--planck
  equations.
\newblock {\em SIAM Journal on Mathematical Analysis}, 53(2):1958--1992, 2021.

\bibitem{Kingma2014AdamAM}
Diederik~P. Kingma and Jimmy Ba.
\newblock Adam: A method for stochastic optimization.
\newblock {\em CoRR}, abs/1412.6980, 2014.

\bibitem{kreider2025artificial}
Max Kreider, Peter~J Thomas, and Yao Li.
\newblock Artificial neural network solver for fokker-planck and koopman
  eigenfunctions.
\newblock {\em arXiv preprint arXiv:2508.20339}, 2025.

\bibitem{kromer2020long}
Justus~A Kromer and Peter~A Tass.
\newblock Long-lasting desynchronization by decoupling stimulation.
\newblock {\em Physical Review Research}, 2(3):033101, 2020.

\bibitem{li2020numerical}
Yao Li and Shirou Wang.
\newblock Numerical computations of geometric ergodicity for stochastic
  dynamics.
\newblock {\em Nonlinearity}, 33(12):6935, 2020.

\bibitem{mattingly2002ergodicity}
Jonathan~C Mattingly, Andrew~M Stuart, and Desmond~J Higham.
\newblock Ergodicity for sdes and approximations: locally lipschitz vector
  fields and degenerate noise.
\newblock {\em Stochastic processes and their applications}, 101(2):185--232,
  2002.

\bibitem{melland2023attractor}
Pake Melland and Rodica Curtu.
\newblock Attractor-like dynamics extracted from human electrocorticographic
  recordings underlie computational principles of auditory bistable perception.
\newblock {\em Journal of Neuroscience}, 43(18):3294--3311, 2023.

\bibitem{milstein2004transition}
Grigori~N Milstein, John~GM Schoenmakers, and Vladimir Spokoiny.
\newblock Transition density estimation for stochastic differential equations
  via forward-reverse representations.
\newblock {\em Bernoulli}, 10(2):281--312, 2004.

\bibitem{perez2023universal}
Alberto P{\'e}rez-Cervera, Boris Gutkin, Peter~J Thomas, and Benjamin Lindner.
\newblock A universal description of stochastic oscillators.
\newblock {\em Proceedings of the National Academy of Sciences},
  120(29):e2303222120, 2023.

\bibitem{sceniak2001visual}
Michael~P Sceniak, Michael~J Hawken, and Robert Shapley.
\newblock Visual spatial characterization of macaque v1 neurons.
\newblock {\em Journal of neurophysiology}, 85(5):1873--1887, 2001.

\bibitem{wang2022modulation}
Jixuan Wang, Bin Deng, Tianshi Gao, Jiang Wang, and Chen Liu.
\newblock Modulation of cortical oscillations by periodic electrical
  stimulation is frequency-dependent.
\newblock {\em Communications in Nonlinear Science and Numerical Simulation},
  110:106356, 2022.

\bibitem{wang2024deep}
Xili Wang, Kejun Tang, Jiayu Zhai, Xiaoliang Wan, and Chao Yang.
\newblock Deep adaptive sampling for surrogate modeling without labeled data.
\newblock {\em Journal of Scientific Computing}, 101(3):77, 2024.

\bibitem{zhai2022deep}
Jiayu Zhai, Matthew Dobson, and Yao Li.
\newblock A deep learning method for solving fokker-planck equations.
\newblock In {\em Mathematical and scientific machine learning}, pages
  568--597. PMLR, 2022.

\end{thebibliography}

%\printbibliography
\end{document}